%% file: main.tex
\documentclass[a4paper, 10pt]{article}

\input{preamble.tex}

\usepackage[style=alphabetic, maxnames = 5, maxalphanames=5]{biblatex}
\title{Equivariant $K$-theory, affine Grassmannian and perfection}
\author{Jakub Löwit}
\date{}

\begin{document}

\maketitle
\begin{abstract}
We study torus-equivariant algebraic $K$-theory of affine Schubert varieties in the perfect affine Grassmannians over $\F_p$. We further compare it to the torus-equivariant Hochschild homology of perfect complexes, which has a geometric description in terms of global functions on certain fixed-point schemes. We prove that $\F_p$-linearly, this comparison is an isomorphism. Our approach is quite constructive, resulting in new computations of these $K$-theory rings. We establish various structural results for equivariant perfect algebraic $K$-theory on the way; we believe these are of independent interest.
\end{abstract}

\tableofcontents

\input{0_introduction}
\input{1_fixed_point_schemes_and_trace_maps}

\input{2_perfect_schemes}
\input{3_perfect_K_theory}
\input{4_perfect_trace_maps}
\input{5_trace_map_for_affine_grassmannian}
\input{6_some_examples}
\input{7_more_examples}
\printbibliography

\bigskip
\noindent Jakub Löwit, \newline
Institute of Science and Technology Austria (ISTA), \newline
Am Campus 1, \newline 
3400 Klosterneuburg, \newline
Austria \newline
\texttt{jakub.loewit@ist.ac.at}

\end{document}

%% file: preamble.tex
\usepackage[a4paper, total={6in, 9.5in}]{geometry}
\usepackage[utf8]{inputenc}
\usepackage{amssymb}
\usepackage{amsthm}
\usepackage{amsmath}
\usepackage{amsfonts}

\usepackage{graphicx}

\usepackage{comment}
\usepackage{titling}
\usepackage{tikz-cd}

\usepackage{graphicx}
\usepackage{hyperref}
\usepackage{mathrsfs}
\usepackage{mathtools}

\usepackage[mathscr]{euscript}

\usepackage{stmaryrd}

\newcounter{cislo} \numberwithin{cislo}{section}

\numberwithin{equation}{subsection}

\newtheorem{theorem}[cislo]{Theorem}
\newtheorem{lemma}[cislo]{Lemma}
\newtheorem{proposition}[cislo]{Proposition}
\newtheorem{corollary}[cislo]{Corollary}
\newtheorem{claim}[cislo]{Claim}
\newtheorem{observation}[cislo]{Observation}
\newtheorem{conjecture}[cislo]{Conjecture}

\theoremstyle{definition}
\newtheorem{definition}[cislo]{Definition}
\newtheorem{notation}[cislo]{Notation}
\newtheorem{remark}[cislo]{Remark}
\newtheorem{note}[cislo]{Note}
\newtheorem{example}[cislo]{Example}

\newtheorem{discussion}[cislo]{Discussion}

\newtheorem{question}[cislo]{Question}
\newtheorem{construction}[cislo]{Construction}
\newtheorem{setup}[cislo]{Setup}

\newtheorem{recollection}[cislo]{Recollection}

\theoremstyle{remark}

\newcommand{\llp}{(\!(}
\newcommand{\rrp}{)\!)}
\newcommand{\llb}{\llbracket}
\newcommand{\rrb}{\rrbracket}

\newcommand{\lL}{\mathbf{L}}

\DeclareMathOperator{\Sch}{Sch}

\DeclareMathOperator{\Coh}{Coh}
\DeclareMathOperator{\Hom}{Hom}

\DeclareMathOperator{\RG}{R\Gamma}

\DeclareMathOperator{\Rep}{Rep}
\DeclareMathOperator{\End}{End}

\DeclareMathOperator{\Mat}{Mat}

\DeclareMathOperator{\mo}{mod}
\DeclareMathOperator{\Alg}{Alg}

\let\lim\relax
\DeclareMathOperator*{\lim}{lim}
\DeclareMathOperator*{\colim}{colim}

\DeclareMathOperator{\coker}{coker}

\DeclareMathOperator{\cone}{cone}

\DeclareMathOperator{\Ind}{Ind}
\DeclareMathOperator{\Maps}{Maps}

\DeclareMathOperator{\Qcoh}{QCoh}
\DeclareMathOperator{\Vect}{Vect}

\newcommand{\cl}{\mathrm{cl}}
\DeclareMathOperator{\dL}{\mathcal{L}}
\DeclareMathOperator{\W}{\mathbb{W}}

\DeclareMathOperator{\id}{id}

\DeclareMathOperator{\Spec}{Spec}
\DeclareMathOperator{\red}{red}
\DeclareMathOperator{\chara}{char}

\DeclareMathOperator{\A}{\mathbb{A}}
\let\P\relax
\DeclareMathOperator{\P}{\mathbb{P}}
\DeclareMathOperator{\Gr}{\mathbf{Gr}}

\DeclareMathOperator{\VB}{Vect}

\newcommand{\s}{\mathfrak{s}}

\newcommand{\D}{D}

\DeclareMathOperator{\Eig}{Eig}

\DeclareMathOperator{\Ad}{Ad}

\DeclareMathOperator{\awn}{awn}
\newcommand{\cyc}{\mathrm{cyc}}

\DeclareMathOperator{\fp}{fp}

\newcommand{\ev}{\mathrm{ev}}
\newcommand{\coev}{\mathrm{coev}}

\DeclareMathOperator{\Cat}{Cat}
\newcommand{\ex}{\mathrm{perf}}

\DeclareMathOperator{\tr}{tr}

\DeclareMathOperator{\reg}{reg}

\DeclareMathOperator{\comp}{\varkappa}
\newcommand{\cat}{\text{cat}}
\newcommand{\geom}{\text{geom}}

\DeclareMathOperator{\Flag}{Flag}
\DeclareMathOperator{\Grass}{Grass}

\DeclareMathOperator{\Fix}{Fix}

\newcommand{\Gm}{\mathbb{G}_m}
\newcommand{\GL}{GL}

\DeclareMathOperator{\X}{\mathbf{X}}
\newcommand{\lo}{\mathrm{L}}
\newcommand{\lop}{\mathrm{L^+}}

\newcommand{\fl}{\mathrm{fl}}

\let\O\relax
\DeclareMathOperator{\O}{\mathscr{O}}
\DeclareMathOperator{\eL}{\mathscr{L}}
\DeclareMathOperator{\eE}{\mathscr{E}}
\DeclareMathOperator{\eF}{\mathscr{F}}

\DeclareMathOperator{\eG}{\mathscr{G}}
\DeclareMathOperator{\eP}{\mathscr{P}}

\DeclareMathOperator{\eH}{\mathscr{H}}
\DeclareMathOperator{\eX}{\mathscr{eX}}

\newcommand{\E}{E}

\DeclareMathOperator{\Pol}{Pol}

\DeclareMathOperator{\F}{\mathbb{F}}
\DeclareMathOperator{\Z}{\mathbb{Z}}

\DeclareMathOperator{\N}{\mathbb{N}}
\DeclareMathOperator{\C}{\mathbb{C}}

\DeclareMathOperator{\pr}{pr}

\newcommand{\pfp}{\mathrm{pfp}}

\DeclareMathOperator{\rank}{rank}

\DeclareMathOperator{\pt}{pt}

\newcommand{\perf}{\mathrm{perf}}

\newcommand{\qcqs}{\mathrm{qcqs}}

\DeclareMathOperator{\Perf}{Perf}

\DeclareMathOperator{\Tor}{Tor}

\newcommand{\Tfrac}[2]{%
  \ooalign{%
    $\genfrac{}{}{1.2pt}1{#1}{#2}$\cr%
    $\color{white}\genfrac{}{}{.8pt}1{\phantom{#1}}{\phantom{#2}}$}%
}

%% file: 0_introduction.tex
\section{Introduction}

\subsection{Motivation and context}

\paragraph{Affine Grassmannians and multiplicative Higgs moduli stacks.}
Motivated by the geometric Langlands program, \cite{Hau21, Hau24} formulated a conjectural picture for mirror symmetry between Higgs moduli spaces for Langlands dual groups. The supports of his mirrors are modelled by families of certain affine Springer fibers. The global functions on these families are conjecturally given by the equivariant cohomology of affine Schubert varieties, the case of partial flag varieties being worked out in \cite{HR23}.

This comparison conceptually clarifies if we replace equivariant cohomology with equivariant algebraic $K$-theory, which relates to functions on families of multiplicative affine Springer fibers, and consequently to multiplicative Higgs moduli spaces \cite{Wan24} by Beauville--Laszlo gluing.

The aim of this paper is to study these $K$-groups and their relationship to the rings of functions on multiplicative affine Springer fibers independently in the setup of perfect algebraic geometry. In particular, we prove the expected isomorphism in this setting -- see Theorem \ref{introtheorem: trace map for affine grassmannian}.

\paragraph{Equivariant algebraic $K$-theory and trace maps.}
Given a group scheme $G$ acting on a qcqs scheme $X$, we have the equivariant algebraic $K$-theory spectrum $K^G(X)$ of \cite{Tho88, TT90} built from the category of $G$-equivariant perfect complexes on $X$. Its zeroth homotopy group $K^G_0(X)$ is a commutative ring via tensor product. 

Starting from the same datum, we can record the fixed points of each $g \in G$ on $X$ and glue them to a family over $G$ denoted $\Fix_G(X)$, see \S \ref{section: fixed-point schemes of group actions} for details. Doing this for the loop group action on the affine Grassmannian returns certain families of multiplicative affine Springer fibers \cite[\S 2.7.3]{Yun15}. 

There is a natural comparison map from equivariant algebraic $K$-theory to $G$-invariant global functions on $\Fix_G(X)$, see \S \ref{section: trace maps from equivariant K-theory}. This is an isomorphism in interesting cases, allowing us to describe such rings by $K^G_0(X)$. Reinterpreting $\Fix_G(X)$ in terms of equivariant Hochschild homology, this is precisely the Dennis trace map \cite{McC94, KM00, DGM12, BGT13, BFN10, Toe14, HSS17, AGR17}, which was put to great use in algebraic $K$-theory and arithmetic geometry over last fifty years -- however, its use for equivariant algebraic $K$-theory of projective varieties has not been explored yet.

For singular schemes, algebraic $K$-theory is hard to compute. In geometric representation theory, this is mostly bypassed by instead computing algebraic $G$-theory -- algebraic $K$-theory of coherent sheaves \cite{Tho88, TT90}. Unfortunately, this simplification discards a lot of important information: for example, $G$-theory does not have a natural ring structure, making it useless for our purpose. 

We believe that it is actually possible to compute equivariant algebraic $K$-theory -- or at least its homotopy-invariant version $KH$, see \cite{Wei89, Hoy21, Kha20} -- for the examples of our interest, and link these computations to geometric representation theory.

\paragraph{Perfect geometry.}
Algebraic geometry over $\F_p$ has the feature of carrying the Frobenius endomorphism $\varphi$ and there is a canonical procedure called perfection which turns $\varphi$ into an automorphism. It was realized in \cite{Zhu14, BS16} that perfection forces desirable homological properties (such as descent for vector bundles along proper maps); this was in particular used to construct the Witt-vector affine Grassmannian as a perfect ind-scheme.

In fact, perfection simplifies equivariant algebraic $K$-theory and the above-discussed trace map, while preserving a good amount of interesting information. In the non-equivariant case, the result of perfection on algebraic $K$-theory is also considered in \cite{KM21, Cou23, AMM22}. It is thus very inviting to study our trace map in the perfect setup: perfect affine Grassmannians provide important examples in their own right, while many technical issues (present in characteristic zero) simplify. 

The representation-theoretic phenomena occurring in perfect geometry are also of considerable interest: they encode Frobenius twisting in positive characteristic representation theory. The effect of perfection on reductive group schemes and their representations was studied by \cite{CW22}.

\subsection{Results}
Our main aim in this paper is to study the trace map \eqref{definition: trace map in the group case}, \eqref{equation: dennis trace map} for affine Schubert varieties $X_{\leq \mu}$ in the perfect affine Grassmannian $\Gr$ for $\GL_n$ over $k=\F_p$ discussed in \S \ref{subsection: affine grassmannians and their perfections}. Let $T$ be the perfected diagonal torus of $\GL_n$. We prove the following.
\begin{theorem}[Theorem \ref{theorem: trace for GL_n affine grassmannian}]\label{introtheorem: trace map for affine grassmannian}
Let $\mu$ be a dominant coweight of $T$ and $X_{\leq \mu}$ the corresponding perfect affine Schubert variety. Then the trace map gives an equivalence
\begin{equation}\label{introeq: trace map}
    \tr: K^T(X_{\leq \mu}, k) \xrightarrow{\simeq} \RG(\Fix_{\frac{T}{T}}(X_{\leq \mu}), \O).
\end{equation}
Moreover:
\begin{itemize}
    \item both sides are supported in homological degree zero,
    \item integrally, $K^T(X_{\leq \mu}) \simeq KH^T(X_{\leq \mu})$. 
\end{itemize}
\end{theorem}

The proof of Theorem \ref{introtheorem: trace map for affine grassmannian} goes as follows. 

We first interpret $\RG(\Fix_{\frac{T}{T}}(X_{\leq \mu}), \O)$ in terms of equivariant Hochschild homology $HH^T(X_{\leq \mu}, k)$ in the sense of \S \ref{section: two viewpoints on hochschild homology}. This puts both sides of \eqref{introeq: trace map} on the equal footing of $k$-linear localizing invariants \cite{BGT13, HSS17, Tam17, LT19}. 
In doing so, we take the opportunity to discuss the trace map, record its properties, relate it to existing literature and show some examples. Most of this is known to experts, but the literature is a bit sparse: the case of global quotients was not exploited in detail and some of our statements even seem to be new. 
To give some intuition, we discuss partial flag varieties where the trace map gives an isomorphism in degree zero in Example \ref{example: partial flag variety}. This identifies their $K^G_0$ with global invariant functions on multiplicative Grothendieck--Springer resolutions.

Secondly, we discuss equivariant algebraic $K$-theory and equivariant Hochschild homology for perfect schemes. After starting this project, we noted that \cite{KM21, AMM22, Cou23} studied non-equivariant algebraic $K$-theory of perfect schemes and we take advantage of their methods. We establish basic structural properties, noting that both sides of the trace map behave compatibly under perfection. Most notably, we check descent for localizing invariants on $T$-equivariant perfect abstract blowup squares.
\begin{theorem}[Theorem \ref{lemma: homotopy fiber square for K_T}]\label{introtheorem: perfect proper excision}
Localizing invariants satisfy proper excision on $T$-equivariant perfect abstract blowup squares. This in particular applies to $K^T(-)$ and $HH^T(-, k)$.
\end{theorem}
We further observe that $K^G_i(X)$ are $\Z[\frac{1}{p}]$-modules for $i \geq 1$ in Lemma \ref{lemma: kratzers argument}. We also prove in Observation \ref{lemma: K-theory and G-theory of perfect schemes} that equivariant $K$-theory and $G$-theory agree for perfect schemes -- consequently, $G$-theory of perfect schemes can be non-connective by Note \ref{note: non-connective G-theory}.

We then deduce Theorem \ref{introtheorem: trace map for affine grassmannian} inductively on the affine Schubert stratification by descent from partial affine Demazure resolutions, the main input being Theorem \ref{introtheorem: perfect proper excision} and semi-orthogonal decompositions for stratified Grassmannian bundles \cite{Jia23}. 

\bigskip

Finally, it is not only the case that the trace map is an abstract isomorphism, but we can employ Theorem \ref{introtheorem: perfect proper excision} to obtain explicit presentations for the rings in question. To showcase this, we compute $K^T_0$ for some singular Schubert varieties in the $\GL_2$ and $\GL_3$ affine Grassmannians -- see Examples \ref{example: K-theory of adjoin GL_2}, \ref{example: K-theory of adjoint GL_3} for details.

\bigskip

We expect direct analogues for $\GL_n$-equivariant (or even $\lop \GL_n$-equivariant) algebraic $K$-theory. Also, our computations can be run for equivariant homotopy $K$-theory $KH^{\GL_n}$ in any characteristic (without perfection); see \cite{Low25}. We further believe the rings in question are then given by the base-change of equivariant cohomology along the map $H^{\bullet}_{\GL_n}(\pt) \to K^{\GL_n}_0(\pt) \to K^{\GL_n}_{\bullet}(\pt)$, the first arrow given in coordinates as $\Z[t_1, \dots, t_n]^{S_n} \to \Z[t_1^{\pm 1}, \dots, t_n^{\pm n}]^{S_n}$. However, we refrain from discussing these questions here.

\bigskip
We close up the paper with more computations of perfect $K$-theory. In particular, we use the existing literature together with Theorem \ref{introtheorem: perfect proper excision} to understand equivarint $K$-theory and the trace map for all perfectly proper toric varieties.
\begin{theorem}[Theorem \ref{theorem: trace map for perfectly proper toric varieties}]
Let $T$ be a perfect split torus and $X$ any perfectly proper toric variety. Then the trace map induces an equivalence
\begin{equation*}
    K^T(X, k) \xrightarrow{\simeq} \RG(\Fix_{\frac{T}{T}}(X), \O).
\end{equation*}
and both sides are supported in homological degrees $\leq 0$. Moreover, $K^T(X) \simeq KH^T(X)$.
\end{theorem}
However, there can be negative degrees for singular $X$.
In particular, we deduce that proper singular toric varieties over $k$ can have nontrivial negative $K^T(-)$ by exhibiting cohomology classes on the right-hand side; this is the case even before perfection -- see Note \ref{note: negative K-theory of perfect toric varieties} for details. We did not find similar computations in the literature.

\subsection{Structure}
We start in \S \ref{section: fixed-point schemes and traces} by introducing fixed-point schemes of group actions and their basic properties. We continue by discussing the trace map from equivariant algebraic $K$-theory to their global functions in elementary terms. We explain the relationship to the existing literature on Hochschild homology and derived loop spaces. The contents of this section work in any characteristic.

In \S \ref{section: perfect schemes} we give a short overview of perfect algebraic geometry, which is the setup for the rest of the paper. We discuss equivariant $K$-theory of perfect schemes in \S \ref{section: perfect K-theory}, explaining some of its main structural properties. In \S \ref{section: perfect trace maps} we show how the previous two sections fit together via the trace map, allowing us to control its behaviour.

We focus on perfect affine Grassmannians in \S \ref{section: trace map for GL_2 affine grassmannian}, proving that the trace map is an isomorphism for the $T$-equivariant $K$-theory of affine Schubert varieties in the $\GL_n$-affine Grassmannian -- see Theorem \ref{theorem: trace for GL_n affine grassmannian}.
In \S \ref{section: some examples} we explicitly compute presentations for these rings in some small examples. In fact, these computations work equally well for equivariant homotopy $K$-theory $KH$ in the classical setup (including characteristic zero).

We conclude with more examples in \S \ref{section: more examples}. These should give intuition about perfect $K$-theory, but are orthogonal to our original motivation. We in particular discuss the trace map for perfect toric varieties in Theorem \ref{theorem: trace map for perfectly proper toric varieties}.

\subsection{Acknowledgements}
I would like to thank the following people for fruitful discussions, helpful sanity checks or comments on previous drafts:
Roman Bezrukavnikov, Jens Niklas Eberhardt, Mischa Elkner, Tamás Hausel, Andres Fernandez Herrero, Adeel Khan, Bernhard Köck, Andrei Konovalov, Quoc Ho, Mirko Mauri, Matthew Morrow, Charanya Ravi, Kamil Rychlewicz, Shyiu Shen, Vova Sosnilo, Georg Tamme, Xinwen Zhu. I would further like to thank Marc Hoyois and the anonymous referee for spotting an error in a previous version.

This work was done during author's PhD at the Institute of Science and Technology Austria (ISTA). It was funded by a DOC Fellowship of the Austrian Academy of Sciences and by the Austrian Science Fund (FWF) 10.55776/P35847. For open access purposes, the author has applied a CC BY public copyright license to any author-accepted manuscript version arising from this submission.

%% file: 1_fixed_point_schemes_and_trace_maps.tex
\section{Fixed-point schemes and traces from equivariant \texorpdfstring{$K$}{K}-theory}\label{section: fixed-point schemes and traces}

We work over a fixed base ring $k$. We denote by $\Sch_k$ the category of $k$-schemes. Let $G$ be a group scheme over $k$ and denote $\Sch_k^G$ the category of $G$-equivariant $k$-schemes. Unless specified otherwise, quotients are taken as stack quotients in the fpqc topology.

\subsection{Fixed-point schemes of group actions}\label{section: fixed-point schemes of group actions}

\paragraph{Definition of fixed point families.}
Suppose $G$ acts on $X \in \Sch^G_k$ via $m: G\times X \to X$.  We define the associated {\it fixed-point scheme} as the fiber product of schemes
\begin{equation}\label{equation: definition of fixed-point schemes}
\begin{tikzcd}
\Fix_G(X) \arrow[r] \arrow[d] & G\times X \arrow[d, "m \times \pr_2"] \arrow[r, "\pr_1"] & G \\
X \arrow[r, "\Delta_X"] & X\times X
\end{tikzcd}
\end{equation}
Its functor of points is given by
\begin{equation*}
    \Fix_G(X): R \mapsto \{ (g, x) \in G(R) \times X(R) \mid gx = x \}.
\end{equation*}
The horizontal composition in \eqref{equation: definition of fixed-point schemes} gives a map $\pi: \Fix_G(X) \to G$. For any $g \in G(k)$, the fiber $\pi^{-1}(g)$ parametrizes those $x \in X(R)$ which are fixed by $g$. 
Given any $S \in \Sch_k$ with a map $S \to G$, we obtain
\begin{equation}\label{equation: fixed-point scheme over S}
\Fix_S(X) := S \underset{G}{\times} \Fix_G     
\end{equation}
and view it as the fixed point family of the restriction of the $G$-action to $S$.

\begin{example}\label{example: fixed-point schemes of partial flag varieties}
For instance, when applied to the full flag variety $X = G/B$ with the obvious left $G$-action, $\Fix_G(X)$ returns the multiplicative Grothendieck--Springer resolution; the fibers of $\pi: \Fix_G(G/B) \to G$ are the multiplicative Springer fibers. For other parabolic subgroups $P$, it is a parabolic version thereof. See Example \ref{example: partial flag variety} for more. These play an important role in geometric representation theory. 
\end{example}

\paragraph{Functoriality and equivariant structure.}
If $X \xrightarrow{f} Y$ is a $G$-equivariant morphism, the defining diagram of $\Fix_G(X)$ naturally maps to the defining diagram $\Fix_G(Y)$ via $f$ on each occurrence of $X$ and $\id_G$ on the group. By the universal property of fiber products, this induces a map
\begin{equation*}
\Fix_G(f): \Fix_G(X) \xrightarrow{} \Fix_G(Y),    
\end{equation*}
which we also call $f$ by abusing notation. 

Now let $g \in G(R)$ and note the following: a point $x \in X(R)$ is fixed by an element $h \in G(R)$ if and only if its translate $gx \in X(R)$ is fixed by $ghg^{-1} \in G(R)$. In other words, $\Fix_G(X)$ carries a functorial $G$-action
\begin{align*}
   (\Ad_G \times m):  G \times \Fix_G(X) & \xrightarrow{} \Fix_G(X), \\
   (g, (h, x)) & \mapsto (ghg^{-1}, gx). 
\end{align*}

In particular, $\Fix_G(\pt) \in \Sch^G_k$ returns $G$ with the adjoint action $G\times G \xrightarrow{\Ad} G$, $(g, h) \mapsto ghg^{-1}$. We denote the stack quotient by $\frac{G}{G}$. For general $X \in \Sch^G_k$ we denote the stack quotient of $\Fix_G(X)$ by $G$ as $\Fix_{\frac{G}{G}}(X)$. This has a canonical structure map
\begin{equation*}
    \Fix_{\frac{G}{G}}(X) \xrightarrow{\pi} \tfrac{G}{G}.
\end{equation*}
Taking quotients of \eqref{equation: definition of fixed-point schemes} further gives the following fiber square of global quotient stacks.
\begin{equation}\label{equation: fixed-point stack as fiber product}
\begin{tikzcd}
\Fix_{\frac{G}{G}}(X) \arrow[r] \arrow[d] & \frac{G\times X }{G} \arrow[d, "m \times \pr_2"] \arrow[r, "\pi"] & \frac{G}{G} \\
G \backslash X \arrow[r, "\Delta_X"] & \frac{X\times X}{G}
\end{tikzcd}
\end{equation}

As in \eqref{equation: fixed-point scheme over S}, for any algebraic stack $S$ over $k$ with a map $S \to \frac{G}{G}$ we write $\Fix_S(X)$ for the base change of $\Fix_{\frac{G}{G}}(X)$ to $S$. In particular, base change along the quotient map $G \to \frac{G}{G}$ returns back $\Fix_G(X)$, making our notation consistent.

\paragraph{Steinberg section.}\label{paragraph: steinberg section}
Let $G$ be a reductive group over a field and denote $\s := \Tfrac{G}{G}$ the GIT quotient of $G$ by its adjoint action. There is the canonical affinization map $\frac{G}{G} \to \s$. Denote further $G^{\reg} \hookrightarrow G$ the regular locus of $G$ consisting of points whose centralizer has the minimal possible dimension (equal to the rank of $G$) \cite[\S 4]{Hum95}; this is an open subvariety of codimension $\geq 2$ by \cite[\S 4.13]{Hum95}. Let $\rho: G^{\reg} \hookrightarrow G \to \s$ be the natural map. A {\it Steinberg section} is a section of this map $\rho$. Such a section is not unique, but always exists at least when $G$ is split, semisimple, simply connected group or $GL_n$ by \cite[\S 4.15]{Hum95}.

Via such a Steinberg section, $\s \hookrightarrow G$ is a closed subscheme intersecting every regular conjugacy class of $G$ in a single geometric point.
It is often very useful to restrict attention to $\Fix_{\s}(X)$. This is a scheme whose global functions often agree with global functions on the stack $\Fix_{\frac{G}{G}}(X)$.

\paragraph{Open and closed immersions.} 
\begin{proposition}\label{proposition: open and closed immersions}
Assume that a $G$-equivariant morphism $f: U \to X$ has one of the following properties. Then the same is true for $f: \Fix_G(U) \to \Fix_G(X)$.
\begin{itemize}
    \item[(i)] monomorphism
    \item[(ii)] open immersion
    \item[(iii)] closed immersion
\end{itemize}
In particular, $\Fix_G(-)$ preserves Zariski covers and locally closed stratifications.
\end{proposition}
\begin{proof}
    If $f: U \to X$ is a $G$-equivariant monomorphism, one can immediately check that on functors of points
    \begin{equation*}
        \Fix_G(U)(R) 
        = \{ (g, u) \in G(R)\times U(R) \mid gu=u\}
        = (\Fix_G(X)\times_{X} U)(R)
    \end{equation*}
    so that the map on fixed point schemes is also a monomorphism, proving (i).
    Moreover, the above computation shows that the diagram
    \begin{equation}\label{equation: fixed-point schemes and base change}
    \begin{tikzcd}
    \Fix_G(U) \arrow[r] \arrow[d] & \Fix_G(X) \arrow[d]\\
    U \arrow[r] & X
    \end{tikzcd}
    \end{equation}
    is fibered.   
    To prove (ii), (iii), note that these properties are stable under base change; the result thus follows from the fiber square \eqref{equation: fixed-point schemes and base change}; similarly for Zariski covers. The case of locally closed stratifications is now immediate.
\end{proof}

\paragraph{Properness and global functions.}
\begin{proposition}\label{proposition: basic properties of fixed point families}
The fixed point scheme satisfies the following.
\begin{itemize}
    \item[(i)] If $X \in \Sch_k$ is separated, then the $\Fix_G(X) \to G\times X$ is a closed immersion.
    \item[(ii)] If $X \in \Sch_k$ is proper, then $\Fix_G(X) \to G$ is also proper. 
    \item[(iii)] If $X \in \Sch_k$ is proper, then $\Fix_\frac{G}{G}(X) \to \frac{G}{G}$ is also proper.
\end{itemize}
\end{proposition}
\begin{proof}
When $X$ is separated over $k$, $\Delta_X$ is a closed immersion, so the same is true for its base change $\Fix_G(X) \to G\times X$, proving (i).
When $X$ is proper over $k$, then also $G\times X \to G$ is proper. Since $X$ is in particular separated over $k$, the map $\Fix_G(X) \to G\times X$ is a closed immersion by part (i), so in particular proper. Therefore also the composition $\Fix_G(X) \to G\times X \to G$ is proper, proving (ii).
Part (iii) follows from (ii) by fpqc descent for proper maps.
\end{proof}

\begin{notation}\label{notation: homological grading}
We will use the homological grading on derived global sections $\RG(X, -)$ of coherent sheaves. We denote it by lower indices as
\begin{equation*}
    H_i(X, \O) = \RG_i(X, \O) := \RG^{-i}(X, \O), \qquad \forall i \in \Z.
\end{equation*}
The reason for this notation is compatibility with the homological grading on algebraic $K$-theory.
\end{notation}

\begin{remark}
Given a $G$-action on a scheme $X$, we tautologically have 
\begin{equation*}
    H_0(\Fix_{\frac{G}{G}}(X), \O) = H_0(\Fix_G(X), \O)^G.
\end{equation*}
We have seen that whenever $X$ is proper, the structure map $\Fix_G(X) \to G$ is proper as well, so the global functions $H_0(\Fix_G(X), \O)$ form a finitely generated $H_0(G, \O)$-module. 
\end{remark}

\subsection{Comparison to derived loop spaces}

\paragraph{Reduced and derived versions.}
There are the following versions. On the one hand, $\Fix_\frac{G}{G}(X)$ will be often non-reduced, and we can take its reduction $\Fix^{\red}_{\frac{G}{G}}(X)$. On the other hand, we may take the defining fiber product \eqref{equation: definition of fixed-point schemes} in derived schemes, giving a derived enhancement which we denote $\Fix^{\lL}_{\frac{G}{G}}(X)$. We get closed immersions
\begin{equation}\label{equation: reduced, classical and derived fixed-point schemes}
    \Fix^{\red}_{\frac{G}{G}}(X) \hookrightarrow \Fix_{\frac{G}{G}}(X) \hookrightarrow \Fix^{\lL}_{\frac{G}{G}}(X).
\end{equation}

\paragraph{Derived loop spaces.}
The above construction is not new: it is a special case of the construction of derived loop stacks of \cite[\S 2.4]{BFN10}, \cite[\S 4.4]{Toe14}; also see \cite{KP16}. This interpretation has useful structural properties, so we spell out the (well-known) relationship.\footnote{We thank Quoc Ho for helpful discussions regarding this.}

For any derived stack $\eX$ over $k$, its derived loop space $\dL \eX$ is defined by the derived fiber product
\begin{equation}\label{equation: definition of derived loop spaces}
    \begin{tikzcd}
     \dL \eX \arrow[d] \arrow[r] & \eX \arrow[d, "\Delta"] \\
     \eX \arrow[r, "\Delta"] & \eX \underset{k}{\times} \eX \arrow[r] \arrow[d] & \eX \arrow[d] \\
      & \eX \arrow[r] & \pt
    \end{tikzcd}
\end{equation}
Alternatively, this is the mapping space $\dL(\eX) = \Maps(S^1, \eX)$ of derived stacks. In other words, it is the derived inertia stack of $\eX$.
The same definition can be done in a non-derived way, giving the classical substack $\dL(\eX)^{\cl} \hookrightarrow \dL(\eX)$.

\paragraph{Identification of (derived) fixed-point schemes and (derived) loop spaces.}
In order to compare our fixed-point schemes to loop spaces, consider the following (non-derived) fiber squares.
\begin{equation}\label{equation: fixed-point scheme as diagonal}
\begin{tikzcd}
\Fix_{\frac{G}{G}}(X) \arrow[r] \arrow[d] & \frac{G\times X }{G} \arrow[d, "m \times \pr_2"] \arrow[r] &  G\backslash X \arrow[d, "\Delta_{G\backslash X}"]\\
G \backslash X \arrow[r, "\Delta_X"] & \frac{X\times X}{G} \arrow[r] & G\backslash X \times G\backslash X  \\
\end{tikzcd}
\end{equation}
The left-hand square was constructed in \eqref{equation: fixed-point stack as fiber product}. The right-hand square can be easily constructed on the level of functors of points as follows. Let $R \in \Sch_k$ be a test scheme. Then we obtain the groupoids:

\begin{align*}
    (G \backslash X) (R) 
    &= \left\{ (\eP, f) \ \mid 
    \substack{\eP \text{ principal } G\text{-bundle on } R, \\ 
    f:\eP \to X \ G \text{-equivariant map}} \right\},\\
    \left( \frac{X \times X}{G} \right) (R) 
    &= \left\{ (\eP_1, f_1, f_2) \ \mid 
    \substack{\eP \text{ principal } G\text{-bundle on } R, \\ 
    f_1, f_2:\eP \to X \ G \text{-equivariant maps}} \right\},\\
    \left(  G\backslash X \times G\backslash X \right) (R) 
    &= \left\{ (\eP_1, \eP_2, f_1, f_2) \ \mid
    \substack{\eP_1, \eP_2 \text{ principal } G\text{-bundles on } R, \\ 
    f_1:\eP_1 \to X, \ f_2:\eP_2 \to X \ G \text{-equivariant maps}} 
    \right\},\\
    \left( \frac{G \times X}{G} \right) (R) 
    &= \left\{ (\eP, f_1, f_2) \ \mid 
    \substack{\eP \text{ principal } G\text{-bundle on } R, \\ 
    f_1:\eP \to G, \ f_2:\eP \to X \ G \text{-equivariant maps}} \right\} \\
    &= \left\{ (\eP, f, \psi) \ \mid 
    \substack{\eP \text{ principal } G\text{-bundle on } R, \\
    \psi: \eP \to \eP \text{ automorphism},  \\ 
    f:\eP \to X \ G \text{-equivariant map} } \right\}.
\end{align*}
Here, the first line is the definition of a stack quotient through its functor of points. The next three lines follow from this definition by elementary manipulations. For the last line, fix a principal $G$-bundle $\eP$ over $R$. We only need to identify $G$-equivariant maps $f_1: \eP \to G$ with automorphisms $\psi: \eP \to \eP$. We construct this bijection on functors of points. Let $R'$ be any test $R$-algebra and $s \in \eP(R')$ any section of $\eP$ over $R'$. In the forward direction, if $f_1: s \mapsto f_1(s)$, we define $\psi: s \mapsto f_1(s) \cdot s$. In the backward direction, given $\psi: s \mapsto \psi(s)$, we define $f_1(s)$ to be the unique element of $G(R')$ sending $s$ to $\psi(s)$. These maps are inverse bijections, so the final line follows.

Now, this last expression is indeed given by the fiber product of three preceding groupoids along the obvious maps. Altogether, the right-hand square in \eqref{equation: fixed-point scheme as diagonal} is fibered as claimed.

Finally, let us remark that the horizontal maps in the right-hand square are flat -- indeed, they are quotient maps by the group scheme $G$ over $k$. Hence the right-hand square is fibered in the derived sense as well.

\begin{lemma}\label{corollary: fixed-point schemes as loop stack}
We have the following isomorphism of stacks:
\begin{equation*}
    \Fix_{\frac{G}{G}}(X) = \dL(G\backslash X)^{\cl} \qquad \text{and} \qquad \Fix^{\lL}_{\frac{G}{G}}(X) = \dL(G\backslash X).
\end{equation*}
\end{lemma}
\begin{proof}
    As both small squares in the \eqref{equation: fixed-point scheme as diagonal} are fiber squares of classical stacks, the big rectangle is fibered also, proving the first part. 
    
    The derived version is analogous. Namely, replace $\Fix_{\frac{G}{G}}(X)$ by $\Fix^{\lL}_{\frac{G}{G}}(X)$ in \eqref{equation: fixed-point scheme as diagonal}. Then the left-hand square is fibered in derived stacks by definition; the same is true for the right-hand square by the non-derived case and flatness of the horizontal maps. We again conclude by the two-out-of three property. Also see \cite[Proposition 2.1.8]{Chen20} for related discussion (in characteristic zero).
\end{proof}

\paragraph{Compatibility with twisted products.}

\begin{lemma}\label{lemma: derived loop stack commutes with limits}
Taking $\dL(-)$ commutes with all limits in derived stacks. Taking $\dL^{\cl}(-)$ commutes with all limits in stacks.    
\end{lemma}
\begin{proof}
This is clear since $\dL$ and $\dL^{\cl}$ are defined in two steps using limit diagrams \eqref{equation: definition of derived loop spaces} in derived stacks resp. stacks.    
\end{proof}

\begin{setup}
Let $X_1$ be a $G_1$-scheme, $\rho_1: \eP_1 \to X_1$ a $G_1$-equivariant principal $G_2$-bundle. Let $X_2$ be a $G_2$-scheme.
Then we have the associated twisted product, which is again a $G_1$-scheme:
\begin{equation*}
X_1 \widetilde{\times} X_2 = \eP_1 \overset{G_2}{\times} X_2.     
\end{equation*} 
This formally amounts to the following fiber product of global quotient stacks:
\begin{equation}\label{equation: twisted product as fiber product}
    \begin{tikzcd}
        G_1 \backslash \left( X_1 \widetilde{\times} X_2 \right) \arrow[r] \arrow[d] & G_2 \backslash X_2 \arrow[d] \\
       G_1 \backslash X_1 \arrow[r, "\eP_1"] & G_2 \backslash \pt
    \end{tikzcd}
\end{equation}
\end{setup}

The next useful lemma follows easily from the loop space interpretation of fixed-point schemes (but is less obvious from the fixed-point scheme definition).

\begin{lemma}\label{lemma: fixed point schemes of twisted products}
The fixed-point stack of the twisted product fits into the fibered square
\begin{center}
\begin{tikzcd}
\Fix_{\frac{G_1}{G_1}}(X_1 \widetilde{\times} X_2) \arrow[d] \arrow[r] & \Fix_{\frac{G_2}{G_2}}(X_2) \arrow[d, "\pi_2"] \\
\Fix_{\frac{G_1}{G_1}}(X_1) \arrow[d, "\pi_1"] \arrow[r, "\rho_1"] & \frac{G_{2}}{G_{2}}  \\
\frac{G_1}{G_1}
\end{tikzcd}
\end{center}
\end{lemma}
\begin{proof}
Apply $\dL(-)$ to \eqref{equation: twisted product as fiber product}. The resulting square is a pullback by Lemma \ref{lemma: derived loop stack commutes with limits}. We conclude by Lemma \ref{corollary: fixed-point schemes as loop stack} together with $\Fix_{\frac{G}{G}}(\pt) =\frac{G}{G}$.   
\end{proof}

\subsection{Two viewpoints on Hochschild homology}\label{section: two viewpoints on hochschild homology}
We recall how derived fixed-point schemes relate to equivariant Hochschild homology. This interpretation has useful structural properties such as the projective bundle formula.

We first fix some standard notation. We write $\Vect(-)$ for the category of vector bundles, $\Qcoh(-)$ for the category of quasi-coherent sheaves, $\Perf(-)$ for the $\infty$-category of perfect complexes, $\D_{\Qcoh}(-)$ for the $\infty$-category of quasi-coherent sheaves. Given a group scheme $G$ over $k$, we denote the categories of $G$-equivariant objects by the superscript $(-)^G$; this matches the corresponding category on the associated quotient stack.

\paragraph{Flawless stacks.}
We write $\Ind(-)$ for the functor of ind-completion of $\infty$-categories.
The following definition is due to \cite[Definition 3.2]{BFN10}. 
\begin{definition}[Flawless stacks]
A derived algebraic stack $\eX$ over $k$ is called {\it flawless}\footnote{The authors of \cite{BFN10} call this property {\it perfect} (as it has to do with perfect complexes), but there is an unfortunate clash of terminology with the notion of a {\it perfect stack} in positive characteristic arithmetic geometry (referring to Frobenius being automorphism). Since both notions appear in this paper, we invented the word {\it flawless} to avoid confusion.} iff it has affine diagonal and the canonical comparison map gives an equivalence of $\infty$-categories
\begin{equation*}
\Ind(\Perf(\eX)) \simeq \D_{\Qcoh}(\eX).    
\end{equation*}
\end{definition}

There are many flawless stacks \cite[p.912]{BFN10}. In characteristic zero, a quotient of a quasi-projective derived scheme by an affine group scheme is flawless. 
In characteristic $p$, a quotient of a quasi-projective derived scheme by a linearly reductive group (such as $T = \Gm^r$) is flawless. 

On the other hand, reductive groups in characteristic $p$ are usually not linearly reductive, and global quotients by such are usually not flawless.

\paragraph{Categorical and geometric models for Hochschild homology.}
There are the following two models of Hochschild homology over $k$, defined for any derived stack over $k$. They agree on flawless stacks. For concreteness, we stick to the case of a global quotient $\eX = G \backslash X$. 
\begin{itemize}
    \item Take the small stable $k$-linear $\infty$-category of perfect complexes $\Perf(\eX)$. Taking the geometric realization of the $k$-linear cyclic nerve produces the algebra object
    \begin{equation*}
       HH^G(X, k) := HH(\eX, k) := HH(\Perf(\eX), k) = | N^{\cyc}_{/k}\Perf(\eX) | \in \D(k).
    \end{equation*}
    See \cite{Hoy18}, \cite[Definition 2.2.18]{Chen20}. This fits into the framework of $k$-linear localizing invariants of \cite{BGT13, HSS17}.
    \item Take the derived loop space $\mathcal{L}(\eX) = \Fix^{\lL}_{\frac{G}{G}}(X)$ and consider the algebra object given by the global sections of its structure sheaf
    \begin{equation*}
        \RG(\dL \eX, \O) = \RG(\Fix^{\lL}_{\frac{G}{G}}(X), \O) \in \D(k).
    \end{equation*}  
\end{itemize}

\begin{remark}\label{remark: HH as categorical trace}
Recall that Hochschild homology of the small stable $k$-linear $\infty$-category $\Perf(\eX)$ can be equivalently computed as the image of the monoidal identity $k$ in $\Perf(k)$ under the categorical trace in dualizable presentable stable $k$-linear $\infty$-categories given by the $\Ind$-completions: 
\begin{align*}
    \Ind (\Perf(k)) \xrightarrow{\coev} \Ind &(\Perf(\eX)) \otimes \Ind (\Perf(\eX))^{\vee} \xrightarrow{\ev} \Ind(\Perf(k)) \\
    k \qquad &\longmapsto \qquad HH(\eX, k).
\end{align*}
See \cite[\S 2.2]{Chen20}.
\end{remark}

\begin{discussion}[Comparison and agreement]
There is a natural comparison map
\begin{equation}\label{equation: comparison map for HH}
    HH(\eX, k) \xrightarrow{\comp} \RG(\dL \eX, \O)
\end{equation}
which is an equivalence when $\eX$ is flawless.

The existence of the comparison map $\varkappa$ follows from \cite[Theorem 6.5]{HSS17} -- to bootstrap this from their finer results, one needs to compose with the map from global sections of their Tannakian loop space to the global sections of the derived loop space, and furthermore with the inclusion of homotopy fixed-points to all functions.

If $\eX$ is a flawless derived algebraic stack over $k$, the comparison map $\varkappa$ is an equivalence by the discussion \cite[Example 2.2.20]{Chen20}. Strictly speaking, \cite{Chen20} works in characteristic zero, but the only input is \cite[Theorem 1.2]{BFN10} valid over any base ring $k$. (Also note that the case of qcqs $k$-schemes is classical: it follows by Zariski descent from the case of $k$-algebras.)

For reader's convenience, we recall the arguments from \cite{Chen20, BFN10} in more detail. Assume $\eX$ is flawless. By Remark \ref{remark: HH as categorical trace}, we can then compute $HH(\eX, k)$ as the image of the monoidal identity under
\begin{equation}
D_{\Qcoh}(k) \xrightarrow{\coev} D_{\Qcoh}(\eX) \otimes D_{\Qcoh}(\eX)^{\vee} \xrightarrow{\ev} D_{\Qcoh}(k).    
\end{equation}
Using \cite[Theorem 4.7 and Corollary 4.8]{BFN10}, we can identify $D_{\Qcoh}(\eX) \otimes D_{\Qcoh}(\eX)^{\vee}$ with $ D_{\Qcoh}(X \times X)$. Under this identification, coevaluation and evaluation are given by $\coev = R\Delta_* \circ Lp^*$ and $\ev = Rp_* \circ L\Delta^*$ along the correspondences
\begin{center}
\begin{tikzcd}[row sep=0em]
    & & \dL(\eX) \arrow[swap]{ld}{q_1} \arrow{rd}{q_2} & & \\
    & \eX \arrow[swap]{ld}{p} \arrow[swap]{rd}{\Delta} & & \eX \arrow{ld}{\Delta} \arrow{rd}{p} \\
  \pt & & \eX \times \eX & & \pt.  
\end{tikzcd}
\end{center}
Since all terms in this diagram are flawless by \cite[\S 3.3]{BFN10}, we can use derived base-change \cite[Proposition 3.10]{BFN10} in the middle square to get the desired identification
\begin{equation*}
HH(\eX, k) = \coev \circ \ev (k) 
= Rp_* \circ L\Delta^* \circ R\Delta_* \circ Lp^* (k)
= R(pq_2)_* \circ L(pq_1)^* (k)
= \RG(\dL(\eX), \O).
\end{equation*}
\end{discussion}

\begin{question}
    Let $G = \GL_n$ over $k=\F_p$. Take $X \in \Sch^G_k$. The stacks $\eX = G \backslash X$ are usually not flawless. For example, taking $X = \pt$, the classifying stack $BG = G \backslash \pt$ is not flawless: its structure sheaf is perfect but not compact. Is the comparison map \eqref{equation: comparison map for HH} still an isomorphism?

    To the best of our knowledge, this is an open question. The standard approach to such comparison is built on flawlessness, which does not appear to be the optimal assumption. See also \cite[Remark 2.2.21]{Chen20} for similar questions in characteristic zero.
\end{question}

\subsection{Trace maps from equivariant \texorpdfstring{$K$}{K}-theory}\label{section: trace maps from equivariant K-theory}
We now recall the trace map from equivariant algebraic $K$-theory to functions on the fixed-point scheme and its relationship to the Dennis trace map.

\paragraph{Conventions for equivariant \texorpdfstring{$K$}{K}-theory.}
Let $\Sch_k^{\qcqs}$ be the category of quasi-compact quasi-separated $k$-schemes. 
\begin{notation} Given $X \in \Sch_k^{\qcqs}$, we denote by $K(X)$ the non-connective algebraic $K$-theory spectrum of the category of perfect complexes $\Perf(X)$ on $X$ in the sense of \cite{TT90}. We denote $K_i(X)$, $i \in \Z$ its homotopy groups. 

Let $G$ be an affine algebraic group over $k$.
If $X \in \Sch_k^{\qcqs, G}$, we denote $K^G(X)$ its equivariant algebraic $K$-theory spectrum \cite{Tho88}. This is the algebraic $K$-theory spectrum of the category $\Perf^G(X)$ of $G$-equivariant perfect complexes on $X$. Equivalently, it is the $K$-theory spectrum of perfect complexes $\Perf(G \backslash X)$ on the global quotient stack $G \backslash X$. We denote by $K^G_i(X)$, $i \in \Z$ its homotopy groups.    
\end{notation}

\begin{notation}
We write $G(X)$ for the algebraic $K$-theory spectrum of cohomologically bounded pseudo-coherent complexes on $X$ \cite[Definition 3.3]{TT90}, similarly $G^G(X)$ for the equivariant version \cite{Tho88}. 
\end{notation}

\begin{notation}\label{notation: KH}
We denote $KH(X)$ the homotopy $K$-theory spectrum in the sense of \cite{Wei89} and $KH^G(X)$ its equivariant version defined via the simplicial spectrum $K^G(\Delta^{\bullet} \times X)$.
\end{notation}

\begin{notation}\label{notation: K-theory with coefficients}
We write $\otimes_R$ for tensor product over $R$. We denote $K_i(X)_k := K_i(X) {\otimes}_{\Z} k$.
We write $\otimes^{\lL}_R$ for derived tensor product over $R$ and $\otimes^i_R := \Tor^R_i(-, -)$ for the Tor-groups. 

Given a coefficient ring $\Lambda$, we denote $K(X, \Lambda)$ the $K$-theory spectrum with coefficients in $\Lambda$. Taking $\Lambda = \F_p$, this has homotopy groups
\begin{equation*}
    K_i(X, \F_p) \cong \left(K_i(X) {\otimes}^0_{\Z} \F_p \right) \oplus \left(K_{i-1}(X) {\otimes}^1_{\Z} \F_p \right).
\end{equation*}
\end{notation}

\paragraph{The equivariant trace map in degree zero.}
Let $X \in \Sch^{\qcqs, G}_k$ and $\eE \in \VB^G(X)$ be a $G$-equivariant vector bundle on $X$. Let $x \in \Fix_g(X)(k)$ be a fixed point of a given $g \in G(k)$. Then $g$ acts on the fiber $\eE_x$ and we may take the trace $\tr_g(\eE_x) \in k$. This idea plays out as follows.

\begin{construction}\label{construction: trace map in degree zero}
Let $X \in \Sch^{\qcqs, G}_k$. Then we define the ring homomorphism
\begin{equation}\label{definition: trace map in the group case}
\begin{gathered}
    \tr : K^G_0(X)  \xrightarrow{} H_0(\Fix_{\frac{G}{G}}(X), \O) \\
    [\eE] \mapsto \left( (g, x) \mapsto \tr_{g} (\eE_x) \right)
\end{gathered}
\end{equation}

Indeed, let $\eE \in \Perf^G(X)$. Pick an affine test scheme $\Spec R \in \Sch_k$ and let $(g, x) \in \Fix_G(X)(R)$ be an $R$-valued point. 
Then the pullback $\eE_x \in \Perf(R)$ is strictly perfect -- it is represented by a bounded complex of projective $R$-modules acted on by $g$. Passing to an affine Zariski cover of $\Spec R$, we may even assume all of its terms are free.
Then one can take the usual trace of $g$ on this complex
$$\tr_{(g,x)} \eE := \sum_i (-1)^i \tr_g \eE_{x, i}$$
and obtain a global function on $\Spec R$.

This association is clearly additive: given a short exact sequence $0 \to \eE' \to \eE \to \eE'' \to 0$ in $\Perf^G(X)$, one sees $\tr_{(g,x)} \eE = \tr_{(g,x)} \eE' + \tr_{(g,x)} \eE''$.
Moreover, if $h: \eE' \to \eE$ is a quasi-isomorphism in $\Perf^G(X)$, then $\tr_{(g, x)} \eE' = \tr_{(g, x)} \eE$. Indeed, the mapping cone $\eE'' := \cone(h) \in \Perf^G(X)$ of $h$ is acyclic, hence it suffices to see that its trace is zero by additivity. This can be easily seen by induction on the degrees in which $\eE''$ is nonzero by splitting off the top degree: one can always produce a short exact sequence $0 \to \eG' \to \eE'' \to \eG'' \to 0$ in $\Perf^G(X)$ where $\eG'$ is supported in a shorter interval and $\eG''$ is of the form $(\dots \to 0 \to R^{\oplus a} \xrightarrow{=} R^{\oplus a} \to 0 \dots )$.

Altogether, this assembles into the trace map \eqref{definition: trace map in the group case}: we land in algebraic functions, as we can test on arbitrary $R$-points of $\Fix_G(X)$. The resulting function is also $G$-invariant, as taking traces is invariant under the adjoint action of $G$ on itself. Finally, the trace factors through $K^G_0(X)$ by \cite[\S 1.5.6]{TT90} and the above.    
\end{construction}

\begin{observation}
The trace map $\tr$ is a homomorphism of rings.
\end{observation}
\begin{proof}
The trace of a tensor product is the product of traces of factors.    
\end{proof}
This observation is essential for our applications: we can describe rings of functions on interesting spaces via equivariant $K$-theory. Moreover, it explains why we need equivariant $K$-theory (and not $G$-theory, which doesn't have a natural ring structure). In fact, the above construction of the trace map would not even work for $G$-theory.

\begin{remark}
Since $X$ is defined over $k$, the trace extends to
\begin{equation*}\label{equation: naive k-linear trace map}
   \tr: K^G_0(X)_k \xrightarrow{} H_0(\Fix_{\frac{G}{G}}(X), \O).
\end{equation*}
See Discussion \ref{discussion: traces from K-theory to HH} for a finer version.
Reformulating, the trace map amounts to a morphism
\begin{equation}
    \tr_X: \Fix_{\frac{G}{G}}(X) \to \Spec K^G_0(X)_k
\end{equation}
which factors through the affinization of the left-hand side.
\end{remark}

\paragraph{The Dennis trace map.}
The above-discussed degree zero definition lifts to the level of spectra, and even factors through finer versions of Hochschild homology. This is a classical construction in algebraic $K$-theory. For explicit formulas, see \cite{McC94, KM00} or the encyclopedic treatments \cite{Mad95, DGM12}. Also see \cite{BGT13} for a modern approach via localizing invariants. It was further interpreted geometrically in terms of derived loop stacks \cite{BFN10, Toe14, HSS17, AGR17}.  We recall the simplest version sufficient for our purposes.

\begin{recollection}\label{recollection: localizing invariants}
Let $\Cat^{\ex}_k$ be the $\infty$-category of small, stable, idempotent complete, $k$-linear $\infty$-categories and exact functors. We will use the notion of localizing invariants $\E(-)$ on $\Cat^{\ex}_k$ in the sense of \cite[\S4--5]{HSS17}. The following are notable examples relevant to this paper.
\begin{itemize}
    \item[(i)] non-connective algebraic $K$-theory $K(-)$ and its version with coefficients,
    \item[(ii)] homotopy $K$-theory $KH(-)$ and its version with coefficients,
    \item[(iii)] Hochschild homology $HH(-, k)$ relative to $k$.
\end{itemize}
In fact, $K(-)$ is a universal localizing invariant, see \cite{HSS17, BGT13}. The case of $KH(-)$ follows formally, see \cite{LT19}. The versions with coefficients are also clear. Finally, $HH(-, k)$ is a localizing invariant on $\Cat^{\ex}_k$ in the sense of \cite{HSS17} by the discussion in \cite[p.914 after 3.8]{LT19}.
\end{recollection}

\begin{notation}
We evaluate localizing invariants on derived stacks through their $\infty$-categories of perfect complexes $\Perf(-)$: given a derived stack $\eX$ over $k$ and a localizing invariant $\E$, we denote
\begin{equation*}
    \E(\eX) := \E(\Perf(\eX)).
\end{equation*}
Given $X \in \Sch^{\qcqs, G}_k$, we denote $\E^G(X) := \E(G\backslash X) = \E(\Perf(G \backslash X)) = \E(\Perf^G(X))$.
\end{notation}

\begin{discussion}[Traces]\label{discussion: traces from K-theory to HH}
There are the following trace maps:
\begin{equation}\label{equation: dennis trace map}
    \begin{tikzcd}
        K^G(X, k) \arrow[bend left=15, rr, "\tr_{\geom}"] \arrow[swap]{r}{\tr_{\cat}} & HH^G(X, k) \arrow[swap]{r}{\comp} & \RG(\Fix^{\lL}_{\frac{G}{G}}(X), \O).
    \end{tikzcd}
\end{equation}
The categorical trace $\tr_{\cat}$ exists as $K^G(X, k)$ and $HH^G(X, k)$ fit in the framework of $k$-linear localizing invariants of \cite{BGT13, HSS17}. 
The geometric trace $\tr_{\geom} = \tr_{\geom} \circ \comp$ extends the degree zero definition presented above.    
\end{discussion}

\begin{remark}
This gives a natural map $K^G_0(X, k) \to H_0(\Fix^{\lL}_{\frac{G}{G}}, \O)$ extending Remark \ref{equation: naive k-linear trace map}.
\end{remark}

\begin{recollection}\label{recollection: semiorthogonal decompositions and localizing invariants}
Localizing invariants are in particular additive \cite[Definitions 5.11 and 5.16]{HSS17} or \cite[Definition 2.6]{Kha18}. They send semi-orthogonal decompositions to direct sums by \cite[Lemma 2.8]{Kha18}.

In particular, this is the case for $K(-)$, $KH(-)$, $K(-, k)$ and $HH(-, k)$; such direct sum decompositions are compatible with the trace map.
\end{recollection}

\paragraph{Local constancy along fibers.}
The $T$-equivariant trace map from Construction \ref{construction: trace map in degree zero} lands in functions which are locally constant along fibers. Therefore, it can be close to isomorphism only in situations when all function of $\Fix_{\frac{T}{T}}(X)$ have this property, e.g. when $X$ is proper. We include this observation for context. 

\begin{notation}
Let $S$ be a base scheme and $Z \in \Sch_S$ with structure map $\pi: Z \to S$. A function $f \in H_0(Z, \O)$ is called {\it locally constant relatively to $S$} if it Zariski locally on $Z$ comes as a pullback of a function on $S$ along $\pi$.
\end{notation}

\begin{lemma}\label{lemma: trace functions are locally constant on fibers for T}
Let $k$ be a base field, $T$ a torus, $X \in \Sch^{\qcqs, T}_k$. Then the image of $\tr: K^T_0(X)_k \to H_0(\Fix_T(X), \O)$ lands in functions which are locally constant relatively to $T$.
\end{lemma}
\begin{proof}

Let $\eE \in \Perf^T(X)$ be a perfect complex and $\tr \eE \in H_0(\Fix_T(X), \O)$ the corresponding function.
We only need to prove the claim Zariski locally on $\Fix_T(X)$.
Fix a test algebra $R \in \Alg_k$ with $S = \Spec R$ and let $(t, x) \in \Fix_T(X)(R)$. Such $(t, x)$ form a basis of the Zariski topology on $\Fix_T(X)$, so it is enough to see that $\tr_t(\eE)$ is a pullback of a function along the projection $\Fix_T(X) \to T$.
\begin{equation*}
    \begin{tikzcd}
        S \arrow[r, "{(t, \, x)}"] \arrow{rdd}[swap]{t} & \Fix_T(X) \arrow[d] \\
        & T \times X \arrow[d] \\
        & T
    \end{tikzcd}
\end{equation*}
But $\eE_{(t, x)}$ is a $T$-equivariant perfect complex on $S$ and we may without loss of generality assume it is a vector bundle. Restricting to a single connected component of $S$, we may further assume that $\eE_{(t, x)}$ has constant rank. Equivalently, it is a $\Z^d$-graded projective module $M = \bigoplus_{i \in \Z^d} M_i \cdot t^i$ with each $M_i$ projective of constant rank. After choosing an isomorphism $T \cong \mathbb{G}_m^d$, the map $t$ identifies with a $d$-tuple $(r_1, \dots, r_d)$ of invertible elements in $R$. With this notation, the trace of $M$ at the given $t = (r_1, \dots, r_d) \in T(R)$ is given by $\sum_{i \in \Z^d} (\rank M_i) \cdot r^i$ where $r^i := r_1^{i_1} \cdots r_d^{i_d}$. 
Hence the function $\tr_t \eE_x$ on $S$ indeed comes as a pullback along $t: S \to T$.
Since this holds for any $(t, x)$, we have proved the lemma for $T$.
\end{proof}

\begin{corollary}\label{lemma: trace functions are locally constant on fibers for G}
For a reductive group $G$ and $X \in \Sch_k^{\qcqs, G}$, the image of $\tr: K^G_0(X)_k \to H_0(\Fix^{\red}_G(X), \O)$ lands in functions which are locally constant relatively to $G$.
\end{corollary}
\begin{proof}
Let $\eE \in \Perf^G(X)$ and pick any $k$-valued point $(g, x)$ of $\Fix_G(X)$.
Then $\eE_{(g, x)}$ is equivariant for the whole commutative subgroup scheme $C_g$ of $G$ over $k$ generated by $g$. This sits in a short exact sequence between its unipotent radical and maximal torus:
\begin{equation*}
    1 \to U_g \to C_g \to T_g \to 1
\end{equation*}
The trace does not depend on the unipotent part of $g$, while the torus case was handled above. The lemma follows.
\end{proof}

\subsection{Projective bundle formulas}
We now recall the projective bundle formula and its more elaborate variants. This is essential for controlling algebraic $K$-theory. All of these results ultimately rely on suitable semi-orthogonal decompositions.

\paragraph{Projective bundle formula and consequences.}
Since both $K$-theory and Hochschild homology are localizing invariants, they automatically satisfy the equivariant projective bundle formula going back to \cite{TT90, Tho88}.

\begin{proposition}\label{proposition: projective bundle formula for localizing invariants}
Let $G$ act on $X$. Let $\eE \in \VB^G(X)$ be a $G$-equivariant vector bundle of rank $r$ and $\P(\eE)$ its projectivization. Let $\E(-)$ be a localizing invariant. Then $\E^G(-)$ satisfies the projective bundle formula -- there is a natural equivalence
\begin{equation}\label{equation: projective bundle formula}
\E^G(\P(\eE)) \simeq \bigoplus_{i=0}^{r-1} \E^G(X)
\end{equation}
compatible with natural maps between localizing invariants.
\end{proposition}
\begin{proof}
In the view of Recollection \ref{recollection: semiorthogonal decompositions and localizing invariants}, localizing invariants satisfy the projective bundle formula by the semi-orthogonal decomposition \cite[Theorem 3.3 and Corollary 3.6]{Kha18}.    
\end{proof}

One consequence is the well-known splitting principle.
\begin{lemma}\label{lemma: equivariant splitting principle}
Let $c \in K^G_0(X)$ be the class of a $G$-equivariant vector bundle $\eE$ on $X$ of rank $n$. Then there exists a $G$-equivariant morphism $f: Y \to X$ such that $f^*: K^G_0(X) \to K^G_0(Y)$ is injective and $f^*[\E] = \sum_i^n [\eL_i]$ for some $G$-equivariant line bundles $\eL_i$ on $Y$. Similarly for $K^G_0(-, k)$.
\end{lemma}    
\begin{proof}
Such $Y$ can be constructed by induction on $n$. Let $Y_0:= X$, $\E_0 := \E$. Consider $Y_1 = \P(\E_0)$ with its canonical map $f_1: Y_1 \to Y_0$. The tautological line bundle $\eL_1$ on $Y_1$ is a $G$-equivariant subbundle of $f_1^*\E_0$. The quotient $\E_1 := f^*\E_0 / \eL_1$ is a rank $(n-1)$ bundle on $Y_1$ and we have $f_1^*[\E] = [\eL_1] + [\E_1]$. Moreover, the pullback $f^*: K^G_0(Y_0)_k \to K^G_0(Y_1)_k$ is injective by the projective bundle theorem. We can continue by induction.
\end{proof}

The following standard trick extends the projective bundle formula to flag varieties (alternatively, one can deduce this from a suitable semi-orthogonal decomposition).
\begin{discussion}\label{discussion: partial flag variety bundles}
Take $d \in \mathbb{N}$, let $\mu$ be a dominant cocharacter of the diagonal torus $T \leq \GL_d$ and consider the corresponding Levi subgroup $L_{\mu}$ with block sizes $(d_1, \dots, d_l)$. 

Let $X \in \Sch_k^{\qcqs, G}$ and $\eE \in \VB^G(X)$. Let $Y = \Flag_{X}(\eE, \mu)$ be the associated partial flag variety bundle of type $\mu$ over $X$, whose fibers parametrize flags of vector bundle quotients of $\eE$ with successive graded pieces of dimensions $(d_1, \dots, d_{l-1})$ -- for example, if $\mu = \omega_j$ is the $j$-th fundamental coweight, then $\Flag_X(\eE, \omega_j) = \Grass_X(\eE, j)$ is the relative Grassmannian of rank $j$ quotients of $\eE$.
Take $W = \Flag_X(\eE)$ to be the corresponding full flag variety bundle. Then we have natural maps
\begin{equation}\label{equation: partial flag variety bundles}
    \begin{tikzcd}
        W \arrow[bend left=20, rr, "h"] \arrow[swap]{r}{g} & Y \arrow[swap]{r}{f} & X
    \end{tikzcd}
\end{equation}
and both $g$ and $h$ factor as towers of equivariant projective bundles. 
\end{discussion}

\begin{corollary}\label{corollary: classical trace and projective bundles}
Let $X \in \Sch_k^{\qcqs, G}$ and $\eE \in \VB^G(X)$. Let $Y = \Flag_{X}(\eE, \mu)$ be the associated partial flag variety bundle. Let $E(-)$, $E'(-)$ be localizing invariants and $\tau: E(-) \to E'(-)$ a natural map between them. Pick $i \in \Z$.  Then the following are equivalent:
\begin{itemize}
    \item[(i)] $\tau_i: E^G_i(X) \to E'^G_i(X)$ is an isomorphism,
    \item[(ii)] $\tau_i: E^G_i(Y) \to E'^G_i(Y)$ is an isomorphism.
\end{itemize}
\end{corollary}
\begin{proof}
For projective bundles, this follows immediately from  \eqref{equation: projective bundle formula}.  Applying the case of projective bundles multiple times for the bundles appearing in the factorizations of $g$ and $h$ from \eqref{equation: partial flag variety bundles}, we deduce the result for $f$.
\end{proof}

\begin{remark}
We will later apply Corollary \ref{corollary: classical trace and projective bundles} to
\begin{itemize}
    \item the trace map $\tr: K^G(-, k) \to HH^G(-, k)$,
    \item the natural map $K^G(-) \to KH^G(-)$.
\end{itemize}
To phrase the first case in words: the trace map $\tr: K^G(-, k) \to HH^G(-,k)$ is an isomorphism for $X$ if and only if it is an isomorphism for $Y$.
\end{remark}

\paragraph{Stratified Grassmannian bundles.}
We will employ a similar statement for {\it stratified} projective bundles. To make this work, we need to use their derived enhancements studied in \cite{Jia22a, Jia22b, Jia23}.\footnote{We thank Andres Fernandez Herrero for pointing out the relevance of these results to us.}
\begin{recollection}\label{recollection: derived grassmannian bundles}
Let $\eE \in \Perf^G(X)$ be a perfect complex supported in homological degrees $\geq 0$. Denote by $\eH_0(\eE) \in \Qcoh^G(X)$ its zero homology group. Consider the relative derived Grassmannian of rank $j$ of $\eE$ on $X$ in the sense of \cite[Definition 4.3]{Jia22b}, denoted
\begin{equation}
   Y = \Grass_{X}(\eE, j) \xrightarrow{f} X.
\end{equation}
This is a $G$-equivariant derived scheme. By \cite[Proposition 4.7]{Jia22b}, its classical truncation is the usual relative Grassmannian
\begin{equation}
   Y^{\cl} = \Grass_{X}(\eH_0(\eE), j) \xrightarrow{f^{\cl}} X.
\end{equation}

Assume further that $\eE$ has Tor-amplitude $[1, 0]$ and $j=1$. The category $\Perf^G(Y)$ then has an explicit semi-orthogonal decomposition in terms of a flattening stratification for $\eE$ by \cite[Theorem 3.2 and Remark 1.1]{Jia23}. In particular, $\Perf^G(X)$ is a semi-orthogonal summand of $\Perf^G(Y)$ via $f^*$.
\end{recollection}

In fact, we only need the following consequence. 
\begin{proposition}\label{proposition: K-theory and derived grassmannians}
Let $\eE \in \Perf^G(X)$ of Tor-amplitude $[1, 0]$. Let $Y = \Grass_{X}(\eE, 1) \to X$ be the associated relative derived Grassmannian of rank $1$. Let $\E(-)$ be a localizing invariant.
Then $\E^G(X)$ is a direct summand of $\E^G(Y)$, compatibly with natural maps of localizing invariants.
\end{proposition}
\begin{proof}
Follows from Recollection \ref{recollection: derived grassmannian bundles} and Recollection \ref{recollection: semiorthogonal decompositions and localizing invariants}.    
\end{proof}

In particular, $K^G(X, k) \xrightarrow{\tr_X} HH^G(X, k)$ is a direct summand of $K^G(Y, k) \xrightarrow{\tr_Y} HH^G(Y, k)$. Similarly for the comparison map from $K^G$ to $KH^G$.

\subsection{First examples}\label{section: first examples}

The trace map is an isomorphism in degree zero for interesting varieties arising in geometric representation theory. Let us give the most basic examples of this phenomenon: first for the point and then for partial flag varieties. 
An analogous trace map for equivariant cohomology in these examples (and beyond) was studied in \cite{HR23}, which was an important motivation for us.

\paragraph{Equivariant point.}
Let $G$ be a connected split reductive group of rank $n$ over a field $k$ and $T$ its maximal torus. Let $W$ be the Weyl group. We start by noting how the above works out for the equivariant point.

\begin{example}[Torus-equivariant point]\label{example: T-equivariant point}
Let $X=\pt$ with the trivial $T$-action. Then the $k$-linearized trace map gives an isomorphism
\begin{equation*}
  \tr_X:  K^T_0(\pt, k) \xrightarrow{\cong} H_0(\Fix_{\frac{T}{T}}(\pt), \O).
\end{equation*}
Both sides are given by the localized polynomial ring $k[t_1^{\pm 1}, \dots, t_n^{\pm 1}]$.    
\end{example}
\begin{proof}
Follows from the identification of the left-hand side with the $k$-linearized representation ring $R(T, k)$ and the right-hand side with $H_0(\frac{T}{T}, \O) = H_0(T, \O)$. The trace map is then given by taking characters of representations, giving the desired isomorphism.    
\end{proof}

\begin{example}[Equivariant point]\label{example: trace map in k-theory for point}
Let $X=\pt$ with the trivial $G$-action. Then the $k$-linearized trace map gives an isomorphism
\begin{equation*}
  \tr_X:  K^G_0(\pt, k) \xrightarrow{\cong} H_0(\Fix_{\frac{G}{G}}(\pt), \O).
\end{equation*}
Both sides are given by the localized polynomial ring $k[t_1^{\pm 1}, \dots, t_n^{\pm 1}]^W$, whose spectrum naturally identifies with $\s = \Tfrac{G}{G}$.
\end{example}
\begin{proof}
Under the identification of the left-hand side with the $k$-linearized representation ring $R(G, k)$ of $G$ and the right-hand side with the ring of conjugacy-invariant functions on $G$, this is the classical isomorphism \cite[\S 3.2]{Hum95}, \cite[\S 22.a--22.b]{Mil17} valid in any characteristic.
\end{proof}

\paragraph{Partial flag varieties.}
We illustrate the above notions for partial flag varieties. 
For the sake of clarity, we stick to the case $k=\C$ and $G= GL_n$. Given a dominant coweight $\mu$ of the diagonal torus $T$ of $G$, we write $P = P_{\mu}$ for the associated parabolic subgroup of $G$ and $W_{\mu}$ the corresponding Weyl group. Let $X=G/P$ be the quotient of $G$ by $P$ equipped with the left translation by $G$. Recall from Example \ref{example: fixed-point schemes of partial flag varieties} that $\Fix_G(X)$ amounts to the multiplicative Grothendieck--Springer resolution.

We now show that the degree zero trace map is an isomorphism in this case, realizing (reduced) rings of invariant functions on multiplicative Grothendieck--Springer resolutions via equivariant algebraic $K$-theory. In fact, the derived and infinitesimal structures of the fixed-point scheme are irrelevant.
\begin{example}[Partial flag variety]\label{example: partial flag variety}
Take $k = \C$ and consider $X=G/P$ with $G$ acting by left translations. Then we get isomorphisms
\begin{equation*}\label{equation: trace map for partial flag varieties}
  \tr_X:  K^G_0(X, k) \xrightarrow{\cong} H_0(\Fix^{\lL}_{\frac{G}{G}}(X), \O) \xrightarrow{\cong} H_0(\Fix_{\frac{G}{G}}(X), \O)  \xrightarrow{\cong} H_0(\Fix^{\red}_{\frac{G}{G}}(X), \O) \xrightarrow{\cong} H_0(\Fix^{\red}_{\s}(X), \O)
\end{equation*}
and these are explicitly given by the polynomial ring $k[t_1^{\pm 1}, \dots, t_n^{\pm 1}]^{W_{\mu}}$.
\end{example}

\begin{proof} 
Since the stacks in question are flawless, we can interpret $H_0(\Fix^{\lL}_{\frac{G}{G}}(X), \O)$ as $\C$-linear Hochschild homology. The trace map is an isomorphism in degree $0$ for the point by Example \ref{example: trace map in k-theory for point}, so the first map is an isomorphism by the projective bundle formula from Proposition \ref{proposition: projective bundle formula for localizing invariants}.

Further note that 
\begin{equation*}
    \Fix^{\lL}_{\frac{G}{G}}(X) = \dL(G \setminus (G/P)) = \dL(\pt/P) = \tfrac{P}{P}
\end{equation*}
which is already classical and reduced. Hence $\Fix^{\lL}_\frac{G}{G}(X) = \Fix_\frac{G}{G}(X) = \Fix^{\red}_\frac{G}{G}(X)$ giving the next two isomorphisms. See also \cite[\S 3.1]{BN13} for realted discussion; the classicality can be alternatively deduced through the criterion \cite[Proposition 3.3]{BN13}.

Now note that
\begin{equation*}
    K^G_0(G/P) = K_0(G\backslash G / P) = K_0(\pt / P) = R(P) = R(L). 
\end{equation*}
is the representation ring of the Levi quotient $L = L_{\mu}$ of $P$. This indeed matches $\Z[t_1^{\pm 1}, \dots, t_n^{\pm 1}]^{W_{\mu}}$. Thus $K^G_0(X, k) =  k[t_1^{\pm 1}, \dots, t_n^{\pm 1}]^{W_{\mu}}$, giving the desired presentation.
In particular, this ring is reduced and normal.

For the final term, we argue by classical algebraic geometry. Over a geometric point $s \in \s$, the reduced fiber of $K^G_0(X, k)$ consists of finitely many points labelled by the set $\Eig^{\mu}_s$ of $\mu$-partitions of the multiset of roots of $s$ (regarded as a symmetric polynomial).
Similarly, the reduced fiber of $\Fix_{\s}(X)$ over $s$ consists of finitely many points labelled by the same set $\Eig^{\mu}_s$, as these are the fixed points of (any) regular lift $g_s \in G$ of $s$.
Now, $\tr^{\red}_{\s}: \Fix^{\red}_{\s}(X) \to \Spec K^G_0(X)$ is a well-defined map of (affine) schemes over $\s$. It induces a bijection on geometric points, as this it the case fiberwise over $\s$. Since both sides are reduced and the target is normal (it is the spectrum of a localized polynomial ring), we deduce that it is an isomorphism. Hence the whole composition in the statement is an isomorphism, finishing the proof.

The isomorphism between the last two terms can be alternatively deduced as follows. Note $\s$ intersects all regular conjugacy classes in $G$ and $G \setminus G^{\reg}$ has codimension at least $2$. We deduce that the restriction map $H_0(\Fix^{\red}_{\frac{G}{G}}(X), \O) \to H_0(\Fix^{\red}_{\s}(X), \O)$ is injective. It is also surjective by the above, hence an isomorphism. (In particular, one can proceed this way to prove everything without appeal to semiorthogonal decompositions.)
\end{proof}

\subsection{Equivariant \texorpdfstring{$KH$}{KH} and proper excision}
We now recall the $\A^1$-homotopy invariant version of algebraic $K$-theory $KH$ in the equivariant setup. This in particular satisfies {\it proper excision}, which is one of our main computational tools later.

Let us recall from Notation \ref{notation: KH} that $KH^G(X)$ is defined via the simplicial spectrum $K^G(\Delta^{\bullet} \times X)$. Note that $KH$ is a localizing invariant on $\Cat_k^{\ex}$ by Recollection \ref{recollection: localizing invariants} and we have a natural map $K(-) \to KH(-)$.

Let $T$ be a torus over $k$. Then $KH^T(-)$ agrees with other commonly used definitions by \cite[Theorem 3.1.(4)]{Hoy21}. Moreover, $KH^T(-)$ is $\A^1$-homotopy invariant and the natural map $K^T(-) \to KH^T(-)$ is an isomorphism on smooth schemes \cite[Theorem 1.3]{Hoy21}. Furthermore, $KH^T(-)$ is indifferent to derived structures \cite{LT19}, \cite[Theorem F]{KR22}.
Finally $KH^T(-)$ satisfies proper excision on equivariant abstract blowups in any characteristic, which we now recall.
\begin{recollection}
A pullback square in $\Sch^{\qcqs, T}_k$ 
\begin{equation}\label{equation: equivariant abstract blowup square}
    \begin{tikzcd}
        Y \arrow[d, "f"] \arrow[r, hookleftarrow] & E \arrow[d] \\
        X \arrow[r, hookleftarrow, "i"] & Z
    \end{tikzcd}
\end{equation}
is called an {\it equivariant abstract blowup square} if $f$ is proper and $i$ is a finitely presented closed immersion such that $f$ is an isomorphism over the open complement $X \setminus Z$.   
\end{recollection}

\begin{proposition}\label{proposition: fibered square of KH_G-theory spectra}
Applying $KH^T(-)$ to \eqref{equation: equivariant abstract blowup square} gives a homotopy fiber square
\begin{equation*}\label{equation: fibered square of KH_G-theory spectra}
    \begin{tikzcd}
        KH^{T}(Y) \arrow[r] & KH^{T}(E) \\
        KH^{T}(X) \arrow[u] \arrow[r] & KH^{T}(Z) \arrow[u]
    \end{tikzcd}
\end{equation*}
\end{proposition}
\begin{proof}
This holds by \cite[Theorem G]{KR22}.
\end{proof}

\begin{remark}\label{remark: proper excision}
Given an invariant $F$ with values in spectra, we say that $F$ satisfies {\it (equivariant) proper excision} if it sends any (equivariant) abstract blowup square \eqref{equation: equivariant abstract blowup square} to a homotopy fiber square. If this is the case, \eqref{equation: equivariant abstract blowup square} in particular induces the long exact sequence on homotopy groups of Mayer-Vietoris type
\begin{equation*}
        \dots \to F_1(E) \to F_0(X) \to F_0(Y) \oplus F_0(Z) \to F_0(E) \to F_{-1}(X) \to \dots
\end{equation*}
(and similarly in the equivariant situation).
\end{remark}

%% file: 2_perfect_schemes.tex
\section{Perfect schemes}\label{section: perfect schemes}
In the rest of the paper, we restrict our attention to perfect schemes in characteristic $p$, which we presently recall. See \cite[\S 3]{BS16}, \cite[Appendix A]{Zhu14} for a detailed introduction. 

\paragraph{Perfect schemes in characteristic $p$.}
Let $k$ be a perfect field; we take $k = \F_p$.
A scheme $X \in \Sch_k$ is called perfect, if the Frobenius morphism $\varphi: X \to X$ is an isomorphism. The full subcategory of perfect schemes is denoted $\Sch^{\perf}_k \subseteq \Sch_k$.
Given a scheme $X \in \Sch_k$, its perfection $X_{\perf}\in \Sch^{\perf}_k$ is defined as the $\N$-indexed inverse limit along $\varphi$
\begin{equation*}
X_{\perf} := \underset{\varphi}{\lim} \ X.  
\end{equation*}
Note that $\varphi$ is an affine map; in particular $\lim_{\varphi} X$ exists as a scheme \cite[Tag 01YV]{Sta}. On affines, the perfection is simply given by iteratively adjoining $p$-th roots of all functions. This gives a functor
\begin{align*}
 (-)_\perf: \Sch_k & \to \Sch^{\perf}_k \\
  X & \mapsto X_{\perf}.
\end{align*}
More generally, the same formula defines the perfection $\eX_{\perf} := \lim_{\varphi} \eX$ of any derived stack $\eX$ over $k$.

\begin{setup}\label{setup: perfect geometry}
We denote $\Sch_k^{\fp} \subseteq \Sch_k$ the full subcategory of finitely presented $k$-schemes. We denote its objects by a subscript ``0" as $X_0$, $Y_0$, $f_0: Y_0 \to X_0$ and so on. 

We work with the category $\Sch_k^{\pfp}$ of perfectly finitely presented perfect $k$-schemes (and perfectly finitely presented morphisms) -- see \cite[\S 3]{BS16}: 
we only consider objects and morphisms coming as perfections of objects and morphism from $\Sch_k^{\fp}$. We denote such perfection by dropping the subscript: unless specified otherwise, $X$ stands for the perfection of $X_0$, $f$ stands for the perfection of $f_0$, and so on. 
\end{setup}

\paragraph{Basic properties.}
We recall some facts from \cite[\S 3]{BS16} about perfect schemes.
\begin{itemize}
    \item[(i)] \hypertarget{property: adjoint}{the} functor $(-)_{\perf}: \Sch_k \to \Sch^{\perf}_k$ is right adjoint to the inclusion $\Sch^{\perf}_k \subseteq \Sch_k$,
    \item[(ii)] \hypertarget{property: limits}{the} functor $(-)_{\perf}$ commutes with arbitrary limits, in particular with fiber products,
\end{itemize}
Indeed, the first line is clear from the definition; the second line follows formally.
Many standard properties of schemes are preserved under taking perfection -- see \cite[Lemma 3.4]{BS16} for a list.

\paragraph{Functions and coherent sheaves.}
\begin{itemize}
    \item[(iii)] \hypertarget{property: reduced}{perfect} schemes are reduced,
    \item[(iv)] \hypertarget{property: functions}{global} functions on perfect schemes satisfy proper excision, i.e. taking $\RG(-, \O)$ of a perfection of an abstract blowup square induces long exact sequence on cohomology groups.
    \item[(v)] \hypertarget{property: sections}{for} a coherent sheaf $\eF \in \Qcoh(X_0)$ it holds that $H^i(X, \eF) = \colim_{\varphi^*} H^i(X_0, \varphi^{*, (i)}(\eF))$,
    \item[(vi)] \hypertarget{property: classical}{all} perfect derived schemes are classical.
\end{itemize}

For \hyperlink{property: reduced}{(iii)} note that all non-reducedness disappears in the colimit along Frobenius on functions. Part \hyperlink{property: functions}{(iv)} is \cite[Lemma 3.9]{BST13}, \cite[Lemma 4.6]{BS16}. Part \hyperlink{property: sections}{(v)} is a standard fact about the compatibility of coherent cohomology with inverse limits along affine morphisms (by affineness of bonding maps this follows from \cite[Tag 073D or Tag 0GQU]{Sta}). Part \hyperlink{property: classical}{(vi)} is \cite[Corollary 11.9]{BS16}.

\paragraph{Rationality and perfection.}
Rational singularities are preserved under perfection.
\begin{definition}
A map of schemes $f: Y \to X$ is called {\it rational} if $Rf_* \O_Y = \O_X$.    
\end{definition}

\begin{observation}\label{lemma: perfection preserves rationality}
    Let $f_0: Y_0 \to X_0$ be a rational map in $\Sch_k^{\fp}$. Then its perfection is rational as well.
\end{observation}
\begin{proof}
Working Zariski-locally on $X_0$, the rationality of $f_0$ amounts to the following:
For each open affine $U_0 \subseteq X_0$, $\RG((Y_U)_0, \O_{(Y_U)_0}) = \O_{U_0}$.
Equivalently, this reads as
\begin{align*}
 \RG^i((Y_U)_0, \O_{(Y_U)_0}) =    
\begin{cases*}
 0  & \text{if } $i \neq 0$, \\
 \O_{U_0}  & \text{if } $i = 0$.
\end{cases*}
\end{align*}

We need to prove that for each open affine $U \subseteq X$, $\RG(Y_U, \O_{Y_U}) = \O_U$. This can be checked on each cohomology group separately, where it follows from \hyperlink{property: sections}{(v)}:
\begin{align*}
 \RG^i(Y_U, \O_{Y_U}) = \colim_{\varphi^*} \RG^i((Y_U)_0, \O_{(Y_U)_0}) =   
\begin{cases*}
 0  & \text{if } $i \neq 0$, \\
 \O_U  & \text{if } $i = 0$.
\end{cases*}
\end{align*}
\end{proof}

\paragraph{Topological invariance.}
\begin{itemize}
    \item[(vii)] \hypertarget{property: perfection univ homeo}{the} map $X_{\perf} \to X$ is a universal homeomorphism,
    \item[(viii)] \hypertarget{property: univ homeo}{a} universal homeomorphism of perfect schemes is an isomorphism, 
    \item[(ix)] \hypertarget{property: awn}{the} functor $(-)_{\perf}$ agrees with absolute weak normalization.
\end{itemize}
For \hyperlink{property: perfection univ homeo}{(vii)} and \hyperlink{property: univ homeo}{(viii)}, see \cite[Theorem 3.7, Lemma 3.8]{BS16}.
It is straightforward to deduce \hyperlink{property: awn}{(ix)} -- we do this now for completeness.

\begin{observation}
Let $Y_0 \in \Sch_k$. Then its perfection $Y$ is the absolute weak normalization of $Y_0$.    
\end{observation}
\begin{proof}
We recall the notion of absolute weak normalization \cite[Tag 0EUK]{Sta} (also see \cite{GT80, LV81, Man80}). The absolute weak normalization $Y^{\awn}_0 \to Y_0$ of $Y_0$ can be characterized as the initial scheme equipped with a universal homeomorphism to $Y_0$ \cite[Tag 0EUS]{Sta}. 
Let $f_0: X_0 \to Y_0$ be the absolute weak normalization of $Y_0$. Taking perfections, we obtain the commutative square 
\begin{center}
    \begin{tikzcd}
        X \arrow[r, "f"] \arrow[d] & Y \arrow[d] \arrow[dl, dashed]\\
        X_0 \arrow[r, "f_0"] & Y_0
    \end{tikzcd}
\end{center}
Now $f_0$ is universal homeomorphism by design, so $f$ is an isomorphism by \hyperlink{property: univ homeo}{(viii)}. Thus we get the factorization $Y \to X_0 \to Y_0$ as indicated by the dashed arrow. Since $Y$ is perfect, it is reduced by \hyperlink{property: reduced}{(iii)}, so $Y \to X_0$ is an isomorphism by \cite[Tag 0H3H, (3)]{Sta}.
\end{proof}

\paragraph{Perfection of global quotient stacks.}
Global quotient stacks are compatible with perfections in the following manner.
\begin{proposition}\label{proposition: perfection of global quotient stacks}
Let $G$ be a smooth algebraic group acting on a scheme $X$ over $k$.
Then
\begin{equation*}
    \left( X/G \right)_{\perf} = X_\perf / G_\perf
\end{equation*}
\end{proposition}
\begin{proof}
This is a special case of \cite[Lemma A.11]{Zhu14}.
\end{proof}

\paragraph{Perfect Stein factorization.}
We first recall the classical Stein factorization \cite[Tag 03GX]{Sta}. 
Let $f_0: X_0 \to S_0$ be a morphism. Setting $S_0' = \underline{\Spec}_{\O_{S_0}}(f_{0, *} \O_X)$, we can factor $f_0$ as follows.
\begin{center}
    \begin{tikzcd}
        X_0 \arrow[rr, "f_0'"] \arrow[dr, "f_0"] &  & S'_0 \arrow[dl, "\pi_0"] \\
        & S_0 &
    \end{tikzcd}
\end{center}

Assuming $f_0: X_0 \to S_0$ is proper and $S_0$ is locally noetherian, \cite[Tag 03H0]{Sta} shows that $f'_0$ is proper with geometrically connected fibers, $f'_{0, *}\O_{X_0} = \O_{S'_{0}}$, $\pi_0$ is finite, $S_0'$ is the normalization of $S_0$ in $X_0$.
In particular, for any geometric point $s_0 \in S_0$ it holds that $\pi_0(X_{s_0}) = \pi_0(S'_{s_0})$ as finite sets.    

Assume now $X, S \in \Sch^{\pfp}_k$ and let $f: Y \to S$ be a perfectly proper map.
Let $S' = \underline{\Spec}_{\O_S}(f_* \O_X)$.

\begin{lemma}\label{lemma: perfect stein 1}
The map $S' \to S$ is an isomorphism if and only if all geometric fibers of $f: X \to S$ are connected.   
\end{lemma}
\begin{proof}
Without loss of generality, we may fix finite type models $X_0, Y_0, f_0$ and $S'_0$. Stein factorization is compatible with base change \cite[Tag 03GY]{Sta}.

Consider the Stein factorization $X \to S' \to S$. For each geometric point $s \in S$, $S'_s$ is given by finitely many points, which are in bijection with $\pi_0(X_s)$. By \hyperlink{property: univ homeo}{(viii)}, $S' \to S$ is an isomorphism iff it is a universal homeomorphism. But the latter happens iff each $S_s = \pi_0(X_s)$ has a single point.
\end{proof}

\begin{lemma}\label{lemma: perfect stein 2}
Consider a diagram of pfp perfect schemes
\begin{center}
    \begin{tikzcd}
        X \arrow[rr, "g"] \arrow[dr, "f"] &  & T \arrow[dl, "\pi"] \\
        & S
    \end{tikzcd}
\end{center}
Assume that $X \to S$ is perfectly proper and $T \to S$ is affine.
Assume that for each geometric point $s$ of $S$, $\pi_0(g): \pi_0(X_s) \to \pi_0(T_s)$ is bijective.
Then $g: X \to T$ induces an isomorphism on global sections.
\end{lemma}
\begin{proof}
Let $S' = \underline{\Spec}_{\O_S}(f_* \O_X)$. 
As $T$ is affine over $S$, the Stein factorization $S'$ naturally fits into
$$X \to S' \to T \to S.$$
We will prove that the corresponding map $\O_T \to \O_{S'}$ of $\O_S$-algebras is an isomorphism, which yields the result after pushforward. To this end, we can without loss of generality assume that $S$ and hence $T$ is affine.

Since $X \to T$ is perfectly proper, by Lemma \ref{lemma: perfect stein 1} it is enough to check that its geometric fibers are connected. But this can be checked separately over each geometric point $s \in S$, where it boils down to our assumption.
\end{proof}

%% file: 3_perfect_K_theory.tex
\section{Perfect equivariant \texorpdfstring{$K$}{K}-theory}\label{section: perfect K-theory}
We now describe the main structural properties of equivariant $K$-theory of perfect schemes. Perfection changes $K$-theory in a controlled way by ``inverting the Frobenius" and this forces good descent properties. In fact, \cite{Kra80b, KM21, AMM22, Cou23} have already studied such questions for non-equivariant algebraic $K$-theory (see also \cite{EK18} for motivic results). We take advantage of their methods. 

Let $k = \F_p$. Unless stated otherwise, let $G_0$ be a group scheme over $k$ and $G$ its perfection. Note that $G$ is still a group scheme over $k$. It is affine if and only if $G_0$ is.

\subsection{Effect of perfection on equivariant \texorpdfstring{$K$}{K}-theory}
Perfection changes $K$-theory by inverting the Frobenius. 

\begin{lemma}\label{proposition: K-theory of inverse limit}
Let $X_0 \in \Sch^{\qcqs, G_0}_k$. Let $G$ and $X$ be the perfections. Then
\begin{equation*}
K^G(X) = \colim_{\varphi^*} K^{G_0}(X_0).    
\end{equation*}
\end{lemma}
\begin{proof}
Algebraic $K$-theory $K(-)$ sends filtered inverse limits of qcqs schemes along affine maps to colimits by \cite[Proposition 3.20]{TT90}, giving the result for the non-equivariant $K$-theory of $X = \lim X_0$. The case of equivariant $K$-theory follows by the same argument, using \cite[Theorem 3.20.1]{TT90} first and only then passing to $K^G(-)$.
\end{proof}

Taking exterior powers furnishes equivariant $K$-groups with $\lambda$-operations $\lambda_j$ for $j \in \N_0$. We equip the sum of non-negative $K$-groups $\bigoplus_{i\geq 0} K^G_i(X)$ with the ring structure coming from the usual module structure over $K^G_0(X)$ and zero multiplication on the positive part. This is a standard definition in this context, making $\bigoplus_{i\geq 0} K^G_i(X)$ into a $\lambda$-ring by \cite[\S 2]{Koc98} and \cite[\S 5]{KZ21}. We denote its Adams operations by $\psi^j$ for $j \in \N_0$, see \cite[p. 102]{Weib13}. We then formally extend these Adams operations to negative $K$-groups by Bass delooping \cite[Theorem 6.6]{TT90}.

The following statement is well-known for $K$-theory of rings by \cite{Kra80a, Kra80b, Hil81} or \cite[Lemma 3.1.5]{Cou23}; the case of qcqs schemes follows by Zariski descent. We provide the arguments for the equivariant extension.\footnote{We thank Bernhard Köck for a detailed discussion of these arguments.}

\begin{lemma}\label{lemma: perfect k-theory and adams operations}
For each $i \in \Z$, the action of Frobenius pullback $\varphi^*$ on $K^{G_0}_i(X_0)$ agrees with the $p$-th Adams operation $\psi^p$. 
\end{lemma}

Before starting the proof, we recall the formalism of polynomial functors (over $k$) from \cite[\S 8]{HKT17}, we refer to their nice exposition for details. The category $\Vect(k)$ is enriched in schemes by regarding each $\Hom_{k}(V, W)$ for $V, W \in \Vect(k)$ as the affine space on the underlying vector space. A {\it polynomial functor over $k$} is an enriched endofunctor of $\Vect(k)$. In other words, it is an functor $F: \Vect(k) \to \Vect(k)$ such that each of the maps $\Hom_k(V, W) \to \Hom_k(F(V), F(W))$ on morphism spaces is given by a tuple of polynomials \cite[Definition 8.1]{HKT17}. Polynomial functors form an exact category $\Pol(k)$, equipped with tensor product and exterior power operations. There is a notion of polynomial functors of {\it homogeneous degree $d$} and we write $\Pol(k)_d$ for the corresponding full subcategory; we further write $\Pol(k)_{< \infty}$ for the full subcategory spanned by finite direct sums of homogeneous polynomial functors. Finally, $\Pol(k)^0$ stands for the full subcategory of polynomial functors sending the zero vector space to itself (and similarly for the variants).

In particular, the Frobenius pullback $\varphi^*$ as well as the exterior powers $\Lambda^j$ for $j \in \N_0$ can be regarded as objects of $\Pol(k)$. This can be done by hand from the definition: $\varphi^*$ is the identity on objects, but on morphisms it raises the polynomial coordinates to $p$-th powers (it is the actual Frobenius when we view the morphism spaces as schemes). For $\Lambda^j$, see \cite[Example 8.2]{HKT17}. Furthermore, $\varphi^*$ has degree $p$, while $\Lambda^j$ has degree $j$ and all of these preserve the zero object.

The Grothendieck group $K_0(\Pol(k)_{<\infty})$ is naturally a $\lambda$-ring, graded by the homogeneous degree. It in particular contains the elements $\varphi^*$ and $\lambda_j$, $j \in \N$ corresponding to the Frobenius and exterior powers. It is explicitly given as the free $\lambda$-ring on one variable (also known as the ring of symmetric functions)
\begin{equation*}
K_0(\Pol(k)_{<\infty}) \cong \Z[s_j \mid j\in \N], \qquad \lambda_j \mapsto s_j
\end{equation*}
by \cite[Theorem 8.10]{HKT17}. It is a free polynomial ring on the generators $s_j$ with $\deg s_j = j$, while $K_0(\Pol(k)^0_{<\infty})$ is its augmentation ideal. The degree $d$ part $K_0(\Pol(k)_d) \cong \Z[s_j \mid j\in \N]_d$ can be naturally described as the representation ring of the Schur algebra $\Gamma^d(\Mat(n, k))$ as in \cite[\S 8.B, Theorem 8.8]{HKT17}. This latter representation ring receives a natural surjection from $R_{k}(\GL_n)$ once $n \geq d$ so that $\lambda_j$ comes from the class of the $j$-th fundamental representation of $GL_n$, see \cite[Remark 8.17]{HKT17}, \cite[\S 3.8]{Ser68}.

\begin{proof}[Proof of Lemma \ref{lemma: perfect k-theory and adams operations}]
Consider the $\lambda$-ring $K_0(\Pol(k)^0_{<\infty})$.
The Adams operation $\psi^j$ for $j \in \N$ is by definition a linear combination of degree $j$ monomials in $\lambda$'s, so the elements $\varphi^*$ and $\psi^p$ lie in the graded degree $p$. They are equal: indeed, it is enough to check this on the level of operations on the representation ring $R_{k}(\GL_n)$, which naturally surjects onto $\Z[s_j \mid {j\in \N}]_d$ once $n \geq d$. On $R_{k}(\GL_n)$, the desired equality $\varphi^* = \psi^p$ holds by the classical results \cite[Corollary 3.7]{Kra80a}, also see \cite[proof of Proposition 5.4]{Kra80b}, \cite[Theorem 5.1]{Hil81}.

We conclude via arguments from \cite{KZ21, HKT17}, who develop a method for reducing statements about $\lambda$-operations on higher equivariant $K$-groups to computations in the $\lambda$-ring $K_0(\Pol(k)^0_{<\infty})$. In fact, we can completely follow the reasoning from \cite[proof of Theorem 5.1]{KZ21}, which employs this method (to prove a different equality).

In more detail, to prove an equality of operations on $K^G_i(X)$ for any given $i \geq 0$, it is enough to prove it inside
$K_0(\End((B^q)^i \Vect^G(X))),$
the zeroth $K$-group of the exact category of endofunctors of a suitable exact category $(B^q)^i \Vect^G(X)$ of $i$-dimensional bounded acyclic binary multicomplexes, defined from $\Vect^G(X)$ inductively on $i$ -- see \cite[Definition 1.1]{HKT17} for the notation. This reduction step is explained in \cite[proof of Theorem 5.1]{KZ21} or \cite[proof of Theorem 8.18]{HKT17} (where the notation $\mathcal{P}(G, X) = \Vect^G(X)$ is used): there is a natural surjection $K_0((B^q)^i \Vect^G(X)) \twoheadrightarrow K^G_i(X)$; testing on classes of objects in the left-hand side we reduce to checking the equality inside $K_0(\End((B^q)^i \Vect^G(X)))$.

Furthermore, by \cite[(5.3)]{KZ21}, there is an exact functor
\begin{equation*}
\Pol^0_{<\infty}(k) \to \End((B^q)^i \Vect^G(X)).
\end{equation*}
Consider the induced homomorphism on $K_0$. Both $\varphi^*$ and $\psi^p$ exist and are equal as elements in $K_0(\Pol_{<0}(k))$ by above; they realize the desired operations under this map. Therefore, the same holds in $K_0$ of the right-hand side and we conclude that $\varphi^* = \psi^p$ on all higher equivariant $K$-groups by the previous paragraph. The case of negative $K$-groups then follows by Bass delooping (by our definition of $\psi^p$).   
\end{proof}

In other words, equivariant $K$-theory of $X$ is obtained from $K^{G_0}(X_0)$ by localizing $\psi^p$:
\begin{equation*}
    K^{G}(X) = \colim_{\psi^p} K^{G_0}(X_0).
\end{equation*}

The following is a special case of Lemma \ref{lemma: perfect k-theory and adams operations} in degree zero after $k$-linearization.

\begin{lemma}[Commuting with perfections]\label{proposition: k-theory commutes with perfections}
Given $X_0 \in \Sch_k^{\qcqs, G_0}$, we have
\begin{equation*}
   K^G_0(X)_k = \left( K^{G_0}_0(X_0)_k \right)_{\perf}
\end{equation*}
\end{lemma}
\begin{proof}
The left adjoint $(-)_k = (-\underset{\Z}{\otimes} \F_p)$ commutes with colimits, so by Lemma \ref{proposition: K-theory of inverse limit} we can rewrite the left-hand side as
\begin{equation*}
    \colim_{\varphi^*} \left( K^{G_0}_0(X_0)_k \right).
\end{equation*}
We now claim that the maps in this colimit are given by taking $p$-th powers. If this is the case, the colimit becomes $\left( K^{G_0}_0(X_0)_k \right)_{\perf}$ by definition. It thus suffices to prove Claim \ref{claim: frobenius and p-the powers}.
\end{proof}

\begin{claim}\label{claim: frobenius and p-the powers}
The Frobenius pullback $\varphi^*$ on $K^{G_0}_0(X_0)_k$ is given by taking $p$-th powers $c \mapsto c^p$.
\end{claim}
\begin{proof}[Proof of claim.]
Let $c$ be the class of a $G_0$-equivariant vector bundle $\eE$ of rank $n$ on $X_0$. Then we use Lemma \ref{lemma: equivariant splitting principle}: there exists a $G_0$-equivariant morphism $f: Y_0 \to X_0$ such that $f^*: K^{G_0}_0(X_0)_k \to K^{G_0}_0(Y_0)_k$ is injective and $f^*[\eE] = \sum_i^n [\eL_i]$ for some $G_0$-equivariant line bundles $\eL_i$. To prove $\varphi^*[\eE] = [\eE]^p$, we can thus wlog assume that $[\eE]$ is a sum of classes of equivariant line bundles. 

First assume that $c$ is a class of a $G_0$-equivariant line bundle $\eL$. Then $\varphi^* \eL \cong \eL^{\otimes p}$ with its canonical $G_0$-equivariant structure. Indeed, this can be seen explicitly on the level of transition functions; see \cite[Lemma 3.5]{BS16} for the non-equivariant statement.
Since we have base changed to $k =\F_p$, the same holds true for sums of classes of line bundles. Indeed,
\begin{equation*}
    \varphi^*: (c_1 + \dots + c_n) \mapsto (c_1^p + \dots + c_n^p) = (c_1 + \dots + c_n)^p.
\end{equation*}
\end{proof}

\subsection{Kratzer's \texorpdfstring{$p$}{p}-divisibility of perfect higher \texorpdfstring{$K$}{K}-theory}
The interpretation of the Frobenius pullback on $K$-theory as the $p$-th Adams operation $\psi^p$ implies $p$-divisibility of higher $K$-theory of perfect schemes \cite{Kra80b}. This is an important structural result, which immediately generalizes to equivariant $K$-theory.

\begin{lemma}\label{lemma: kratzers argument}\label{lemma: kartzers lemma}
Let $X \in \Sch^{\pfp, G}_k$. Then for all $i \geq 1$, $K^G_i(X)$ is a $\Z[\tfrac{1}{p}]$-module.
\end{lemma}
\begin{proof}
This is a direct generalization of the non-equivariant statement \cite[Corollary 5.5 and Example (1) below it]{Kra80b}; see also \cite[\S II.4 and \S IV.5]{Weib13}. We review Kratzer's proof.   

Since the multiplication on $\bigoplus_{i \geq 0} K^G_i(X)$ compatible with the $\lambda$-ring structure is trivial in positive degrees, Lemma \ref{lemma: perfect k-theory and adams operations} and \cite[Proposition 5.3.(i)]{Kra80b} show that on $K^G_i(X)$ with $i \geq 1$ we have
\begin{equation}\label{equation: Frobenius pullback and lambda operation}
\varphi^* = \psi^p = (-1)^{p-1} \cdot p \cdot \lambda^p.    
\end{equation}
Since $G \backslash X$ is perfect, its Frobenius $\varphi$ is an isomorphism, hence $\varphi^*$ is an isomorphism. By \eqref{equation: Frobenius pullback and lambda operation} and additivity of $\lambda^p$, the multiplication by $p$ needs to be isomorphism as well. Hence $K^G_i(X)$ is a $\Z[\tfrac{1}{p}]$-module for all $i \geq 1$.
\end{proof}

Lemma \ref{lemma: kratzers argument} in particular shows the following vanishing after $k$-linearization.
\begin{corollary}\label{corollary: kratzers vanishing with k-coefficients}
Let $X \in \Sch^{\pfp, G}_k$, then
\begin{itemize}
    \item $K^G_i(X, k) = 0$ for $i \geq 2$,
    \item $K^G_1(X, k) = K^G_0(X) \otimes^1_{\Z} k$.
\end{itemize}
\end{corollary}
\begin{proof}
For all $i \geq 1$, we know that $K^G_i(X)$ is a $\Z[\tfrac{1}{p}]$-module by Lemma \ref{lemma: kartzers lemma}. The result follows by the universal coefficient theorem
$K^G_j(X, k) = (K^G_j(X) \otimes^0_{\Z} k ) \oplus (K^G_{j-1}(X) \otimes^1_{\Z} k ).$
\end{proof}
In fact, we don't know examples where $K^G_1(X, k) \neq 0$. In the non-equivariant setup, this is always zero for weight reasons -- see Observation \ref{lemma: perfect trace map for trivial group}.

\subsection{Perfect \texorpdfstring{$K$}{K}-theory and \texorpdfstring{$G$}{G}-theory agree}
After perfection, $K$-theory agrees with $G$-theory.
\begin{observation}[Perfect $K$-theory and $G$-theory]\label{lemma: K-theory and G-theory of perfect schemes}
    Let $X \in \Sch^{\pfp, G}_k$. Then there is a homotopy equivalence of K-theory spectra
    \begin{align*}
        K^G(X) & \simeq G^G(X).
    \end{align*}
\end{observation}
\begin{proof}
    By \cite[Proposition 11.31, Remark 11.32]{BS16}, any complex in $D^b_{\Qcoh}(X)$ has finite Tor amplitude. By \cite[Proposition 2.2.12]{TT90}, it follows that the notions of perfect complexes and pseudo-coherent complexes with locally bounded cohomology agree on $X$. Hence perfect complexes of globally finite Tor-amplitude and pseudo-coherent complexes with globally bounded cohomology agree, so \cite[Definition 3.1]{TT90} and \cite[Definition 3.3]{TT90} agree, proving the statement for the trivial group $G$. Also see \cite[Theorem 3.21]{TT90}.
    The same reasoning works with any $G$, adding ``$G$-equivariant complex" everywhere in the proof.    
\end{proof}

\subsection{Equivariant perfect proper excision}
Algebraic $K$-theory usually doesn't enjoy proper excision, the best general supplement being the pro-excision of \cite{Mor12, KST16}. Forcing it on all $\Sch_k^{\qcqs}$ results in $KH$ by \cite[Theorem 6.3]{KST16}.

However, for pfp perfect schemes, $K$-theory satisfies proper excision \cite[Proof of Theorem 4.3]{KM21}. We conjecture that perfect proper excision works $G$-equivariantly, and we prove it in the $T$-equivariant case in Theorem \ref{lemma: homotopy fiber square for K_T}. This is one of the main computational tools for our results.

\begin{setup}\label{setup: perfect abstract blowup squares}
Let $k=\F_p$.
Let $G_0$ be a reductive group over $k$ and suppose we are given an abstract blowup square of $G_0$-equivariant schemes in $\Sch_k^{\fp}$:
\begin{equation}\label{equation: abstract blowup square}
    \begin{tikzcd}
        Y_0 \arrow[d, "f_0"] \arrow[r, hookleftarrow] & E_0 \arrow[d, "f'_0"] \\
        X_0 \arrow[r, hookleftarrow] & Z_0
    \end{tikzcd}
\end{equation}
Passing to perfections (by dropping the subscript ``$0$") we obtain a {\it perfect abstract blowup square} of $G$-equivariant pfp perfect schemes over $k$:
\begin{equation}\label{equation: perfect abstract blowup square}
    \begin{tikzcd}
        Y \arrow[d, "f"] \arrow[r, hookleftarrow] & E \arrow[d, "f'"] \\
        X \arrow[r, hookleftarrow] & Z
    \end{tikzcd}
\end{equation}
\end{setup}

Let $T_0 =\Gm^n$ and $T$ its perfection. Before we embark on proving $T$-equivariant proper excision, we record the following useful lemma.
\begin{lemma}\label{lemma: perfectness of perfect torus quotients}
Let $k = \F_p$. Let $T_0= \Gm^n$ be a split torus acting on $X_0 \in \Sch_k^{\qcqs}$. Let $T$ and $X$ be the corresponding perfections. Then the stack $T \backslash X$ is flawless.
\end{lemma}
\begin{proof}
Let us first prove the case of quasi-projective $X_0$ following \cite[Corollary 3.22]{BFN10}.
First note that $BT = T \backslash \pt$ is flawless. Clearly, it has affine diagonal $T$. 
Both dualizable (perfect) and compact objects are given by finite-dimensional representations of the affine multiplicative group scheme $T$. These also generate, so $BT$ is flawless by say \cite[Proposition 3.9]{BFN10}.

Since $X_0$ is a quasi-projective $k$-variety, it is qcqs and possesses an ample family of line bundles; the same follows for $X$ by \cite[Lemmas 3.4 and 3.6]{BS16}. 
Now apply \cite[Proposition 3.21]{BFN10} to the morphism $T\backslash X \to T \backslash \pt$ of algebraic stacks as in the proof of \cite[Corollary 3.22]{BFN10}.

More generally, this works for any qcqs $X_0$. Since $T_0$ is of multiplicative type, \cite[Theorem 1.40, Example 1.41.(ii)]{Kha20} show that the global quotient $T_0 \backslash X_0$ of a qcqs scheme $X_0$ by $T_0$ is flawless. Since $\Perf(-)$, $D_{\Qcoh}(-)$ turn the Frobenius limit to a colimit and $\Ind(-)$ commutes with colimits (being left adjoint), we deduce $T \backslash X$ is flawless as well.
\end{proof}

\begin{remark}
If $G_0 = GL_n$ with $n \geq 2$ over $k$, then $BG_0 = G_0 \backslash \pt$ is not flawless: its structure sheaf is non-compact. This is the technical point that causes problems.   
\end{remark}

We are now ready to prove the $T$-equivariant case.
\begin{theorem}\label{lemma: homotopy fiber square for K_T}
Consider a $T$-equivariant perfect abstract blowup square of pfp schemes from Setup \ref{setup: perfect abstract blowup squares}.
Then the following square of equivariant K-theory spectra is a homotopy fiber square.
\begin{equation}\label{equation: fibered square of K-theory spectra}
    \begin{tikzcd}
        K^T(Y) \arrow[r] & K^T(E) \\
        K^T(X) \arrow[r] \arrow[u] & K^T(Z) \arrow[u]
    \end{tikzcd}
\end{equation}
In particular, we get the long exact sequence:
\begin{equation}\label{equation: long exact proper excision sequence in perfect T-equivariant k-theory}
        \dots \to K^T_1(E) \to K^T_0(X) \to K^T_0(Y) \oplus K^T_0(Z) \to K^T_0(E) \to K^T_{-1}(X) \to \dots
\end{equation}
The same conclusion holds if we replace $K$-theory by any localizing invariant $\E$.
\end{theorem}
\begin{proof}
It is enough to verify the following three conditions:
\begin{itemize}
    \item[(i)] $\Ind(\Perf(T\backslash S)) = D_{\Qcoh}(T \backslash S)$ as $\infty$-categories for the spaces $S = X, Y, Z, E$ featuring above,
    \item[(ii)] the induced square of $\infty$-categories given by applying $D_{\Qcoh}(-)$ is a pullback,
    \item[(iii)] the functor $Ri_*: D_{\Qcoh}(T \backslash E) \to D_{\Qcoh}(T\backslash Y)$ is fully faithful.
\end{itemize}

Assuming (i), (ii), (iii), we follow the reasoning in \cite[Proof of Theorem 4.3]{KM21}. To apply \cite[Theorem 18]{Tam17}, we need to check that the square given by $\Perf^T(-)$ is excisive \cite[Definition 14]{Tam17}. By (i), the Ind-completion of this square agrees with the analogous square of $\D_{\Qcoh}(-)$, which satisfies the desired conditions by (ii) and (iii).
We prove (i), (ii), (iii) below, concluding the proof.
\end{proof}

\begin{proof}[Proof of (i)]
This follows from Lemma \ref{lemma: perfectness of perfect torus quotients}.
\end{proof}

\begin{proof}[Proof of (ii)]
We prove that the following square is a pullback of $\infty$-categories:
\begin{equation}\label{equation: fibered square for qcoh of global qution stacks}
    \begin{tikzcd}
        D_{\Qcoh}(T \backslash Y) \arrow[r] & D_{\Qcoh}(T \backslash E) \\
        D_{\Qcoh}(T \backslash X) \arrow[r] \arrow[u] & D_{\Qcoh}(T \backslash Z) \arrow[u]
    \end{tikzcd}
\end{equation}
For perfect schemes, this is proved in \cite[Theorem 11.2.(1)]{BS16}. To deduce the same result for perfect global quotient stacks, note that perfection commutes with the formation of global quotients by Proposition \ref{proposition: perfection of global quotient stacks}. Moreover, the diagram
\begin{equation}\label{diagram: resolution of global quotient stack}
   T\backslash X \leftarrow X \ \substack{\leftarrow\\[-1em] \leftarrow} \ T \times X \ \substack{\leftarrow\\[-1em] \leftarrow \\[-1em] \leftarrow } \ T \times T \times X \ \substack{\leftarrow\\[-1em] \leftarrow \\[-1em] \leftarrow \\[-1em] \leftarrow} \dots
\end{equation}
allows to compute
\begin{equation}\label{equation: qcoh on stacks as a limit}
 D_{\Qcoh}(T \backslash X) = \lim \left( D_{\Qcoh}(X) \ \substack{\rightarrow\\[-1em] \rightarrow} \ D_{\Qcoh}(T \times X) \ \substack{\rightarrow\\[-1em] \rightarrow \\[-1em] \rightarrow} \ D_{\Qcoh}(T \times T \times X) \ \substack{\rightarrow\\[-1em] \rightarrow \\[-1em] \rightarrow \\[-1em] \rightarrow} \ \dots \right)   
\end{equation}
by faithfully flat descent. More precisely, $T \backslash X \leftarrow X$ is a $v$-cover by \cite[Definition 11.1 and Example 2.3]{BS16} so this also follows from \cite[Theorem 11.2.(1)]{BS16} and the definition of $D_{\Qcoh}(T \backslash X)$. Taking iterated self-products of this cover precisely returns \eqref{diagram: resolution of global quotient stack}. 

Each of the terms in the limit \eqref{equation: qcoh on stacks as a limit} is a scheme. Taking $(T\times \dots \times T \times (-))$ of the original abstract blowup is an abstract blowup, so the corresponding square of $\infty$-categories after applying $D_{\Qcoh}(-)$ is a pullback. Commuting limits, we deduce that the limit square \eqref{equation: fibered square for qcoh of global qution stacks} is a pullback as well.
\end{proof}

\begin{proof}[Proof of (iii)]
Since $i: E \to Y$ is a closed immersion of perfect schemes, $Li^* Ri_* = \id$ on $D_{\Qcoh}(E)$ by \cite[Lemma 3.16]{BS16}. The equivariant structure carries through, so the same holds on $D_{\Qcoh}(T \backslash E)$. Hence $Ri_*$ is fully faithful as
$\Hom(Ri_*(\eF), Ri_*(\eG)) = \Hom(Li^*Ri_*(\eF), \eG) = \Hom(\eF, \eG)$.
\end{proof}

\begin{remark}
Note that claims (ii) and (iii) hold without change for any perfectly reductive group $G$ instead of $T$. The technical issue in generalizing Theorem \ref{lemma: homotopy fiber square for K_T} from $T$ to $G$ is the failure of (i). 
To apply \cite[Theorem 18]{Tam17}, we would need to work with the square given by $\Ind(\Perf^G(-))$ directly.    
\end{remark}

In general, we conjecture the following.
\begin{conjecture}
Perfect proper excision holds $G$-equivariantly.    
\end{conjecture}
Regarding the validity of this conjecture, we know:
\begin{itemize}
    \item[(a)] In the non-equivariant setup, it holds by \cite[Proof of Theorem 4.3]{KM21}. 
    \item[(b)] For $K^T(-)$, we proved it in Theorem \ref{lemma: homotopy fiber square for K_T}. 
    \item[(c)] For $KH^{T}(-)$, it holds by standard techniques (without perfection) by Proposition \ref{proposition: fibered square of KH_G-theory spectra}.
    \item[(d)] For Borel-type equivariant $K$-theory $K^G_{\lhd}(-)$, it holds by Kan-extension from (a).
\end{itemize}

\subsection{Homotopy \texorpdfstring{$K$}{K}-theory}
Let us record some immediate properties of $KH^G(-)$ for perfect schemes. 
\begin{lemma}
\begin{equation*}
    KH^G_i(X) = \underset{\varphi^*}{\colim} \ KH^{G_0}_i(X_0).
\end{equation*}
\end{lemma}
\begin{proof}
    Formally follows from Lemma \ref{proposition: K-theory of inverse limit} and the definition of $KH(-)$ by commuting colimits.
\end{proof}

\begin{lemma}\label{lemma: homotopy fiber square for KH}
$KH^T(-)$ satisfies perfect proper excision.    
\end{lemma}
\begin{proof}
Take the perfection of Proposition \ref{proposition: fibered square of KH_G-theory spectra}.   
\end{proof}

It is reasonable to expect the following.
\begin{conjecture}\label{conjecture: perfect K is perfect KH}
The canonical map $K^G(X) \to KH^G(X)$ is an equivalence.   
\end{conjecture}
Non-equivariantly, the map $K(X) \to KH(X)$ is actually an equivalence for any pfp perfect scheme $X$ by \cite[Proposition 5.1]{AMM22}. However, their argument needs proper excision as an input.

\begin{remark}
If $X$ is $T$-equivariant and perfectly smooth, then $K^T(X) \to KH^T(X)$ is an equivalence.   
By the above theorems, both sides satisfy descent on equivariant perfect abstract blowup squares. We conclude that $K^T(X) \to KH^T(X)$ is an equivalence whenever there exists a sequence of equivariant abstract blowups reducing everything to perfectly smooth schemes. 
We will see this happen in situations where explicit resolutions are known -- see Theorem \ref{observation: K = KH for perfect affine grassmannian} and Theorem \ref{observation: K = KH for toric varieties}.
\end{remark}

%% file: 4_perfect_trace_maps.tex
\section{Perfect fixed-point schemes and trace maps}\label{section: perfect trace maps}
We now record how the fixed-point schemes and trace map from \S \ref{section: fixed-point schemes and traces} play out in the perfect setup of \S \ref{section: perfect schemes}, \S \ref{section: perfect K-theory}. We discuss how to control whether the trace map is an isomorphism; we also provide some direct geometric intuition.

\subsection{Perfect fixed-point schemes}
The fixed-point schemes and trace maps from \S \ref{section: fixed-point schemes and traces} specialize to the perfect setup. 
\paragraph{Immediate observations.}
In Setup \ref{setup: perfect geometry}, we can apply the fixed-point scheme functor $\Fix_{\frac{G}{G}}(-)$ and its variants.
Since perfect schemes are reduced \hyperlink{property: reduced}{(iii)} and perfect derived schemes are classical \hyperlink{property: classical}{(iv)}, the natural maps \eqref{equation: reduced, classical and derived fixed-point schemes} give isomorphisms
\begin{equation*}
    \Fix^{\mathbf{L}}_{\frac{G}{G}}(X) \cong \Fix_{\frac{G}{G}}(X) \cong \Fix^{\red}_{\frac{G}{G}}(X).
\end{equation*}
Since $(-)_{\perf}$ commutes with fiber products \hyperlink{property: limits}{(ii)}, it formally follows that
\begin{equation*}
    \Fix_{G}(X) \cong \left( \Fix_{G_0}(X_0) \right)_{\perf}.
\end{equation*}
Compatibility with stack quotients from Proposition \ref{proposition: perfection of global quotient stacks} further shows
\begin{equation*}
    \Fix_{\frac{G}{G}}(X) \cong \Fix_{\frac{G_0}{G_0}}(X_0)_{\perf}.
\end{equation*}
and similarly for any restriction $\Fix_S(X)$.
Passing through the limit along Frobenius, we in particular note that our trace map
\begin{equation*}
    \tr_{X}: K^G_0(X)_k \to H_0(\Fix_{\frac{G}{G}}(X), \O)
\end{equation*}
is given by the perfection of 
\begin{equation*}
    \tr_{X_0}: K^{G_0}_0(X_0)_k \to H_0(\Fix_{\frac{G_0}{G_0}}(X_0), \O).
\end{equation*}
In particular, if $\tr_{X_0}$ is an isomorphism, the same is true for $\tr_{X}$.

\paragraph{Perfect tori.}
Let $T$ be a perfect torus; then the situation simplifies further.
Since global quotients by $T$ are flawless by Lemma \ref{lemma: perfectness of perfect torus quotients}, we have
\begin{equation}\label{equation: HH^T and fixed point schemes for perfect schemes}
    HH^T_i(X, k) = \RG_i(\Fix_{\frac{T}{T}}(X), \O). 
\end{equation}
In particular,
\begin{equation}\label{equation: perfect HH^T vanishes in positive degrees}
        HH^T_i(X, k) = \RG_i(\Fix_{\frac{T}{T}}(X), \O) = 0, \qquad i \geq 1.
\end{equation}
Since $(-)^T$ is exact on functions, we can commute it with taking cohomology to get
\begin{equation}\label{equation: torus invariant functions}
\RG_i(\Fix_{\frac{T}{T}}(X), \O) = H_i(\Fix_T(X), \O)^T, \qquad i \in \Z.    
\end{equation}
In particular, this can be nonzero only in those degrees where coherent cohomology of $\Fix_T(X)$ is nonzero. Such vanishing can be further checked fiberwise over $T$.
\begin{lemma}\label{lemma: perfect base change on fixed point schemes}
Let $i \in \Z$. Then $\RG_i(\Fix_{T}(X), \O) = 0$ if and only if $\RG_i(\Fix_t(X), \O) = 0$ for all geometric points $t$ of $T$.
\end{lemma}
\begin{proof}
Considering the structure map $\Fix_T(X) \to T$ with $T$ affine, the result holds by perfect base-change \cite[Lemma 3.18]{BS16}.    
\end{proof}

\subsection{Trace map for equivariant points}
Given a split reductive group $G_0$ over $k$, we discuss perfect equivariant $K$-theory of the point. This closely relates to the representation theory of perfect reductive groups studied in \cite{CW22}. In classical terms, the Frobenius map corresponds to the multiplication by $p$ on characters and consequently to Frobenius twisting on $\Rep_k(G_0)$.
\begin{lemma}\label{lemma: trace for point}\label{lemma: perfect equivariant point}
Let $X_0 = \pt$ and $T_0 = \Gm^n$. Then 
\begin{equation*}
    \tr: K^T(\pt, k) \to \RG(\Fix_{\frac{T}{T}}(\pt), \O)
\end{equation*}
is an isomorphism of spectra. Both sides are supported in homotopical degree $0$ with value
\begin{equation*}
    k[t_1^{\pm \tfrac{1}{p^{\infty}}}, \dots, t_n^{\pm \tfrac{1}{p^{\infty}}}].
\end{equation*}
\end{lemma}
\begin{proof}
Since $\Rep_T$ is equivalent to the category of $\Z[\tfrac{1}{p}]^n$-graded vector spaces, we can compute $K^T(\pt)$ from $K(\pt)$ by additivity. But the latter was computed by Quillen \cite[Corollary IV.1.13]{Weib13}. In particular, the $k$-linear version $K^T(\pt, k)$ is supported in degree $0$ with value given by the representation ring of $T$.

On the other hand, $\Fix_{\frac{T}{T}}(\pt) = \frac{T}{T}$, so the right-hand side is $\RG(T, \O)^T$. Since $T$ is affine, this is concentrated in degree $0$ by Serre's vanishing theorem. There it is simply given by $H_0(T, \O)$ which agrees with the above representation ring via the trace map. 
\end{proof}

\begin{lemma}\label{lemma: perfect G-equivariant point}
Let $X=\pt$ with the trivial $G$-action. Then the trace map gives an isomorphism
\begin{equation*}
  \tr_X:  K^G_0(\pt, k) \xrightarrow{\cong} H_0(\Fix_{\frac{G}{G}}(\pt), \O) \xrightarrow{\cong} H_0(\s, \O).
\end{equation*}
Both sides are given by
\begin{equation*}
k[t_1^{\pm \tfrac{1}{p^{\infty}}}, \dots, \ t_n^{\pm \tfrac{1}{p^{\infty}}}]^{W}.
\end{equation*}
\end{lemma}
\begin{proof}
This is already the case before perfection for $\s_0$ and $G_0$ by Example \ref{example: trace map in k-theory for point}.
Hence $\tr$ is also isomorphism after perfecting.
\end{proof}
\noindent For example if $G_0 = \GL_n$, the above ring is explicitly given by $ k[c_1^{\tfrac{1}{p^{\infty}}}, \dots, \ c_{n-1}^{\tfrac{1}{p^{\infty}}}, c_n^{^{\pm \tfrac{1}{p^{\infty}}}}]$.

\begin{remark}
Note that $K^G_0(\pt, k)$ is supported in degree zero. On the other hand, whenever the adjoint representation of $G$ on its ring of functions has higher group cohomology, $\RG(\Fix_{\frac{G}{G}}(\pt), \O)$ will be nontrivial also in negative homological degrees. On the other hand, $H_0(\Fix_{\s}(\pt), \O)$ is supported in degree $0$ since $\s$ is an affine scheme. 
\end{remark}

\subsection{Perfect projective bundles and traces}
We now record how $T$-equivariant $K$-theory and fixed-point schemes behave under passage to perfect projective bundles. The easiest argument works for any localizing invariant, but we also sketch simple geometric arguments for fixed-point schemes valid in degree zero in \S \ref{section: global functions on perfect fixed-point schemes}.

\begin{lemma}\label{lemma: stability on perfect projective bundles}
Let $T$ act on $X$ and $\eE \in \VB^T(X)$ of rank $r$. Let $Y_0 = \P(\eE)$ be the corresponding projective space over $X$ and $Y$ its perfection. Let $\tr: K^T(-, k) \to HH^T(-, k)$ be the trace map. Then the following are equivalent:
\begin{itemize}
    \item[(i)] $\tr$ is isomorphism for $X$,
    \item[(ii)] $\tr$ is isomorphism for $Y$.
\end{itemize}
\end{lemma}
\begin{proof}
By Proposition \ref{proposition: projective bundle formula for localizing invariants} we have natural identifications $K^T(\P(\eE), k) = \prod_{i=0}^{r-1} K^T(X, k)$ and $HH^T(\P(\eE), k) = \prod_{i=0}^{r-1} HH^T(X, k)$. Hence $\tr$ is an isomorphism on $X$ if and only if it is an isomorphism on $\P(\eE)$. 
Now, $K_T(Y, k) = \colim_{\varphi^*} K_T(\P(\eE), k)$ and $HH_T(Y, k) = \colim_{\varphi^*} HH_T(\P(\eE), k)$ compatibly with these decompositions.

\noindent
(i) $\implies$ (ii). If $\tr$ is an isomorphism for $X$, it is an isomorphism for $\P(\eE)$ and hence for $Y$ by exactness of filtered colimits. 

\noindent
(ii) $\implies$ (i). Assume $\tr$ is an isomorphism for $Y$. Since $X$ is perfect and $K^T(X,k) \to HH^T(X, k)$ is a retract of $K^T(\P(\eE), k) \to HH^T(\P(\eE), k)$ compatibly with the Frobenius, we know that $K^T(X, k)\to HH^T(X, k)$ is a retract of $K^T(Y, k) \to HH^T(Y, k)$. Hence $\tr$ is an isomorphism for $X$.
\end{proof}

\begin{lemma}\label{lemma: perfect grassmannian bundle formula}
Let $T$ act on $X$ and $\eE \in \VB^T(X)$ of rank $r$. Let $Y = \Flag_X(\eE, \mu)_{\perf}$ be the corresponding perfect flag variety of type $\mu$ over $X$. Let $\tr: K^T(-, k) \to HH^T(-, k)$ be the trace map. Then the following are equivalent:
\begin{itemize}
    \item[(i)] $\tr$ is isomorphism for $X$,
    \item[(ii)] $\tr$ is isomorphism for $Y$.
\end{itemize}
\end{lemma}
\begin{proof}
Take perfection of \eqref{equation: partial flag variety bundles}. Applying Lemma \ref{lemma: stability on perfect projective bundles} multiple times for the bundles appearing in the factorizations of $g$ and $h$, we deduce the result for $f$. 
\end{proof}

\begin{example}\label{lemma: perfect trace for projective space}
Let $\eE \in \VB^T(\pt)$. Denote $\lambda_1, \dots, \lambda_m$ the set of distinct weights of $T$ on $\eE$. Let $Y = \P(\eE)_{\perf}$. Then the trace map 
$$K^T(Y) \xrightarrow{\simeq} \RG(\Fix_{\frac{T}{T}}(Y), \O)$$
is an equivalence supported in degree zero. In degree zero, both sides are given by the ring
$$\left( k[t_1^{\pm 1}, \dots, t_n^{\pm 1}, x^{\pm 1}] / \left( \prod_{i=1}^m (x-\lambda_i) \right) \right)_{\perf}.$$
\end{example}
\begin{proof}
For the point, the trace map is an isomorphism concentrated in degree zero. The same follows for $\P(\eE)$ by the projective bundle formula for $K(-, k)$ and $HH(-, k)$. This isomorphism is compatible with Frobenius pullback and taking the colimit along it is exact. Hence $K^T(Y) \to HH^T(Y, k)$ is still an isomorphism supported in degree zero.

The explicit presentation follows by the projective bundle formula for $\P(\eE)$ and by noting that perfection discards nilpotence (so weight multiplicities don't contribute).
\end{proof}

Finally, let us record how this plays out for perfect stratified projective bundles.
\begin{lemma}\label{lemma: K-theory of perfect stratified grassmannian bundles}
Let $X \in \Sch^{\pfp, T}_k$ and $\eF \in \Qcoh^T(X)$. Assume there is a finitely presented model $X_0 \in \Sch^{\fp, T_0}_k$ with an ample family of line bundles so that $\eF$ is a pullback of a coherent sheaf $\eF_0 \in \Coh^{T_0}(X_0)$. Let $Y = \Grass_X(\eF, 1)_{\perf}$ be the associated perfect stratified Grassmannian.
Then $K^T(X)$ is a natural direct summand of $K^T(Y)$; similarly for other localizing invariants and natural maps between them.
\end{lemma}
\begin{proof}
By the assumptions on $X_0$, the quotient stack $T_0 \backslash X_0$ has the resolution property \cite[\S 2]{Tho87}. We can thus write $\eF_0$ as the zeroth homology $\eH_0$ of a $T_0$-equivariant perfect complex on $X_0$ of Tor-amplitude $[1, 0]$. Pulling this back to $X$ presents $\eF$ on $X$ as $\eH_0$ of a $T$-equivariant perfect complex $\eE$ of Tor-amplitude $[1, 0]$. Now we can form the derived enhancement $\Grass_X(\eE, 1)$ of $\Grass_X(\eF, 1)$ of \cite{Jia22a, Jia22b} and apply Recollection \ref{recollection: semiorthogonal decompositions and localizing invariants} and Proposition \ref{proposition: K-theory and derived grassmannians}. This remains true after perfection by exactness of the Frobenius colimit. Since perfection discards the derived structure, the final statement indeed features the classical $Y = \Grass_X(\eF, 1)_{\perf} = \Grass_X(\eE, 1)_{\perf}$. 
\end{proof}

\subsection{Global functions on perfect fixed-point schemes}\label{section: global functions on perfect fixed-point schemes}
We record a few more auxiliary results on global functions on perfect fixed-point schemes. We hope this gives a clear geometric intuition about their behaviour, independent of the categorical construction of Hochschild homology above. Let $G$ be the perfection of $\GL_n$ (or alternatively any perfect, split, semisimple, simply connected group).

\paragraph{Perfect proper excision for functions on fixed-point schemes.}
Consider an abstract blowup square \eqref{equation: perfect abstract blowup square}. Applying the functor $\Fix_{\frac{G}{G}}(-)$ we get the following pullback square.
\begin{equation}\label{equation: fixed-point scheme of an abstract blowup}
    \begin{tikzcd}
        \Fix_{\frac{G}{G}}(Y) \arrow[d] \arrow[r, hookleftarrow] & \Fix_{\frac{G}{G}}(E) \arrow[d] \\
        \Fix_{\frac{G}{G}}(X) \arrow[r, hookleftarrow] & \Fix_{\frac{G}{G}}(Z)
    \end{tikzcd}
\end{equation}

The corresponding proper excision can be seen geometrically as follows.
\begin{lemma}\label{lemma: perfect descent for functions on fixed-point schemes}
There is a distinguished triangle
\begin{equation*}
    \RG(\Fix_{\frac{G}{G}}(X), \O) \to \RG(\Fix_{\frac{G}{G}}(Y), \O) \oplus \RG(\Fix_{\frac{G}{G}}(Z), \O) \to \RG(\Fix_{\frac{G}{G}}(E), \O).
\end{equation*}
and similarly for the versions with $\Fix_S(-)$.
In particular, there is an exact sequence   
\begin{equation*}
    0 \to H_0(\Fix_{\frac{G}{G}}(X), \O) \to H_0(\Fix_{\frac{G}{G}}(Y), \O) \oplus H_0(\Fix_{\frac{G}{G}}(Z), \O) \to H_0(\Fix_{\frac{G}{G}}(E), \O) \to \cdots
\end{equation*}
\end{lemma}
\begin{proof}
We first treat the fixed-point scheme $\Fix_G(-)$. 
Since perfections are compatible with open and closed immersions by Proposition \ref{proposition: open and closed immersions}, the square \eqref{equation: fixed-point scheme of an abstract blowup} is still an abstract blowup. The induced sequence is exact by \hyperlink{property: functions}{(iv)}.
To pass to $\Fix_{\frac{G}{G}}(-)$, apply the right derived functor of $(-)^G$. The version for $\Fix_S(-)$ works by the same reasoning.
\end{proof}

\paragraph{A fiberwise criterion.}
We record the following simple fiberwise criterion for checking whether the trace map is an isomorphism in degree zero. Let $\s \hookrightarrow G$ be the perfect Steinberg section.
For $X = \pt$ it holds that $\Fix_{\s}(\pt) \cong \s$ and the trace map induces canonical isomorphisms
\begin{equation*}
\Fix_T(\pt) \xrightarrow{\cong} \Spec K^T_0(\pt)_k,
\qquad \text{and} \qquad
\Fix_{\s}(\pt) \xrightarrow{\cong} \Spec K^G_0(\pt)_k.
\end{equation*}
The degree zero trace map can be understood fiberwise over the equivariant base by the following simple criterion.
\begin{lemma}\label{lemma: fiberwise criterion}
Let $X \in \Perf^{\pfp, G}_{k}$ and assume that for each closed geometric point $s \in \s$, the trace map 
\begin{equation*}
    \begin{tikzcd}
        \Fix_{\s}(X) \arrow[rr, "\tr_X"] \arrow[dr]& & \Spec K^G_0(X)_k \arrow[dl] \\
        & \s &
    \end{tikzcd}
\end{equation*}
induces a bijection on $\pi_0$ of the fibers over $s$.
Then $K^G_0(X)_k \to H_0(\Fix_{\s}(X), \O)$ is an isomorphism.
\end{lemma}
\begin{proof}
    Follows from Lemma \ref{lemma: perfect stein 2}.
\end{proof}

\paragraph{Geometric viewpoint on the projective bundle formula.}

The degree zero part of Lemma \ref{lemma: stability on perfect projective bundles} for $G$ can be also seen directly from the geometry of the fixed-point scheme as follows. Let $X \in \Sch_k^{\pfp, G}$, $\eE \in \VB^G(X)$ and $Y = \P(\eE)_{\perf}$. We claim that
\begin{equation*}
 \tr: K^G_0(X, k) \to H_0(\Fix_{\s}(X), \O) \ \text{isomorphism} \iff   \tr: K^G_0(Y, k) \to H_0(\Fix_{\s}(Y), \O) \ \text{isomorphism}.
\end{equation*}
Indeed, consider the diagram
\begin{equation}
\begin{tikzcd}
    \Fix_{\s}(Y) \arrow[d] \arrow[r, "\tr"]&  \Spec K^G_0(Y, k) \arrow[d] \\
    \Fix_{\s}(X)  \arrow[r, "\tr"]&  \Spec K^G_0(X, k)
\end{tikzcd}
\end{equation}
Let $(g, x): \pt \to \Fix_G(X)$ be a geometric point and $(g, x)': \pt \to \Fix_G(X) \to \Spec K^G_0(X, k)$ the corresponding geometric point of $\Spec K^G_0(X, k)$.
We restrict attention to the fiber $\Fix_G(Y)_{(g, x)} = \Fix_{g}(Y_x)$ of $\Fix_G(Y)$ over $(g, x)$, whose points are given by the $1$-dimensional subspaces of $\eE_x$ fixed by $g$. Hence $\pi_0(\Fix_G(Y)_{(g, x)})$ identifies with the set $\Eig_g(\eE_x)$ of distinct eigenvalues of $g$ on $\eE_x$.

By the projective bundle formula for equivariant $K$-theory, we also have:
\begin{equation}
\Spec (K^G_0(Y, k))_{(g, x)'} = \Spec \left( k[\xi] \ / \left( \sum_{i=0}^r (-1)^{i} \cdot \tr_g [\wedge^i \eE_x] \cdot \ \xi^{r-i} \right)  \right)_{\perf} 
\end{equation} 
Factoring this polynomial, it follows that $\pi_0(\Spec (K^G_0(Y, k))_{(g, x)'})$ is also given by the set $\Eig_g(\eE_x)$ of distinct eigenvalues of $g$ on $\eE_x$.
The restricted trace map is a bijection
\begin{equation}\label{equation: relative fiber bijection}
\tr: \pi_0(\Fix_{\s}(Y)_{(g, x)}) \xrightarrow{\cong}   \pi_0(\Spec(K^G_0(Y, k))_{(g, x)'})
\end{equation}
given by identity on $\Eig_g(\eE_x)$ under the above identifications. We then conclude by Lemma \ref{lemma: fiberwise criterion}.

%% file: 5_trace_map_for_affine_grassmannian.tex
\section{Trace map for perfect affine Grassmannians}\label{section: trace map for GL_2 affine grassmannian}

Our main examples of interest are the affine Grassmannians and affine Schubert varieties, in both classical and perfect setups. As a proof of concept, we prove our main result for $T$-equivariant $K$-theory of affine Schubert varieties $X_{\le \mu}$ in the perfect affine Grassmannian $\Gr$ for $\GL_n$.

We recall affine Grassmannians and their perfect versions in \S \ref{subsection: affine grassmannians and their perfections}. We prove our conjecture for the $\GL_2$ affine Grassmannian in \S \ref{subsection: trace for perfect GL_2 affine grassmannian} using equivariant proper excision from Theorem \ref{lemma: homotopy fiber square for K_T} and usual projective bundle formula; this allows for explicit computations. We then in \S \ref{section: trace for GL_n affine grassmannian} use the semi-orthogonal decomposition for derived projective bundles \cite{Jia23} to prove the general case of the $\GL_n$ affine Grassmannian.

\subsection{Affine Grassmannians and their perfections}\label{subsection: affine grassmannians and their perfections}
We make a brief overview of affine Grassmannians and their perfect versions. This is a standard material; up to perfection, we follow the excellent introduction \cite{Zhu15}. For perfection and the Witt-vector affine Grassmannian, see \cite{Zhu14, BS16}.

\paragraph{Notation.}
Let $k$ be a perfect field of positive characteristic; we take $k=\F_p$. As above, we denote $\Sch^{\pfp}_k$ the category of perfectly finitely presented perfect schemes over $k$. Consider the reductive group $G_0 = GL_n$ over $k$ and let $G$ be its perfection. Let $T_0$ be its split diagonal torus and $T$ its perfection. Let $\X_{*} := \X_{*}(T_0)$ be the cocharacter lattice of $T_0$; we denote by $\X^{+}_{*}:= \X^{+}_{*}(T_0)$ the monoid of dominant cocharacters (with respect to the upper-triangular Borel subgroup). If $\mu_{\bullet} = (\mu_1, \mu_2, \dots, \mu_d)$ is a sequence of cocharacters, we denote by $\mu = |\mu_\bullet| = \mu_1 + \mu_2 + \dots + \mu_d$ its sum.

\paragraph{Affine Grassmannian and its perfection.}
We mostly follow the notation from \cite{Zhu15}. By base change, $G_0$ gives a reductive group over $k \llp t \rrp$ with an integral model over $k\llb t \rrb$. Associated to these we have the loop groups with functors of points sending a test algebra $R \in \Alg_k$ to
\begin{equation*}
    \lo G_0: R \mapsto G_0(R(\!(t)\!))), \qquad \qquad \lo^+ G_0: R \mapsto G_0(R\llbracket t \rrbracket), \qquad  \qquad \lo^h G_0: R \mapsto G_0(R\llbracket t \rrbracket / (t^{h+1})).
\end{equation*}
We call these respectively {\it loop group} $\lo G_0$, {\it positive loop group} $\lo^+ G_0$ and {\it truncated loop groups} $\lo^h G_0$. The quotient
\begin{equation*}
(\Gr)_0 = \lo G_0 / \lo^+ G_0     
\end{equation*}
gives a well-defined ind-scheme called the {\it affine Grassmannian} of $G_0$. 
In the case of $\GL_n$, its functor of points on $k$-algebras is given by
\begin{equation*}
    R \mapsto \{ \Lambda \subseteq R\llp t \rrp^{\oplus n} \mid \Lambda \text{ full } R\llb t \rrb\text{-lattice} \}.
\end{equation*}

We denote $\lo G$, $\lo^+ G$, $\lo^h G$, $\Gr$ the perfections of these objects. We call $\Gr$ the {\it (equal characteristic) perfect affine Grassmannian}; it is the quotient
\begin{equation*}
    \Gr = \lo G / \lo^+ G.
\end{equation*}
From now on, everything is implicitly perfect. For the sake of readability, we omit this adjective. Since the same results hold before perfection, there is no risk of confusion.

\paragraph{Affine Schubert cells.}
For $\mu \in \X_{*}$, we let $t^\mu$ be the the image of the point $t \in \lo \Gm(k)$ under $\lo \mu: \lo \Gm(k) \to \lo G(k)$. The map $\mu \mapsto t^{\mu}$ gives a group homomorphism $\X_{*} \to \lo G(k)$. 
The $\lo^+ G$ (double) cosets induce a scheme theoretic stratification
\begin{align*}
\lo G(k) = \bigsqcup_{\mu \in \X_{*}^+} \lo^+G(k) \cdot t^\mu \cdot \lo^+G(k) \qquad \text{and} \qquad  \Gr(k) = \bigsqcup_{\mu \in \X_{*}^+} \lo^+G(k) \cdot t^\mu 
\end{align*}
of $\lo G$ resp. $\Gr$ into locally closed pieces $\lo G_{\mu}$ resp. $\Gr_{\mu}$ labelled by $\mu \in \X^{+}_{*}$ and called the {\it affine Schubert cells}. 

Explicitly, $\Gr_{\mu}$ is an affine bundle of dimension $(2\rho, \mu)$ over the partial flag variety $G/P_{\mu}$, where $P_\mu$ is the parabolic subgroup associated to $\mu$. The structure map is given by setting the formal parameter $t$ equal to zero as follows:
\begin{equation}\label{diagram: comparison of torsors}
\begin{tikzcd}[column sep=huge]
\lo G_{\mu} \arrow[d] \arrow[r, hookleftarrow] &  
\lo^+ G  \cdot t^{\mu} \arrow[d] \arrow[r, "(g \mapsto g \cdot t^{-\mu})"] &
\lo^+ G  \arrow[d] \arrow[r, twoheadrightarrow, "(t \mapsto 0)"]&
G  \arrow[d]\\
\Gr_{\mu} \arrow[r, equal] & \Gr_{\mu} \arrow[r, equal] & \Gr_{\mu} \arrow[r, twoheadrightarrow] & G/P_{\mu}
\end{tikzcd}
\end{equation}
The vertical maps in \eqref{diagram: comparison of torsors} are right torsors for the following groups and comparison maps between them:
\begin{equation*}
\begin{tikzcd}[column sep=50]
\lo^+G \arrow[r, hookleftarrow] & 
\lo^+G \cap (t^{-\mu} \cdot \lo^+G \cdot t^{\mu} ) \arrow[r, "g \mapsto t^{\mu}gt^{-\mu}"] & 
(t^{\mu} \cdot \lo^+G \cdot t^{-\mu}) \cap \lo^+G \arrow[r, "t\mapsto 0"]&
P_{\mu}
\end{tikzcd}
\end{equation*}
The map $\Gr_{\mu} \to G/P_{\mu}$ is an isomorphism if and only if $\mu$ is minuscule.

\paragraph{Affine Schubert varieties.}
The closures $X_{\leq \mu} = \Gr_{\leq \mu}$ of the affine Schubert cells $\Gr_{\mu}$ respect this stratification: there is a set-theoretic decomposition $\Gr_{\leq \mu} = \bigsqcup_{\lambda \leq \mu} \Gr_{\lambda} $. These $\Gr_{\leq \mu}$ are called the {\it affine Schubert varieties}. In particular, they carry an action of $\lo^+G$.

In contrast to the ind-scheme $\Gr$, affine Schubert varieties are of finite type. The singular locus of $\Gr_{\leq \mu}$ is precisely $\Gr_{\leq \mu} \setminus \Gr_{\mu} = \bigsqcup_{\lambda < \mu} \Gr_{\lambda}$.
In particular, $\Gr_{\leq \mu}$ is smooth if and only if $\mu$ is minimal in the dominance order.

The singularities of $\Gr_{\leq \mu}$ are rational by \cite[Theorem 2.1.21]{Zhu15} and Observation \ref{lemma: perfection preserves rationality}.

In terms of the moduli description, $\Gr_{\leq \mu}$ is the closed subfunctor of $\Gr$ parametrizing those $\Lambda$ which are in relative position $\leq \mu$ with respect to the standard lattice $\Lambda_0 = k\llb t \rrb^{\oplus n}$. 

\paragraph{Affine Demazure resolutions.}
For any $d$-tuple $\mu_\bullet = (\mu_1, \dots, \mu_d)$ of dominant cocharacters
$\mu_1, \dots \dots, \mu_d \in \X_{*}^+$, the associated {\it affine Demazure variety} is the iterated twisted product
\begin{align*}
Y_{\leq \mu_{\bullet}} = \Gr_{\leq \mu_\bullet}  &:= \lo G_{\leq \mu_1 }\overset{\lo^+G}{\times} \lo G_{\leq \mu_2 }\overset{\lo^+G}{\times} \dots \overset{\lo^+G}{\times} \lo G_{\leq \mu_d } \overset{\lo^+G}{\times} \pt \\
& = \Gr_{\leq \mu_1 }\overset{\sim}{\times} \Gr_{\leq \mu_2 }\overset{\sim}{\times} \dots \overset{\sim}{\times} \Gr_{\leq \mu_d}.
\end{align*}
with respect to the $\lo^+G$--torsors $\lo G_{\leq \mu_i } \to  \Gr_{\leq \mu_i }$ for $i = 1, 2, \dots, d$.

Note that $\Gr_{\leq \mu_{\bullet}}$ carries a natural action of $\lo_{\leq \mu_1}G$ induced from the left translation on the first factor. In particular, it has an action of $\lo^+G$. The $(d-1)$-fold multiplication
\begin{equation*}
\lo G_{\leq \mu_1}\times \lo G_{\leq \mu_2 } \times \dots \times \lo G_{\leq \mu_d } \xrightarrow{m} \lo G_{\leq |\mu_\bullet|}
\end{equation*}
factors through the affine Demazure variety, giving rise to the convolution map
\begin{equation}\label{equation: convolution map}
\Gr_{\leq \mu_{\bullet}} \xrightarrow{m} \Gr_{\leq |\mu_{\bullet}|}.    
\end{equation}
This is a proper surjective $\lo^+G$-equivariant map. It restricts to an an isomorphism over the open cell $\Gr_{|\mu_{\bullet}|} \hookrightarrow \Gr_{\leq |\mu_{\bullet}|}$ in the target.

In terms of the moduli description, $\Gr_{\leq \mu_\bullet}$ parametrizes chains of lattices $(\Lambda_0, \Lambda_1, \dots, \Lambda_d)$ with $\Lambda_i$ and $\Lambda_{i-1}$ in relative position $\mu_i$ and $\Lambda_0$ the standard lattice. The convolution map sends this chain to $\Lambda_d$.

In particular, assume all $\mu_i$ are minuscule. Since each $\Gr_{\leq\mu_i} = G/P_{\mu_i}$ is smooth, $\Gr_{\leq \mu_{\bullet}}$ is smooth as well. In fact, it is an iterated Grassmannian bundle, coming from a sequence of equivariant algebraic vector bundles, which we now describe.

\paragraph{Affine Demazure resolutions as Grassmannian bundles.}
Given $\mu \in \X^{+}_{\bullet}$, the variety $X_{\leq \mu} := \Gr_{\leq \mu}$ carries a natural vector bundle $\eE = \eE_{\leq \mu}$ defined as follows. We write $\mo^{\fl}(R\llb t \rrb)$ for the abelian category of finite length $R\llb t \rrb$-modules. 

Let $\omega_1, \omega_2, \dots, \omega_n$ be the fundamental dominant coweights of $GL_n$ given as $\omega_j = (1, \dots, 1, 0, \dots, 0)$ with $j$ ones and $n-j$ zeroes. Let $\omega_1^*, \dots, \omega_n^*$ be given as $\omega^*_j = \omega_{n-j} - \omega_n = (0, \dots, 0, -1, \dots, -1)$ with $n-j$ zeroes and $j$ minus ones. For any $\mu \in \X^{+}_{*}$, there is a unique integer $m$ such that $\mu = b \omega_n + \mu_1 + \mu_2 + \dots + \mu_{d}$ with each $\mu_i$ among $\omega^*_1, \dots, \omega^*_{n-1}$. This decomposition is unique up to reordering the terms; we call $\lg(\mu) := d$ the length of $\mu$.

Note that $b$ is the smallest integer such that any $\Lambda \in X_{\leq \mu}(R)$ contains $t^{b} \Lambda_0$.
Then define
\begin{align*}
  \eE = \eE_{\leq \mu}: X_{\leq \mu}(R) & \to \mo^{\fl}(R\llb t \rrb) \\
     \Lambda & \mapsto \Lambda / t^{b}\Lambda_0.
\end{align*}
This assembles into a $\lop G$-equivariant vector bundle on $X_{\leq \mu}$, which moreover carries a nilpotent operator $t: \eE \to \eE$ making it an object of $\mo^{\fl}(R\llb t \rrb)$. Consider now
\begin{equation*}
  \eE_{\bullet} := [ \eE \xrightarrow{t} \eE ] \ \qquad \text{and} \qquad \eF = \eF_{\leq \mu} := \eH_0(\eE_{\bullet}) = \coker ( \eE \xrightarrow{t} \eE ).
\end{equation*}
Then $\eE_{\bullet}$ is a perfect complex of Tor amplitude $[1, 0]$ resolving the coherent sheaf $\eF$. The affine Schubert stratification of $X_{\leq \mu}$ is a flattening stratification for $\eF$, meaning that the restrictions of $\eF$ to affine Schubert cells are vector bundles.

The affine Demazure resolution is given inductively as follows. Write $\mu_{\bullet} = ( \lambda, \mu_d)$ with $\mu_d$ coming from above and $\lambda = \mu - \mu_d$. Say $\mu_d = \omega^*_j$ for some $1 \leq j \leq n-1$. Then the convolution map \eqref{equation: convolution map} is given by the structure map
\begin{equation*}
    Y_{\leq \mu_{\bullet}} = \Grass_{X_{\leq \mu}}(\eF, j) \xrightarrow{f} X_{\leq \mu}.
\end{equation*}
Denote $\eE'$ the preimage under $\eE \twoheadrightarrow \eF$ of the tautological quotient of $\eF$ on $Y_{\leq \mu_{\bullet}}$. The left-hand term of the tautological exact sequence on $Y_{\leq \mu_{\bullet}}$ 
$$0 \to \eE' \to f^* \eE \to \eE'' \to 0$$
is an equivariant coherent sheaf on $Y_{\leq \mu_{\bullet}}$ with an action of $t$ and one can continue inductively.

On the other hand, $Y_{\leq \mu_{\bullet}}$ is an honest Grassmannian bundle over $X_{\leq \lambda}$. For this, consider the vector bundle $\eG = \eG_{\leq \lambda}$ on $X_{\leq \lambda}$ given by
\begin{align*}
  \eG: X_{\leq \lambda}(R) & \to \mo(R) \hookrightarrow \mo^{\fl}(R\llb t \rrb) \\
     \Lambda & \mapsto t^{-1}\Lambda / \Lambda.
\end{align*}
This is a $\lop G$-equivariant vector bundle on $X_{\leq \lambda}$, carrying the zero action of $t$. We then have
\begin{equation*}
    Y_{\leq \mu_{\bullet}} = \Grass_{X_{\leq \lambda}}(\eG, j).
\end{equation*}
Altogether, we obtain the convolution diagram
\begin{equation}\label{diagram: partial demazure resolution as strarified grassmannian bundle}
    \begin{tikzcd}
       Y_{\leq \mu_{\bullet}} \arrow[r, "g"] \arrow[d, "f"] & X_{\leq \lambda} \\
       X_{\leq \mu}
    \end{tikzcd}
\end{equation}

\paragraph{Witt-vector affine Grassmannian.}
The analogs of the above spaces for mixed characteristic groups were constructed in \cite{Zhu14, BS16}. These objects are defined only in the perfect setup.

Given a non-archimedean local field of mixed characteristic $\mathscr{K}$ with ring of integers $\mathscr{O}$ and residue field $k$, base change of $G_0$ gives a reductive group over $\mathscr{K}$ with an integral model over $\mathscr{O}$. Let $\W(-)$ be the functor of ramified Witt vectors on perfect $k$-algebras, with a uniformizer $\varpi$. 
Associated to this we have the $p$-adic loop groups with functors of points sending a test algebra $R \in \Alg^{\perf}_k$ to
\begin{equation*}
    \lo G_0: R \mapsto G_0(\W(R)[\tfrac{1}{\varpi}]), \qquad \qquad \lo^+ G_0: R \mapsto G_0(\W(R)), \qquad  \qquad \lo^h G_0: R \mapsto G_0(\W_h(R)).
\end{equation*}
The quotient
\begin{equation*}
\Gr = \lo G / \lo^+ G     
\end{equation*}
is a well-defined perfect ind-scheme called the {\it Witt-vector affine Grassmannian}. 
In the case of $\GL_n$, its functor of points on perfect $k$-algebras is given by
\begin{equation*}
    R \mapsto \{ \Lambda \subseteq \W(R)[\tfrac{1}{\varpi}]^{\oplus n} \mid \Lambda \text{ full } \W(R)\text{-lattice} \}.
\end{equation*}
The discussion of the above paragraphs can be repeated in the Witt vector case upon replacing $R\llb t \rrb$ by $\W(R)$ and $t$ by $\varpi$; see \cite{Zhu14, BS16}. The resulting affine Schubert varieties are rational by \cite[Remark 8.5]{BS16}.

We still have the sheaf
\begin{align*}
  \eE: X_{\leq \mu}(R) & \to \mo^{\fl}(\W(R)) \\
     \Lambda & \mapsto \Lambda / \varpi^{b} \Lambda_0.  
\end{align*}
Strictly speaking, this is not a vector bundle on $X_{\leq \mu}$ since its fibers are not $\F_p$-modules. However, $\eF = \coker ( \eE \xrightarrow{\varpi} \eE )$ is a quasi-coherent sheaf on $X_{\leq \mu}$ and $\eG: \Lambda \mapsto \varpi^{-1} \Lambda / \Lambda$ is a vector bundle, making the rest of the discussion preceding \eqref{diagram: partial demazure resolution as strarified grassmannian bundle} valid; see \cite[\S 7-8]{BS16}.

\subsection{Trace map for perfect \texorpdfstring{$\GL_2$}{GL2} affine Grassmannian}\label{subsection: trace for perfect GL_2 affine grassmannian}
To convey the main idea, we now prove Theorem \ref{introtheorem: trace map for affine grassmannian} for the perfect $GL_2$ affine Grassmannian. This only needs equivariant proper excision and the projective bundle formula. Furthermore, proper excision gives a constructive recipe for computing the ring $K^T(X_{\leq \mu})$ integrally.

Let $\Gr$ be either the perfect affine Grassmannian or the Witt-vector affine Grassmannian.

\begin{theorem}\label{theorem: trace for GL_2 affine grassmannian}
Let $G_0 = \GL_2$ with its diagonal torus $T_0$ over $k$. Then for each affine Schubert variety $X_{\leq \mu}$ of $\Gr$, the trace
\begin{equation}\label{equation: trace iso for GL2}
\tr: K^T(X_{\leq \mu}, k) \to \RG(\Fix_{\frac{T}{T}}(X_{\leq \mu}), \O)    
\end{equation}
is an equivalence.
\end{theorem}

\begin{proof}
Note that $\X^{+}_{\bullet} = \N_{\geq 0} [\omega_1, \omega_2]$ is the free commutative monoid on the set of fundamental coweights. Each affine Schubert variety is isomorphic to one of the form $X_{\leq k\omega_1}$. Indeed, there is an equivariant isomorphism $X_{\leq(a\omega_1 + b\omega_2)} \cong X_{\leq a\omega_1}$. We prove the result by induction on $a$.

The base cases of the induction are $\mu = 0 \cdot \omega_1$ and $\mu = 1 \cdot \omega_1$. In the first case, $X_{\leq 0} = \pt$ so the trace map is indeed an isomorphism supported in homological degree zero by Lemma \ref{lemma: trace for point}. In the second case, $X_{\leq \omega_1}$ is the perfection of $\P^1$ with the natural $T$-action, so the trace map is also an isomorphism supported in homological degree zero by Example \ref{lemma: perfect trace for projective space}.

Let $a \geq 1$ and assume that we know the result for all $j \omega_1$ with $j \leq a$; we prove it for $ \mu := (a+1) \omega_1$. 
Denote $\mu' := (a-1) \omega_1 + \omega_2$ the maximal element strictly below $\mu$ in the dominance order.

Put $\mu_{\bullet} := (\lambda, \omega_1)$ and consider \eqref{diagram: partial demazure resolution as strarified grassmannian bundle}. This gives the affine Demazure variety $Y_{\leq \mu_{\bullet}}$ with a map $g: Y_{\leq \mu_{\bullet}} \to X_{\leq \lambda}$. We also have the affine Demazure morphism $f: Y_{\leq \mu_{\bullet}} \to X_{\leq \mu}$ which leads to an abstract blowup square
\begin{equation}\label{abs1 in GL2}
    \begin{tikzcd}
        Y_{\leq \mu_{\bullet}} \arrow[d, "f"] & \arrow[l] E \arrow[d, "f'"] \\
        X_{\leq \mu} & \arrow[l] X_{\leq \mu'}
    \end{tikzcd}
\end{equation}
\begin{itemize}
\item 
By construction, $Y_{\leq \mu_{\bullet}}$ is a perfect $\P^1$-bundle over $X_{\leq \lambda}$ via $g$.
Hence the claim holds for $Y_{\leq \mu_{\bullet}}$ by induction and Lemma \ref{lemma: stability on perfect projective bundles}.

\item 
Since $X_{\leq \mu'} \cong X_{\leq (a-1)\omega_1}$, we know the claim by induction. 

\item 
The exceptional part $E$ is a perfect $\P^1$-bundle over $X_{\leq \mu'}$ under $f'$. Indeed, it is the projectivization of the restriction of the quasi-coherent sheaf $\eF_{\leq \mu}$ to the subvariety $X_{\leq \mu'}$; this restriction is a rank two vector bundle on $X_{\leq \mu'}$ by the combinatorics of coweights in $GL_2$. Since we already know the claim for $X_{\leq \mu'}$, we deduce it for $E$ by Lemma \ref{lemma: stability on perfect projective bundles}.
\end{itemize}

Now consider the following fiber sequences coming from \eqref{abs1 in GL2} via Theorem \ref{lemma: homotopy fiber square for K_T}.
\begin{equation*}
    \begin{tikzcd}
       K^T(X_{\leq \mu}, k) \arrow[d] \arrow[r] & K^T(Y_{\leq \mu_{\bullet}}, k) \oplus K^T(X_{\leq \mu'}, k) \arrow[d] \arrow[r] & K^T(E, k) \arrow[d]  \\
       HH^T(X_{\leq \mu}, k) \arrow[r] & HH^T(Y_{\leq \mu_{\bullet}}, k) \oplus HH^T(X_{\leq \mu'}, k) \arrow[r] & HH^T(E, k)
    \end{tikzcd}
\end{equation*}
Combining the previous three items and the associated long exact sequences, we deduce by 5-lemma that $\tr: K^T(X_{\leq \mu}, k) \to HH^T(X_{\leq \mu}, k)$ is an isomorphism.
\end{proof}

\begin{corollary}
Taking the zeroth homotopy group in \eqref{equation: trace iso for GL2} returns the natural isomorphism
\begin{equation}
\tr: K^T_0(X_{\leq \mu}, k) \cong H_0(\Fix_{\frac{T}{T}}(X_{\leq \mu}), \O) \cong H_0(\Fix_{T}(X_{\leq \mu}), \O).
\end{equation}
\end{corollary}

The above also shows that both sides of \eqref{equation: trace iso for GL2} are supported in homological degrees $\leq 0$. We will later see that they are actually supported in the homological degree zero only.
\begin{corollary}
Both $K^T(X_{\leq \mu}, k)$ and $\RG(\Fix_{\frac{T}{T}}(X_{\leq \mu}), \O)$ are supported in homological degrees $\leq 0$. In particular, $K^T_0(X_{\leq \mu})$ has no $p$-torsion.
\end{corollary}
\begin{proof}
Note that $(-)^T$ on functions is exact, so the right-hand side of \eqref{equation: trace iso for GL2} is supported in homological degrees $\leq 0$. Since $\tr$ is an isomorphism, the same is true for the left-hand side.   
The final line follows from $K^T_1(X, k) = K^T_0(X) {\otimes}_{\Z}^1 k$.
\end{proof}

\begin{theorem}
Integrally, it holds that
$$K^T(X_{\leq \mu}) \simeq KH^T(X_{\leq \mu}).$$
\end{theorem}
\begin{proof}
The proof of Theorem \ref{theorem: trace for GL_2 affine grassmannian} in particular shows how to describe $K^T(X_{\leq \mu})$ from $K^T(-)$ of smooth $T$-equivariant schemes via homotopy fiber sequences associated to perfect abstract blowups. The same sequences exist for $KH^T(-)$ by Proposition  \ref{lemma: homotopy fiber square for KH}, and there is a natural comparison map $K^T(-) \to KH^T(-)$. This is an isomorphism on smooth $T$-equivariant pfp schemes, so the result follows by induction.
\end{proof}

\subsection{Trace map for perfect \texorpdfstring{$\GL_n$}{GLn} affine Grassmannian}\label{section: trace for GL_n affine grassmannian}
In general, it is demanding to control the $K$-theory of the exceptional fibers by proper excision only. The extra ingredient is the following: we also get split homotopy fiber sequences associated to stratified Grassmannian bundles $Y \to X$ by the semi-orthogonal decomposition of \cite{Jia23}. With this extra knowledge, we prove the general case of Theorem \ref{introtheorem: trace map for affine grassmannian}.
Let $\Gr$ be either the perfect affine Grassmannian or the Witt-vector affine Grassmannian for $\GL_n$.
\begin{theorem}\label{theorem: trace for GL_n affine grassmannian}
Let $\mu \in \X^{+}_{\bullet}$ and $X_{\leq \mu}$ the corresponding perfect affine Schubert variety in $\Gr$. Then
\begin{equation}\label{equation: trace map for GLn affine schubert varieties}
    \tr: K^T(X_{\leq \mu}, k) \to \RG(\Fix_{\frac{T}{T}}(X_{\leq \mu}), \O)
\end{equation}
is an equivalence supported in homological degree zero.
\end{theorem}
\begin{proof}
We prove the theorem by induction on $\mu \in \X^{+}_{\bullet}$ with respect to $\lg(\mu)$. If $\mu = 0$, then $X_{\leq \mu} = \pt$ and the statement is clear by Lemma \ref{lemma: perfect equivariant point}.

Assume now $\mu \neq 0$ and decompose it as $\mu = b\omega_n + \mu_1 + \dots + \mu_{d}$ with $d \geq 1$ and each $\mu_i$ among $\omega_1^*, \dots, \omega^*_{n-1}$. Say $\mu_d = \omega^*_j$. Let $\lambda = \mu - \mu_d$ and put $\mu_{\bullet} = (\lambda, \mu_d)$.
Consider the following diagram building on \eqref{diagram: partial demazure resolution as strarified grassmannian bundle}:
\begin{equation}
    \begin{tikzcd}
       \widetilde{Y}_{\leq \mu_{\bullet}} \arrow[r, "h"] \arrow[dr] & Y_{\leq \mu_{\bullet}} \arrow[r, "g"] \arrow[d, "f"] & X_{\leq \lambda} \\
       & X_{\leq \mu}
    \end{tikzcd}
\end{equation}
Here $\widetilde{Y}_{\leq \mu_{\bullet}}$ parametrizes the data of $Y_{\leq \mu_{\bullet}}$ together with a full flag in the rank $j$ quotient corresponding to $\mu_d$. The map $h$ is thus a perfected full flag variety bundle (for the group $\GL_{j}$).

Now, the map $fh$ is an iterated perfected relative $\Grass(-, 1)$-bundle.
Indeed, having $\eF = \eH_0(\eE \to \eE)$ on $X_{\leq \mu}$, first take $\Grass_{X_{\leq \mu}}(\eF, 1)$. The tautological exact sequence on $\Grass_{X_{\leq \mu}}(\eF, 1)$ then reads as
$$0 \to \eF' \to f^*\eF \to \eF'' \to 0$$
and we can continue the process with $\eF'$. After $j$ steps we arrive at $\widetilde{Y}_{\leq \mu_{\bullet}}$. Note that although these $\Grass(-, 1)$-bundles may carry a nontrivial derived structure, their perfections are classical by \hyperlink{property: classical}{(vi)}.

We already know the statement of the theorem holds for $X_{\leq \lambda}$ by induction.
By stability on perfected partial flag variety bundles from Lemma \ref{lemma: perfect grassmannian bundle formula}, we deduce it for $Y_{\leq \mu_{\bullet}}$ and subsequently for $\widetilde{Y}_{\leq \mu_{\bullet}}$ as well.

Applying Lemma \ref{lemma: K-theory of perfect stratified grassmannian bundles} to $fh$, the trace map $K^T(X_{\leq \mu}, k) \to HH^T(X_{\leq \mu}, k)$ is a direct summand of the trace map $K^T(\widetilde{Y}_{\leq \mu_{\bullet}}, k) \to HH^T(\widetilde{Y}_{\leq \mu_{\bullet}}, k)$, which is an isomorphism supported in homological degree zero by the above. We deduce the same for the former, completing the inductive step.
\end{proof}

\begin{remark}
Note that at least in the equal characteristic case, $\eE_{\bullet}$ gives a preferred presentation of $\eF$ with Tor-amplitude $[1, 0]$, simplifying the application of Lemma \ref{lemma: K-theory of perfect stratified grassmannian bundles}. 
\end{remark}

\begin{theorem}\label{observation: K = KH for perfect affine grassmannian}
Integrally, the natural map is an equivalence
$$K^T(X_{\leq \mu}) \simeq KH^T(X_{\leq \mu}).$$
\end{theorem}
\begin{proof}
Both $K(-)$ and $KH(-)$ are localizing invariants and there is a natural map $K(-) \to KH(-)$. Note that $K^T(-)$ and $KH^T(-)$ agree on smooth $T$-equivariant schemes (in particular on a point). The inductive argument from the proof of Theorem \ref{theorem: trace for GL_n affine grassmannian} thus applies.
\end{proof}

%% file: 6_some_examples.tex
\section{Sample computations in affine Grassmannians}\label{section: some examples}

The proof of Theorem \ref{theorem: trace for GL_2 affine grassmannian} is constructive: it yields an inductive way to compute the ring $K^T_0(X_{\leq \mu})$ integrally in terms of generators and relations. We provide a few concrete instances of this computation in practice. 

\begin{example}\label{example: K-theory of adjoin GL_2}
Assume $k$ has $\chara k \geq 3$. Let $G_0 = GL_2$ and $T_0$ its diagonal torus. Let $\mu = 2 \omega_1$ and $X = \Gr_{\leq \mu}$. We get
\begin{equation*}
    K^T_0(X, k) \cong \left( k[t_1^{\pm 1}, t_2^{\pm 1}, e_1] / \left( (e_1 - 2t_1)(e_1 - 2t_2)(e_1 -(t_1 + t_2)) \right) \right)_{\perf}.
\end{equation*}
\end{example}
\begin{proof}
Let $\mu_{\bullet} = (\omega_1, \omega_1)$ and $Y= \Gr_{\leq \mu_{\bullet}}$. The associated abstract blowup square induces a long exact sequence in $K$-theory, which splits into short exact sequences by the description of the $K$-theory of projective bundles used for $Y \to E$. We thus get a Mayer--Vietoris square
\begin{center}
\begin{tikzcd}
    K^T_0(Y, k) \arrow[r] & K^T_0(E, k) \\
    K^T_0(X, k) \arrow[u] \arrow[r] & K^T_0(\pt, k) \arrow[u]
\end{tikzcd}
\end{center}
and we can compute $K^T_0(X, k)$ from the associated short exact sequence as a kernel.
In particular $K^T_0(X, k) = K^T_0(X) {\otimes_{\Z}}k$.
We can explicitly write:
\begin{align*}
    & K^T_0(\pt, k) \cong \left( k[t_1^{\pm 1}, t_2^{\pm 1}] \right)_{\perf} \\
    & K^T_0(E, k) \cong \left( k[t_1^{\pm 1}, t_2^{\pm 1}, \xi] / ((\xi -  t_1)(\xi -  t_2)) \right)_{\perf} \\
    & K^T_0(Y, k) \cong \left( k[t_1^{\pm 1}, t_2^{\pm 1}, \xi_0, \xi_1] / \left( \substack{ (\xi_0 -  t_1)(\xi_0 -  t_2) \\ (\xi_1 -  t_1)(\xi_1 -  t_2)} \right) \right)_{\perf}
\end{align*}
The map $K^T_0(Y, k) \to K^T_0(E, k)$ is given by $\xi_0 \mapsto \xi$, $\xi_1 \mapsto (-\xi + t_1 + t_2)$.
The above Mayer--Vietoris short exact sequence then shows $K^T_0(X, k) \hookrightarrow K^T_0(Y, k)$ is the subring generated by 
\begin{align*}
   & e_1 := \xi_0 + \xi_1 \\
   & e_2 := \xi_0 \xi_1
\end{align*}
The computation of this kernel can be done before perfection by exactness of filtered colimits.
Explicitly, $K^T_0(X, k)$ is given by the rank three flat $K^T_0(\pt, k)$-algebra
\begin{align*}
K^T_0(X, k) 
&\cong \left( k[t_1^{\pm 1}, t_2^{\pm 1}, e_1, e_2] \ / \left(\substack{ e_1^2 - (t_1+t_2)e_1 - 2e_2 +2t_1t_2 \\ 
                       e_2^2 - (t_1+t_2)^2e_2 + t_1t_2(t_1+t_2)e_1 - t_1^2t_2^2  \\
                       e_1 e_2 - 2(t_1+t_2)e_2 + t_1t_2e_1 }\right) \right)_{\perf} \\ 
&\cong \left(k[t_1^{\pm 1}, t_2^{\pm 1}, e_1] \ / \left(e_1^3 - 3(t_1 + t_2)e_1^2 + (4t_1t_2 + 2(t_1 + t_2)^2)e_1 - 4(t_1 + t_2)t_1t_2  \right) \right)_{\perf} \\
&\cong \left(k[t_1^{\pm 1}, t_2^{\pm 1}, e_1] \ / \left( (e_1 - 2t_1)(e_1 - 2t_2)(e_1 - (t_1 + t_2)) \right) \right)_{\perf}.
\end{align*}

The first line holds by checking these relations for $e_1$ and $e_2$; these are all of them by counting ranks in the Mayer--Vietoris sequence. To get the second line, we get rid of $e_2$ from the first relation (this is the only point where we use $\chara k \neq 2$). Substituting for $e_2$ into the third relation, we obtain the desired cubic relation for $e_1$. Substituting for $e_2$ into the second relation and using the above cubic relation for $e_1$, the second relation becomes tautological. The passage to the third line is clear.
\end{proof}

\begin{example}\label{example: K-theory of adjoint GL_3}
    Let $k$ be of $\chara k \neq 3$.  Let $G_0 = GL_3$ and $T_0$ its diagonal torus. Let $\mu = \omega_2 + \omega_1$ and $X \cong \Gr_{\leq \mu}$. We get
\begin{equation*}
    K^T_0(X, k) \cong \left( k[t_1^{\pm 1}, t_2^{\pm 1}, t_3^{\pm 1}, m_1, m_2] \ / 
    \left(
    \substack{
        3m_1^5 + (15c_2-5c_1^2)m_1^3 + (2c_1^3 - 9c_1c_2 + 9c_3)m_2 + (14c_2^2 - 4c_1^2c_2 - 6c_1c_3)m_1 \\
        3m^2_1m_2 - 2c_1m_1^3 + (3c_2 - c_1^2)m_2 + (c_1c_2 - 9c_3)m_1 \\
        3m_2^2 + m_1^4 - 4c_1m_1m_2 + 4c_2m_1^2
    }  
    \right) 
    \right)_{\perf}
\end{equation*}
where we denote elementary symmetric polynomials in $t_1, t_2, t_3$ by
\begin{equation*}
    c_1 := t_1 + t_2 + t_3, \qquad
    c_2 := t_1t_2 + t_2t_3 + t_1t_3, \qquad
    c_3 := t_1t_2t_3.
\end{equation*}
This is rank $7$ finite free algebra over $K^T_0(\pt, k)$.
\end{example}
\begin{proof}
Let $\mu_{\bullet} = (\omega_2, \omega_1)$ and $Y = \Gr_{\leq \mu_{\bullet}}$. The associated abstract blowup square has $Y$ a perfected $\P^2$-bundle over $\P^2$, $Z = \pt$ and $E$ perfect $\P^2$ as a section. One immediately gets
\begin{align*}
    K^T_0(Y, k) &\cong \left( k[t_1^{\pm 1}, t_2^{\pm 1}, t_3^{\pm 1}, \chi_0, \chi_1] \ / 
    \left(
    \substack{
    (\chi_0 - t_1)(\chi_0 - t_2)(\chi_0 - t_3) \\
    (\chi_1 - t_1)(\chi_1 - t_2)(\chi_1 - t_3)
    }  
    \right) 
    \right)_{\perf}\\
    K^T_0(Z, k) &\cong\left( k[t_1^{\pm 1}, t_2^{\pm 1}, t_3^{\pm 1}]\right)_{\perf} \\
    K^T_0(E, k) &\cong \left( k[t_1^{\pm 1}, t_2^{\pm 1}, t_3^{\pm 1}, \chi] \ / \left( (\chi - t_1)(\chi - t_2)(\chi - t_3) \right) \right)_{\perf}
\end{align*}
with the map $K^T_0(Y, k) \to K^T_0(E, k)$ given by $\chi_0 \mapsto \chi$, $\chi_1 \mapsto \chi$.
Since the induced long exact sequence in $K$-theory splits into short exact sequences (by the shape of $Y$ and $E$), we can explicitly compute $K^T_0(X, k)$ from the short exact sequence as a kernel.
Moreover, since the remaining terms are projective $K^T_0(\pt, k)$-modules of ranks $9, 1, 3$, it follows that $K^T_0(X, k)$ is a projective $K^T_0(\pt, k)$-module of rank $7$. The elements
\begin{align*}
    m_1 := \chi_0 - \chi_1 \\
    m_2 := \chi_0^2 - \chi_1^2
\end{align*}
are visibly in this kernel. A lengthy computation shows that they satisfy precisely the above ideal of relations inside $K^T_0(Y, k)$. For rank reasons, it follows that they generate $K^T_0(X, k)$ as an algebra.
\end{proof}

\begin{remark}
The above computations work the same for $KH^T(-)$ and $KH^G(-)$ in any characteristic (without perfection). In particular, this applies to characteristic zero.   
\end{remark}

\begin{remark}
The presentations from Examples \ref{example: K-theory of adjoin GL_2} and \ref{example: K-theory of adjoint GL_3} can be compared to their cohomological analogs studied in \cite{Hau24} by different methods: see \cite[(4.6)]{Hau24} and \cite[(4.10)]{Hau24}.
\end{remark}

We want to briefly mention how the trace map looks for affine Demazure resolutions, as the above rings for $X_{\leq \mu}$ embed into those for $Y_{\leq \mu_{\bullet}}$.
\begin{example}\label{example: trace for perfect affine demazure resolutions}
Let $\mu_{\bullet} = (\mu_1, \dots, \mu_d)$ be a sequence of fundamental coweights and $Y_{\leq \mu_{\bullet}}$ the corresponding affine Demazure resolution. Then 
\begin{equation*}\label{equation: trace iso for affine Demazure resolutions}
\tr: K^T(Y_{\leq \mu_{\bullet}}, k) \xrightarrow{\simeq} \RG(\Fix_{\frac{T}{T}}(Y_{\leq \mu_{\bullet}}), \O)    
\end{equation*}
is an equivalence supported in homological degree zero. Explicitly, 
\begin{align*}
K_0^T(Y_{\leq \mu_{\bullet}}) = & K^T_0(G/P_{\mu_1}) \underset{K^T_0(\pt)}{\otimes} K^T_0(G/P_{\mu_2}) \underset{K^T_0(\pt)}{\otimes} \dots \underset{K^T_0(\pt)}{\otimes} K^T_0(G/P_{\mu_k})
\end{align*}
\end{example}
\begin{proof}
This follows from Lemma \ref{lemma: perfect equivariant point} by iterating Lemma \ref{lemma: perfect grassmannian bundle formula}.

Alternatively, the trace map in degree zero for affine Demazure resolutions can be understood directly from the twisted product description by Lemma \ref{lemma: fixed point schemes of twisted products} and Lemma \ref{lemma: fiberwise criterion}; these arguments work $G$-equivariantly.
\end{proof}

%% file: 7_more_examples.tex
\section{More examples}\label{section: more examples}
We now record a few more examples to convey a feeling about perfect trace maps. These examples are interesting and relate to other results in the literature. However, since they are rather unrelated to our original motivation, we confine them to this section.

\subsection{Trivial group}
If $G = e$ is the trivial group over $k$ and $X$ is perfectly proper, the trace map in degree zero is the following isomorphism encoding simple topological information.
\begin{observation}\label{lemma: perfect trace map for trivial group}
Assume $X$ is perfectly proper over $k$. Then the trace map is an isomorphism
\begin{equation*}
        \tr_X : K_0(X)_k \xrightarrow{\cong} H_0(\Fix_e(X), \O).
\end{equation*}
Both sides agree with the free $k$-module $k[\pi_0(X)]$ on the set of connected components of $X$. For $i \geq 1$, both $K_i(X, k)$ and $H_i(\Fix_e(X), \O)$ are zero.
\end{observation}
\begin{proof}
Let $X_0$ be a proper model of $X$. 
The reduced global sections $H_0(X_0, \O^{\red})$ have to be constant on each connected component by properness. Since $X_0$ is quasi-compact, there are finitely many connected components, so $H_0(X_0, \O^{\red}) = k[\pi_0(X_0)]$.
Then 
\begin{equation*}
H_0(\Fix_e(X), \O) = H_0(X, \O) = \colim_{\varphi^*} H_0(X_0, \O) = \colim_{\varphi^*} H_0(X_0, \O^{\red}) = k[\pi_0(X_0)] = k[\pi_0(X)]    
\end{equation*}
because $k$ and hence $k[\pi_0(X_0)]$ is already perfect.

Concerning $K$-theory, it holds that $K_0(X) =  \widetilde{K}_0(X) \oplus H_0(X, \Z)$ with $\widetilde{K}_0(X)$ being a $\Z[\tfrac{1}{p}]$-module. Indeed, \cite[Theorem 3.1.2 and Remark 3.1.3]{Cou23} show that this is the case on affines, where it follows from the finite $\gamma$-filtration \cite[Theorem II.4.6]{Weib13}. Moreover, all $K_i(X)$ with $i > 0$ are also $\Z[\tfrac{1}{p}]$-modules. The general case of non-affine pfp $X$ now follows by the Zariski descent spectral sequence \cite[Proposition 8.3, (8.3.2)]{TT90}. See also \cite[Proposition II.8.8.4]{Weib13}.

Tensoring $K_0(X)$ with $k$, the first term $\widetilde{K}_0(X)$ vanishes, while the second becomes $k[\pi_0(X_0)]$. Since the trace map is induced by locally taking the rank, it induces the identity under the above identifications.

For the final statement, $(K_0(X) \otimes_{\Z}^{1} k) \cong  (\widetilde{K}_0(X)\otimes_{\Z}^{1} k) \oplus (H_0(X, \Z)\otimes_{\Z}^{1} k) \cong 0$ together with Corollary \ref{corollary: kratzers vanishing with k-coefficients} shows that $K_i(X, k)$ vanishes for all $i \geq 1$. The vanishing of $H_i(\Fix_e(X), \O)$ for $i \geq 1$ is clear.
\end{proof}

\begin{note}
In the non-equivariant case, the perfected trace map is an isomorphism in homological degrees $\geq 0$ by Observation \ref{lemma: perfect trace map for trivial group}. 
However, this may still fail in negative degrees -- the following example is based on \cite[Remark 11.7]{BS16}.
\end{note}

\begin{example}
Let $Y_0$ be a smooth ordinary elliptic curve over $k$. Take $Y$ to be its perfection. Then $K_{-1}(Y, k) = 0$ by smoothness, but $H_{-1}(Y, \O) \neq 0$ by ordinarity. In particular, the trace map is not an isomorphism in degree $-1$.    
\end{example}

\subsection{Nodal curve and negative \texorpdfstring{$K$}{K}-theory}

\begin{note}\label{note: non-connective G-theory}
Algebraic $K$-theory of perfect schemes can be non-connective: the negative $K$-groups can be non-zero, see Example \ref{example: nodal cubic curve} below. At the same time, we know by Observation \ref{lemma: K-theory and G-theory of perfect schemes} that $K(X) = G(X)$ for perfect schemes. In other words, $G$-theory of perfect schemes can be non-connective. This gives a negative answer to a question posed by \cite[Remark 3.2]{Kha20}.
\end{note}

We now record a concrete example of this non-connectivity; see also \cite{AMM22}.
\begin{example}[Nodal cubic curve]\label{example: nodal cubic curve}
Let $k = \F_p$. Consider the affine nodal cubic curve $C_0 =  \Spec \left( k[x, y] / \left( y^2 - x^2(x+1) \right) \right)$. This has a resolution of singularities by the affine line $\A^1_0$; the preimage of the singular point $(0, 0)$ in $C_0$ is a disjoint union of two points. Perfecting, we get a perfect abstract blowup square     
\begin{center}
    \begin{tikzcd}
        \A^1 \arrow[d] \arrow[r, hookleftarrow] &  \arrow[d] \pt \sqcup \pt \\
        C \arrow[r, hookleftarrow] & \pt
    \end{tikzcd}
\end{center}
The final piece of the associated long exact sequence in $K$-theory
\begin{equation*}
\to K_0(\A^1) \oplus K_0(\pt) \to K_0(\pt \sqcup \pt) \to K_{-1}(C) \to 0   
\quad \text{looks like} \quad
\to \Z \oplus \Z \xrightarrow{} \Z \oplus \Z  \to K_{-1}(C) \to 0  
\end{equation*}
with the map $\Z \oplus \Z \xrightarrow{} \Z \oplus \Z$ given by the matrix
\begin{equation*}
\begin{pmatrix}
    1 & 1 \\
    1 & 1
\end{pmatrix}.    
\end{equation*}
It follows that $K_{-1}(C) = \Z$. Since $K$-theory in lower degrees vanishes, $K_{-1}(C, k) = k$.
\end{example}

\subsection{Toric varieties}\label{section: toric varieties}
We illustrate how $K^T$ and the trace map behave on perfectly proper toric varieties without any smoothness assumptions: the trace map is an isomorphism; moreover $K^T$ and $KH^T$ agree. This easily follows from our setup and available literature \cite{VV03} and \cite{AHW09, CHWW09, CHWW14, CHWW15, CHWW18}. Also see \cite[Theorem 6.3.2]{Rych24} for related discussion of the comparison between $K^T_0(-, \C)$ and fixed-point schemes for complex smooth projective toric varieties.

In fact, the notion of {\it dilation} on the $K$-groups appearing in \cite{CHWW14} is given by the Frobenius and their main results -- such as \cite[Theorem 0.2]{CHWW14} over a field of positive characteristic -- can be reinterpreted and easily proved through perfect geometry. Moreover, we can use the perfect setup to exhibit examples of singular toric varieties with nonzero negative equivariant $K$-groups.

\begin{notation}\label{notation: toric varieties}
Let $T_0$ be a split torus over $k$. Let $X_0$ be a toric variety for $T_0$ corresponding to a fan $\Delta$. Given any cone $\sigma \in \Delta$, there is the corresponding subtorus $T_{\sigma, 0} \leq T_0$ and the associated $T_0$-orbit $T_0/T_{\sigma, 0}$ on $X_0$. Note that in our conventions, toric varieties are assumed to be {\it split} -- we do not consider any non-split forms over $k$.
\end{notation}

\begin{discussion}[Toric resolutions]\label{discussion: toric resolutions}
Let $X_0$ be a toric variety with respect to a split torus $T_0$. Then one combinatorially defines a resolution by a smooth toric variety with respect to the same torus $T_0$ as in \cite[\S 5]{Cox00}. Furthermore, such a resolution is obtained from a sequence of partial resolutions of the form \cite[(2.2)]{CHWW09} -- these fit into abstract blowup squares
\begin{equation*}
    \begin{tikzcd}
        Y_0 \arrow[d] & E_0 \arrow[d] \arrow[l] \\
        X_0 & \arrow[l] Z_0
    \end{tikzcd}
\end{equation*}
This square is a $T_0$-equivariant abstract blowup of toric varieties for quotients of $T_0$.
\end{discussion}

\begin{observation}\label{observation: K = KH for toric varieties}
Let $T$ be a perfect split torus and $X$ any perfect toric variety. Then integrally
\begin{equation*}
    K^T(X) \simeq KH^T(X).
\end{equation*}
\end{observation}
\begin{proof}
For a perfectly smooth toric variety $X$ we clearly have $K^T(X) \simeq KH^T(X)$.
In general, work by induction on the Krull dimension of $X$.
Since both $K^T(-)$ and $KH^T(-)$ satisfy proper excision by Theorem \ref{lemma: homotopy fiber square for K_T} and Lemma \ref{lemma: homotopy fiber square for KH}, looking at Discussion \ref{discussion: toric resolutions}, the statement for $X$ reduces to the statement for $Y$, $Z$, $E$. Since $Z$ and $E$ have smaller Krull dimension and the statement is stable under enlarging the torus, we have reduced the claim to the partial resolution $Y$. In finitely many steps, we reduce to the smooth situation.
\end{proof}

The irreducible components of fixed-point schemes for toric varieties have a combinatorial description. First, $\Fix_T(X)$ over $T$ has several horizontal components isomorphic to $T$ and labelled by the finite set of torus fixed-points $|X^T|$. Over more special subvarieties of $T$, we get further vertical components which can be read off from the fan combinatorics.
More precisely, the fixed point scheme -- together with its structure map -- is given by the commutative diagram
\begin{equation}\label{diagram: fixed-point scheme of toric variety}
\begin{tikzcd}
\bigcup_{\sigma \in \Delta} T_{\sigma} \times X^{T_{\sigma}} \arrow[d, "\pr_{T_{\sigma}}"] \arrow[r, "\cong"] & \Fix_T(X) \arrow[d]\\
\bigcup_{\sigma \in \Delta} T_{\sigma} \arrow[r] & T
\end{tikzcd}     
\end{equation}
compatibly with the residual $T$-action. By definition, the {\it horizontal} components are the irreducible components labeled by the $\sigma$ of maximal dimension. We call the other components {\it vertical}.

Assuming $X$ is perfectly proper, the vertical components are also perfectly proper and thus carry only constant global functions. Hence $H_0(\Fix_\frac{T}{T}, \O)$ injects into the direct sum $\bigoplus_{|X^T|} H_0(\frac{T}{T}, \O)$ of functions on the horizontal components. It is cut out by the equalizer conditions given by restriction to the vertical components.

\begin{theorem}\label{theorem: trace map for perfectly proper toric varieties}
Let $T$ be a perfect split torus and $X$ any perfectly proper toric variety. Then the trace map induces an equivalence
\begin{equation}\label{equation: trace map for toric varieties}
    K^T(X, k) \xrightarrow{\simeq} \RG(\Fix_{\frac{T}{T}}(X), \O).
\end{equation}
and both sides are supported in homological degrees $\leq 0$. 
\end{theorem}
\begin{proof}
    First assume $X$ is smooth. Note that $\RG(\Fix_{\frac{T}{T}}(X), \O)$ is supported in homological degrees $\leq 0$ by perfectness; we claim that it vanishes in degrees $\leq -1$ as well. Since $(-)^T$ is exact, it is enough to check this on $\Fix_T(X)$. By perfect base change \cite[Lemma 3.18]{BS16}, it is enough to check it on $\Fix_t(X)$ for each geometric point $t$ of $T$. 
    But now it follows from the fan combinatorics that $\Fix_t(X)$ is a disjoint union of (perfectly proper smooth) toric varieties\footnote{The following argument for this claim was explained to us by Kamil Rychlewicz.}. Indeed, $\Fix_t(X)$ is the subvariety given by the set of those cones in the fan of $X$ whose linear span contains $t$. We need to prove that these cones decompose into a disjoint union of stars. Since $X$ is smooth, the fan of $X$ is in particular simplicial. So if $t$ lies in some cones $\sigma_1$, $\sigma_2$, it also lies in the intersection of their linear spans, which is itself the linear span of a cone by simpliciality. We thus succeed by decomposing the fixed cones according to the non-emptyness of their intersections.
    Since toric varieties have no higher cohomology by \cite[end of proof of Theorem 5.1]{Cox00}, we deduce that $\RG(\Fix_{\frac{T}{T}}(X), \O)$ is supported in degree zero.

   Moreover, the desired isomorphism in degree zero follows from the description \cite[Theorem 6.2]{VV03} together with the properness of $X$. Indeed, the equalizer condition in \cite{VV03} matches the condition on global sections of the fixed-point scheme imposed by its vertical fibers via proper excision (using the description from \eqref{diagram: fixed-point scheme of toric variety}). Hence the trace map is an isomorphism in degree zero. Since $X$ is smooth, $K^T(X, k)$ is supported in homological degrees $\geq 0$. Moreover, by \cite[Theorem 6.9]{VV03} we in particular know that $K^T(X)$ has no $p$-torsion, so $K^T(X, k)$ is supported in degree zero by Lemma \ref{lemma: kratzers argument}. This finishes the proof for smooth $X$; we in fact showed that the isomorphism \eqref{equation: trace map for toric varieties} is supported in degree zero.

   Now, both $K^T(-, k)$ and $\RG(\Fix_{\frac{T}{T}}(X), \O) = HH^T(-, k)$ satisfy perfect proper excision by Theorem \ref{lemma: homotopy fiber square for K_T}. By Discussion \ref{discussion: toric resolutions} and induction on dimension, the trace map is an isomorphism. Both sides are then supported in homological degrees $\leq 0$ by \eqref{equation: perfect HH^T vanishes in positive degrees}.
\end{proof}

\begin{remark}\label{remark: toric varieties and quotient by T}
In the above theorem, it is not necessary to take the quotient by $T$ -- we have an equivalence
\begin{equation*}
\RG(\Fix_{\frac{T}{T}}(X), \O) \xrightarrow{\simeq} \RG(\Fix_{T}(X), \O).   
\end{equation*}
Indeed, the action of $T$ on $H_0(\Fix_T(X), \O)$ is trivial because of the embedding into a direct sum of copies of $H_0(T, \O)$. In particular, the claim holds for perfectly smooth toric varieties. Since both sides satisfy proper excision, we deduce the singular case as above by toric resolutions.
\end{remark}

In particular, we can use the above results to find examples of proper singular toric varieties $X_0$ with non-connective $K^{T_0}(X_0)$. Note that this statement does not contain any perfection; we did not find similar computations elsewhere. Let us record this in more detail.
\begin{note}\label{note: negative K-theory of perfect toric varieties}
When $X$ is smooth (or more generally when the fan of $X$ is simplicial), the proof of Theorem \ref{theorem: trace map for perfectly proper toric varieties} actually shows that both sides are supported in homological degree $0$. 

However, negative degrees may occur in general. For an example, consider the fan in $\Z^{\oplus 3}$ spanned by the eight vertices $(\pm 2, \pm 2, -1)$, $(\pm 1, \pm 1, 1)$ and $t \in T$ spanning the third coordinate axis. Then $\Fix_t(X) \cong C_4 \sqcup \pt \sqcup \pt$, where $C_4$ denotes a cyclic chain of four perfect $\P^1$'s. Resolving $C_4$ by the disjoint union of four perfect $\P^1$'s and computing $\RG(C_4, \O)$ from the associated abstract blowup square, we deduce that $H_{-1}(C_4, \O) \cong k \neq 0$. It follows from Remark \ref{remark: toric varieties and quotient by T} and Lemma \ref{lemma: perfect base change on fixed point schemes} that $K^T_{-1}(X, k) \cong \RG_{-1}(\Fix_{\frac{T}{T}}(X), \O) \cong \RG_{-1}(\Fix_T(X), \O) \neq 0$.

In conclusion, perfect singular toric varieties may have nontrivial negative equivariant algebraic $K$-theory after perfection. In particular, this happens classically (before perfection) by Lemma \ref{proposition: K-theory of inverse limit}.
\end{note}

%% file: bibliography.bib
@incollection{Zhu15,
  author = {Xinwen Zhu},
  booktitle = {Geometry of moduli spaces and representation theory}, 
  title = {An introduction to affine {G}rassmannians and the geometric {S}atake equivalence},
  year = {2017},
  publisher = {American Mathematical Society},
  doi = {10.1090/pcms/024},
  series = {IAS/Park City Mathematics Series},
  volume = {24},
}

@incollection{Yun15,
  author = {Zhiwei Yun},
  booktitle = {Geometry of moduli spaces and representation theory}, 
  title = {Lectures on {S}pringer theories and orbital integrals},
  year = {2017},
  publisher = {American Mathematical Society},
  doi = {10.1090/pcms/024},
  series = {IAS/Park City Mathematics Series},
  volume = {24},
}

@article{Zhu14,
author = {Zhu, Xinwen},
year = {2017},
pages = {403–492},
title = {Affine {G}rassmannians and the geometric {S}atake in mixed characteristic},
volume = {185},
journal = {Annals of Mathematics},
doi = {10.4007/annals.2017.185.2.2},
}

@article{BS16,
author = {Bhatt, Bhargav and Scholze, Peter},
year = {2017},
pages = {329–423},
title = {Projectivity of the {W}itt vector affine {G}rassmannian},
volume = {209},
journal = {Inventiones mathematicae},
doi = {10.1007/s00222-016-0710-4},
}

@misc{Sta,
    shorthand     = {Sta},
    label = {Sta},
    author    = "Stacks Project Authors",
    title     = "Stacks project",
    url       = {https://stacks.math.columbia.edu/}
}

@book{Weib13,
  title={The $K$-book: An Introduction to Algebraic $K$-theory},
  author={Weibel, Charles A.},
  series={Graduate Studies in Mathematics},
  volume ={145},
  year={2013},
  publisher={American Mathematical Society},
  doi={10.1090/gsm/145},
  isbn={9781470409432},
}

@inproceedings{TT90,
  title={Higher algebraic $K$-theory of schemes and of derived categories},
  author={Robert W. Thomason and Thomas Trobaugh},
  year={1990},
  booktitle={The Grothendieck Festschrift, Volume III},
  series = {Progress in Mathematics},
  volume = {88},
  pages = {247–435},
  publisher = {Birkhäuser},
  doi = {10.1007/978-0-8176-4576-2_10},
}

@inbook{Tho88,
title = {XX. Algebraic $K$-theory of group scheme actions},
booktitle = {Algebraic Topology and Algebraic $K$-Theory},
series = {Annals of Mathematics Studies},
volume ={113},
author = {Robert W. Thomason},
publisher = {Princeton University Press},
pages = {539--563},
year = {1988},
doi = {10.1515/9781400882113-021},
}

@article{Tho87,
title = {Equivariant resolution, linearization, and Hilbert's fourteenth problem over arbitrary base schemes},
author = {Robert W. Thomason},
journal = {Advances in Mathematics},
volume = {65},
number = {1},
pages = {16-34},
year = {1987},
doi = {10.1016/0001-8708(87)90016-8},
}

@article{Kha20,
  title={$K$-theory and $G$-theory of derived algebraic stacks},
  author={Adeel A. Khan},
  journal={Japanese Journal of Mathematics},
  year={2022},
  volume={17},
  pages={1 - 61},
  doi = {10.1007/s11537-021-2110-9},
}

@article{KR22,
title = {Generalized cohomology theories for algebraic stacks},
journal = {Advances in Mathematics},
volume = {458},
pages = {109975},
year = {2024},
issn = {0001-8708},
doi = {10.1016/j.aim.2024.109975},
author = {Adeel A. Khan and Charanya Ravi},
}

@incollection{Wei89,
	doi = {10.1090/conm/083/991991},
	year = 1989,
	publisher = {American Mathematical Society},
	pages = {461--488},
	author = {Charles A. Weibel},
	title = {Homotopy algebraic $K$-theory},
    booktitle = {Algebraic $K$-Theory and Algebraic Number Theory},
}

@article{Mor12,
  title={Pro cdh-descent for cyclic homology and {$K$}-theory},
  author={Matthew Morrow},
  journal={Journal of the Institute of Mathematics of Jussieu},
  year={2016},
  volume={15},
  pages={539 - 567},
  doi = {10.1017/S1474748014000413},
}

@article{KST16,
  title={Algebraic $K$-theory and descent for blow-ups},
  author={Moritz Kerz and Florian Strunk and Georg Tamme},
  journal={Inventiones mathematicae},
  year={2018},
  volume={211},
  pages={523-577},
  doi={10.1007/s00222-017-0752-2},
}

@article{Tam17,
  title={Excision in algebraic $K$-theory revisited},
  author={Georg Tamme},
  journal={Compositio Mathematica},
  year={2018},
  volume={154},
  pages={1801 - 1814},
  doi = {10.1112/S0010437X18007236},
}

@article{LT19,
  title={On the $K$-theory of pullbacks},
  author={Markus Land and Georg Tamme},
  journal={Annals of Mathematics},
  year={2019},
  volume={190},
  pages={877-930},
  doi = {10.4007/annals.2019.190.3.4},
}

@article{AMM22,
    author = {Benjamin Antieau and Akhil Mathew and Matthew Morrow},
    title = {$K$-theory of perfectoid rings},
    journal = {Documenta Mathematica},
    year = {2022},
    volume = {27},
    pages ={1923--1951},
    doi = {10.25537/dm.2022v27.1923-1951},
}

@article{Kha18,
    author = {Adeel A. Khan},
    title = {Algebraic $K$-theory of quasi-smooth blow-ups and cdh descent},
    journal = {Annales Henri Lebesgue},
    pages = {1091--1116},
    volume = {3},
    year = {2020},
    doi = {10.5802/ahl.55}
}

@unpublished{Cou23,
author = {Kevin Coulembier},
year = {2023},
title = {$K$-theory and perfection},
note = {preprint},
url = {https://arxiv.org/abs/2304.01421},
}

@article{CW22, 
title = {Perfecting group schemes},
doi={10.1017/S1474748024000033}, 
journal={Journal of the Institute of Mathematics of Jussieu}, 
author={Coulembier, Kevin and Williamson, Geordie}, 
year={2024}, 
pages={1–43}
}

@article{BST13,
  title={The weak ordinarity conjecture and {F}-singularities},
  author={Bhargav Bhatt and Karl Schwede and Shunsuke Takagi},
  journal ={Advanced Studies in Pure Mathematics},
  volume = {74},
  year = {2017},
  pages ={11-39},
  doi = {10.2969/aspm/07410011},
}

@article{KM21, 
title={$K$-theory of valuation rings}, 
volume={157},  
journal={Compositio Mathematica}, 
publisher={London Mathematical Society}, 
author={Shane Kelly and Matthew Morrow}, 
year={2021}, 
pages={1121–1142},
doi = {10.1112/S0010437X21007119},
}

@article{EK18,
  title={Perfection in motivic homotopy theory},
  author={Elden Elmanto and Adeel A. Khan},
  journal={Proceedings of the London Mathematical Society},
  year={2020},
  volume={120},
  pages={28-38},
  doi = {10.1112/plms.12280},
}

@incollection{Hau21,
  title = {Enhanced mirror symmetry for {L}anglands dual {H}itchin systems},
  author = {Tam{\'a}s Hausel},
  year={2023},
  booktitle={Proceedings of the International Congress of Mathematicians 2022},
  publisher = {EMS Press},
  doi = {10.4171/ICM2022/164},
  pages={2228–2249}
}

@article{HR23,
author = {Tam{\'a}s Hausel and Kamil Rychlewicz},
title = {Spectrum of equivariant cohomology as a fixed point scheme},
year     = {2025},
journal = {Épijournal de Géométrie Algébrique},
volume = {9},
doi ={10.46298/epiga.2025.12591},
}

@book{Hum95,
    author    = {James E. Humphreys},
    title     = {Conjugacy classes in semisimple algebraic groups},
    publisher = {Americal Mathematical Society},
    series={Mathematical Surveys and Monographs},
    volume = {43},
    year = {1995},
    doi = {10.1090/surv/043},
}

@book{Mil17,
  title={Algebraic Groups: The Theory of Group Schemes of Finite Type over a Field},
  author={James S. Milne},
  isbn={9781009018586},
  series={Cambridge Studies in Advanced Mathematics},
  year={2017},
  publisher={Cambridge University Press},
  doi = {10.1017/9781316711736},
}

@article{LV81,
author = {John V. Leahy and Marie A. Vitulli},
title = {{Seminormal rings and weakly normal varieties}},
volume = {82},
journal = {Nagoya Mathematical Journal},
pages = {27 -- 56},
year = {1981},
doi = {10.1017/S0027763000019279},
}

@article{GT80,
author = {Silvio Greco and Carlo Traverso},
title = {On seminormal schemes},
journal = {Compositio Mathematica},
volume = {40},
pages = {325-365},
year = {1980},
}

@article{Man80,
  title={Some properties of weakly normal varieties},
  author={Mirella Manaresi},
  journal={Nagoya Mathematical Journal},
  year={1980},
  volume={77},
  pages={61 - 74},
  doi = {10.1017/S0027763000018663},
}

@article{KP16,
author = {Kondyrev, Grigory and Prihodko, Artem},
year = {2020},
pages = {1739-1763},
title = {Categorical proof of holomorphic {A}tiyah-{B}ott formula},
volume = {19},
number={5},
journal = {Journal of the Institute of Mathematics of Jussieu},
doi = {10.1017/S1474748018000543},
}

@article{BFN10,
    author = {David Ben-Zvi and John Francis and David Nadler},
    title = {Integral transforms and Drinfeld centers in derived algebraic geometry},
    journal = {Journal of the American Mathematical Society},
    year = {2010},
    volume = {23},
    pages = {909-966},
    doi = {10.1090/S0894-0347-10-00669-7},
}

@article{BN13,
author = {David Ben-Zvi and David Nadler},
title = {Loop spaces and representations},
volume = {162},
journal = {Duke Mathematical Journal},
number = {9},
publisher = {Duke University Press},
pages = {1587 -- 1619},
year = {2013},
doi = {10.1215/00127094-2266130},
}

@article{Chen20,
title = {Equivariant localization and completion in cyclic homology and derived loop spaces},
journal = {Advances in Mathematics},
volume = {364},
pages = {107005},
year = {2020},
doi = {10.1016/j.aim.2020.107005},
author = {Harrison Chen},
}

@unpublished{Hoy18,
    author = {Marc Hoyois} ,
    title = {The homotopy fixed points of the circle action on Hochschild homology},
    note = {unpublished},
    year = {2018},
    doi = {10.48550/arXiv.1506.07123},
}

@article{HSS17,
title = {Higher traces, noncommutative motives, and the categorified Chern character},
author = {Marc Hoyois and Sarah Scherotzke and Nicolò Sibilla},
journal = {Advances in Mathematics},
volume = {309},
pages = {97-154},
year = {2017},
doi = {10.1016/j.aim.2017.01.008},
}

@article{BGT13,
author = {Andrew J Blumberg and David Gepner and Goncalo Tabuada},
title = {{A universal characterization of higher algebraic $K$–theory}},
volume = {17},
number = {2},
journal = {Geometry \& Topology},
publisher = {MSP},
pages = {733 -- 838},
year = {2013},
doi = {10.2140/gt.2013.17.733},
}

@article{Toe14,
    author = {Bertrand Toën},
    title = {Derived algebraic geometry},
    journal = {EMS Surveys in Mathematical Sciences},
    volume = {1},
    number = {2},
    year = {2014},
    pages = {153–240},
    doi = {10.4171/EMSS/4},
}

@article{McC94,
    author = {Randy McCarthy},
    title = {The cyclic homology of an exact category},
    journal = {Journal of Pure and Applied Algebra},
    volume = {93},
    number = {3},
    pages = {251-296},
    year = {1994},
    doi = {10.1016/0022-4049(94)90091-4},
}

@article{KM00,
    author = {Miriam Ruth Kantorovitz and Claudia Miller},
    title = {An explicit description of the Dennis trace map},
    journal = {Communications in Algebra},
    volume = {28},
    number = {3},
    pages = {1429-1447},
    year = {2000},
    doi = {10.1080/00927870008826904},
}

@book{DGM12,
    author = {Bjørn Ian Dundas and Thomas G. Goodwillie and Randy McCarthy},
    title = {The Local Structure of Algebraic $K$-Theory},
    publisher = {Springer},
    year = {2012},
    isbn = {978-1-4471-4392-5},
    doi  = {10.1007/978-1-4471-4393-2}
}

@incollection{Mad95,
    author = {Ib Madsen},
    title = {Algebraic $K$-theory and traces},
    booktitle = {Current Developments in Mathematics},
    pages = {191–321},
    year = {1995},
    publisher ={International Press, Cambridge},
    doi = {10.4310/CDM.1995.v1995.n1.a3},
}

@unpublished{AGR17,
    author = {David Ayala and Aaron Mazel-Gee and Nick Rozenblyum},
    title = {The geometry of the cyclotomic trace},
    note = {preprint},
    year ={2017}, 
    url ={https://arxiv.org/abs/1710.06409},
}

@article{VV03,
    author = {Gabriele Vezzosi and Angelo Vistoli},
    title = {Higher algebraic $K$-theory for actions of diagonalizable groups},
    journal = {Inventiones mathematicae},
    volume = {153},
    pages ={1–44},
    year = {2003},
    doi = {10.1007/s00222-002-0275-2},
}

@article{AHW09,
title = {The equivariant $K$-theory of toric varieties},
author = {Suanne Au and Mu-wan Huang and Mark E. Walker},
journal = {Journal of Pure and Applied Algebra},
volume = {213},
number = {5},
pages = {840-845},
year = {2009},
doi = {10.1016/j.jpaa.2008.10.010},
}

@Inbook{CHWW18,
author={Corti{\~{n}}as, G. and Haesemeyer, C. and Walker, M. E. and Weibel, C. A.},
title={The $K$-theory of toric schemes over regular rings of mixed characteristic},
booktitle={Singularities, Algebraic Geometry, Commutative Algebra, and Related Topics: Festschrift for Antonio Campillo on the Occasion of his 65th Birthday},
year={2018},
publisher={Springer},
pages={455--479},
isbn={978-3-319-96827-8},
doi={10.1007/978-3-319-96827-8_19},
}

@article{CHWW09,
author = {Guillermo Cortiñas and Christian Haesemeyer and Mark E. Walker and Charles Weibel},
 journal = {Transactions of the American Mathematical Society},
 pages = {3325--3341},
 title = {The $K$-theory of toric varieties},
 volume = {361}, 
 number = {6},
 year = {2009},
 doi = {10.1090/S0002-9947-08-04750-8},
}

@article{CHWW14,
author = {Guillermo Cortiñas and Christian Haesemeyer and Mark E. Walker and Charles Weibel},
title = {The $K$-theory of toric varieties in positive characteristic},
journal = {Journal of Topology},
volume = {7},
number = {1},
pages = {247-286},
doi = {10.1112/jtopol/jtt026},
year = {2014}
}

@article{CHWW15,
title = {Toric varieties, monoid schemes and cdh descent},
author = {Guillermo Cortiñas and Christian Haesemeyer and Mark E. Walker and Charles Weibel},
pages = {1--54},
volume = {2015},
journal = {Journal für die reine und angewandte Mathematik (Crelles Journal)},
doi = {10.1515/crelle-2012-0123},
year = {2015},
}

@Inbook{Cox00,
author= {David A. Cox},
title= {Toric varieties and toric resolutions},
bookTitle= {Resolution of Singularities: A research textbook in tribute to Oscar Zariski},
year={2000},
publisher={Birkh{\"a}user},
pages={259--284},
isbn={978-3-0348-8399-3},
doi={10.1007/978-3-0348-8399-3_9},
}

@unpublished{Jia23,
title = {Derived categories of derived grassmannians},
author = {Qingyuan Jiang},
note = {preprint},
url = {https://arxiv.org/abs/2307.02456},
year = {2023},
}

@article{Jia22a,
title = {Derived projectivizations of complexes},
author = {Qingyuan Jiang},
journal = {Memoirs of the American Mathematical Society},
doi = {10.1090/memo/1604},
year = {2025},
volume = {316},
number = {1604},
pages = {v+131 pp}
}

@unpublished{Jia22b,
title = {Derived Grassmannians and derived Schur functors},
author = {Qingyuan Jiang},
note = {preprint},
year = {2022},
url = {https://arxiv.org/abs/2212.10488},
}

@phdthesis{Rych24,
    author = {Kamil Rychlewicz},
    title = {Equivariant cohomology and rings of functions},
    school = {ISTA},
    year = {2024},
}

@unpublished{Wan24,
title = {Multiplicative Hitchin fibration and fundamental lemma},
author = {Griffin Wang},
url = {https://arxiv.org/abs/2402.19331},
note = {preprint},
year = {2024},
}

@article{Hau24,
author = {Tamás Hausel},
year = {2024},
month = {09},
title = {Commutative avatars of representations of semisimple Lie groups},
volume = {121},
number = {38},
pages = {e2319341121},
journal = {Proceedings of the National Academy of Sciences of the United States of America},
doi = {10.1073/pnas.2319341121}
}

@article{Kra80a,
    author = {Charles Kratzer},
    title = {Opérations d’Adams et représentations de groupes},
    journal = {L'Enseignement Mathématique},
    year = {1980},
    volume = {26},
    pages = {141-154},
    doi = {10.5169/seals-51063},
}

@article{Kra80b,
    author = {Charles Kratzer},
    title = {$\lambda$-structure en $K$-théorie algébrique},
    journal = {Commentarii Mathematici Helvetici},
    year = {1980},
    volume = {55},
    pages = {233-254},
    doi = {10.1007/BF02566684},
}

@article{Hil81,
author = {Howard L. Hiller},
title = {$\lambda$-rings and algebraic $K$-theory},
journal = {Journal of Pure and Applied Algebra},
volume = {20},
number = {3},
pages = {241-266},
year = {1981},
doi = {10.1016/0022-4049(81)90062-1},
}

@article{Koc98,
     author = {K\"ock, Bernhard},
     title = {The Grothendieck-Riemann-Roch theorem for group scheme actions},
     journal = {Annales scientifiques de l'\'Ecole Normale Sup\'erieure},
     pages = {415--458},
     volume = {31},
     number = {3},
     year = {1998},
     doi = {10.1016/s0012-9593(98)80140-7},
}

@article{HKT17,
     author = {Tom Harris and Bernhard K\"ock and Lenny Taelman},
     title = {Exterior power operations on higher $K$-groups via binary complexes},
     journal = {Annals of $K$-theory},
     pages = {409–450},
     volume = {2},
     number = {3},
     year = {2017},
     doi = {10.2140/akt.2017.2.409},
}

@article{KZ21,
title = {Comparison of exterior power operations on higher $K$-theory of schemes},
author = {Bernhard K\"ock and Ferdinando Zanchetta},
journal = {Mathematische Zeitschrift},
year = {2025},
doi = {10.1007/s00209-025-03681-2},
volume = {309},
number = {63},
}

@article{Ser68,
     author = {Jean-Pierre Serre},
     title = {Groupe de {Grothendieck} des sch\'emas en groupes r\'eductifs d\'eploy\'es},
     journal = {Publications Math\'ematiques de l'IH\'ES},
     pages = {37--52},
     volume = {34},
     year = {1968},
     doi = {10.1007/bf02684589},
}

@article{Hoy21,
     author = {Marc Hoyois},
     title = {Cdh descent in equivariant homotopy $K$-theory},
     journal = {Documenta Mathematica},
     pages = {457–482},
     volume = {25},
     year = {2020},
     doi = {10.4171/DM/754},
}

@article{Low25,
title = {Equivariant localizing invariants of simple varieties},
author = {Jakub L\"owit},
doi = {10.1093/imrn/rnag058},
journal = {International Mathematics Research Notices},
volume = {2026},
number = {7},
pages = {rnag058},
year = {2026},
}
